\documentclass[11pt]{amsart}

\usepackage{amsmath,amssymb,amscd,amsthm}

\usepackage[matrix,arrow]{xy}
\usepackage[dvips]{graphicx}
\usepackage{epsfig}
\usepackage{graphics}
\usepackage{color}
\graphicspath{{Fig/}}

\usepackage{url}

\newtheorem*{mainthm}{Main Theorem}

\newtheorem{dfn}{Definition}[section]
\newtheorem{thm}[dfn]{Theorem}

\newtheorem{lem}[dfn]{Lemma}
\newtheorem{cor}[dfn]{Corollary}

\theoremstyle{definition}
\newtheorem{rem}[dfn]{Remark}

\def\R{{\mathbb R}}
\def\Z{{\mathbb Z}}
\def\C{{\mathbb C}}

\def\Sph{{\mathbb S}}
\def\CP{{\mathbb C}{\mathrm P}}
\def\RP{{\mathbb R}{\mathrm P}}

\def\phi{\varphi}
\def\epsilon{\varepsilon}

\def\mod{\mathrm{mod}}

\newcommand{\sn}{\operatorname{sn}}
\newcommand{\cn}{\operatorname{cn}}
\newcommand{\dn}{\operatorname{dn}}
\newcommand{\id}{\operatorname{id}}
\renewcommand{\Re}{\operatorname{Re}}
\renewcommand{\Im}{\operatorname{Im}}
\newcommand{\const}{\mathrm{const}}
\renewcommand{\emptyset}{\varnothing}

\title[Flexible Kokotsakis polyhedra with quadrangular base]{Classification of
flexible Kokotsakis polyhedra with quadrangular base}
\author{Ivan Izmestiev}
\address{Institut f\"ur Mathematik\\
Freie Universit\"at Berlin\\
Arnimallee 2\\
D-14195 Berlin, Germany}
\email{izmestiev@math.fu-berlin.de}
\thanks{Supported by the European Research Council under the European Union's Seventh Framework Programme (FP7/2007-2013)/\allowbreak ERC Grant agreement no.~247029-SDModels}
\date{\today}

\begin{document}


\begin{abstract}
A Kokotsakis polyhedron with quadrangular base is a neighborhood of a quadrilateral in a quad surface.
Generically, a Kokotsakis polyhedron is rigid. Up to now, several flexible classes were known, but a complete classification was missing. In the present paper, we provide such a classification. The analysis is based on the fact that the dihedral angles of a Kokotsakis polyhedron are related by an Euler-Chasles correspondence. It results in a diagram of elliptic curves covering complex projective planes. A polyhedron is flexible if and only if this diagram commutes.
\end{abstract}

\maketitle

\section{Introduction}
\subsection{Kokotsakis polyhedra}
A \emph{Kokotsakis polyhedron} is a polyhedral surface in $\R^3$ which consists
of one $n$-gon (the \emph{base}), $n$ quadrilaterals attached to every side of
the $n$-gon, and $n$ triangles placed between each two consecutive
quadrilaterals. See Figure \ref{fig:KokoPol}, left, for the case $n=5$.

\begin{figure}[ht]
\centering
\begin{picture}(0,0)%
\includegraphics{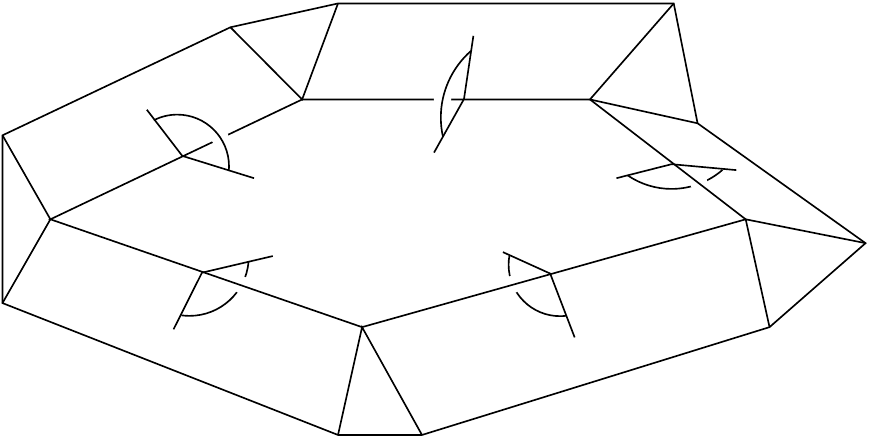}%
\end{picture}%
\setlength{\unitlength}{4144sp}%
\begingroup\makeatletter\ifx\SetFigFont\undefined%
\gdef\SetFigFont#1#2#3#4#5{%
  \reset@font\fontsize{#1}{#2pt}%
  \fontfamily{#3}\fontseries{#4}\fontshape{#5}%
  \selectfont}%
\fi\endgroup%
\begin{picture}(3970,1997)(4444,-2046)
\put(6723,-1582){\makebox(0,0)[lb]{\smash{{\SetFigFont{9}{10.8}{\rmdefault}{\mddefault}{\updefault}{\color[rgb]{0,0,0}$\phi_1$}%
}}}}
\put(5281,-525){\makebox(0,0)[lb]{\smash{{\SetFigFont{9}{10.8}{\rmdefault}{\mddefault}{\updefault}{\color[rgb]{0,0,0}$\phi_4$}%
}}}}
\put(6288,-409){\makebox(0,0)[lb]{\smash{{\SetFigFont{9}{10.8}{\rmdefault}{\mddefault}{\updefault}{\color[rgb]{0,0,0}$\phi_3$}%
}}}}
\put(5393,-1586){\makebox(0,0)[lb]{\smash{{\SetFigFont{9}{10.8}{\rmdefault}{\mddefault}{\updefault}{\color[rgb]{0,0,0}$\phi_5$}%
}}}}
\put(7348,-1015){\makebox(0,0)[lb]{\smash{{\SetFigFont{9}{10.8}{\rmdefault}{\mddefault}{\updefault}{\color[rgb]{0,0,0}$\phi_2$}%
}}}}
\end{picture}%
\caption{A Kokotsakis polyhedron with a pentagonal base.}
\label{fig:KokoPol}
\end{figure}
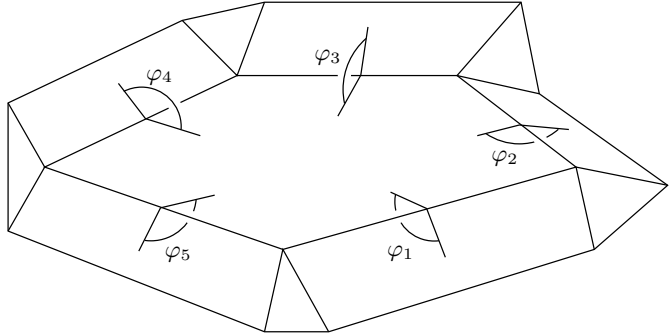

If all faces are viewed as rigid plates and edges represent hinges, then the polyhedron is in general rigid. Indeed, every interior vertex has only one degree of freedom, and as a result any two dihedral angles adjacent to a vertex are related by an equation:
\[
F_{12}(\phi_1, \phi_2) = 0, \quad F_{23}(\phi_2, \phi_3) = 0, \ldots\quad F_{n1}(\phi_n, \phi_1) = 0
\]
In a generic case, the solution set of this system is discrete, that is the polyhedron cannot be deformed. A natural question is to find all special shapes that allow isometric deformations.

Antonios Kokotsakis, a student of Carath\'eodory, studied these polyhedra in his PhD thesis in 1930's. As a result, he found a necessary and sufficient condition for the infinitesimal flexibility and described several classes of flexible polyhedra, \cite{Kok33}. At the same time, Sauer and Graf \cite{SG31} studied the case $n=4$ and also found several flexible special cases.

\subsection{Riemann surfaces and branched covers}
We give a complete classification of flexible Kokotsakis polyhedra with quadrangular base: $n=4$. Our approach is based on a diagram of branched covers between Riemann surfaces constructed as follows.

A substitution $z_i = \tan \frac{\phi_i}2$ transforms $F_{i,i+1}(\phi_i, \phi_{i+1}) = 0$ into a polynomial equation $P_{i,i+1}(z_i, z_{i+1}) = 0$ with $z_i \in \RP^1$. By allowing $z_i$ to take complex values we arrive at a Riemann surface
\begin{equation}
\label{eqn:ConfSpace}
C_{i,i+1} = \{(z_i, z_{i+1}) \mid P_{i,i+1}(z_i, z_{i+1}) = 0\}
\end{equation}
with two coordinate projections
\[
z_i \in \CP^1 \leftarrow C_{i,i+1} \rightarrow \CP^1 \ni z_{i+1}
\]
which form the following diagram
$$
\xymatrix@=1pc{
C_{34} \ar[r] \ar[d] & \CP^1 & C_{41} \ar[l] \ar[d]\\
\CP^1 && \CP^1\\
C_{23} \ar[r] \ar[u] & \CP^1 & C_{12} \ar[l] \ar[u]
}
$$
Now, if a polyhedron is flexible, then there are paths in $C_{i,i+1}$ whose projections to the common $\CP^1$'s coincide. These paths can be lifted to a bigger diagram involving fiber products of $C_{i-1,i}$ and $C_{i,i+1}$, see Section \ref{sec:DiagCovers} for more details. To find the conditions under which this big commutative diagram exists, we study the branch points of the involved branched covers. This approach is close in spirit to Ritt's solution of the problem of composite and commutative polynomials, \cite{Ritt22, Pak09}.

The equation of the configuration space \eqref{eqn:ConfSpace} has the form
\begin{equation}
\label{eqn:EulChasles}
a_{22}z^2w^2 + a_{20}z^2 + a_{02}w^2 + 2a_{11}zw + a_{00} = 0
\end{equation}
This is a special case of the so called Euler-Chasles correspondence, which describes in the generic situation an elliptic curve. The flexibility of a Kokotsakis polyhedron is equivalent to the composition of several correspondences of the form \eqref{eqn:EulChasles} being the identity (to be more exact, having an identity component). Compositions of Euler-Chasles correspondences were studied by Krichever \cite{Kri81} in the context of the quantum Yang-Baxter equation.

In the rigidity theory, elliptic curves came up in the works of Connelly and V.~Alexandrov \cite{Con74,AC11} who studied flexible bipyramids. Actually, a correspondence \eqref{eqn:EulChasles} appears each time when one has to do with deformations of a quadrilateral preserving its edge lengths. In our case this is a spherical quadrilateral associated with a vertex of degree four. The fact that the configuration space of a quadrilateral with fixed edge lengths is in general an elliptic curve was discovered by Darboux \cite{Dar79}. Darboux derived from it a porism on quadrilateral folding
, which fell into oblivion and was rediscovered several times since then.

\subsection{Related work}
Removing a neighborhood of a face of an octahedron yields a Kokotsakis polyhedron with a triangular base; conversely, by extending the faces of a $3$-Kokotsakis polyhedron one obtains an octahedron (or a projectively equivalent polyhedron). Both operations preserve the flexibility. Therefore flexible Kokotsakis polyhedra represent one possible generalization of Bricard's octahedra. This was noted already by Kokotsakis, and recently used by Nawratil \cite{Naw10}. As a generalization of Bricard's flexible polyhedra, Gaifullin \cite{Gai13} studies flexible cross-polytopes in space-forms of arbitrary dimension, using Euler-Chasles correspondences \eqref{eqn:EulChasles}.

A Kokotsakis polyhedron with a quadrangular base can be viewed as a part of a quadrangular mesh, see Figure \ref{fig:QuadMesh}. It is easy to see that a simply connected piece of a quadrangular mesh is flexible if and only if each of its Kokotsakis subpolyhedra is. Quad surfaces are a natural analog of parametrized smooth surfaces, and have been studied in this context since about a century, see the book \cite{Sau70}.

\begin{figure}[ht]
\centering
\includegraphics{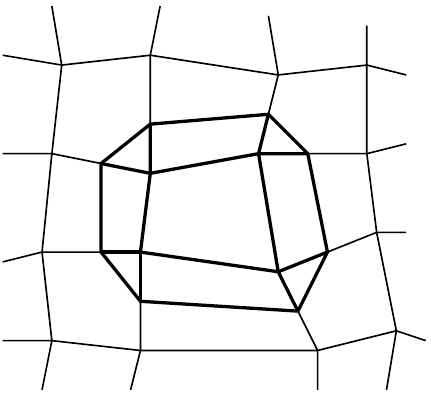}
\caption{Kokotsakis polyhedron with a quadrangular base as a part of a quad surface.}
\label{fig:QuadMesh}
\end{figure}

There was a recent surge of interest to quad surfaces. First, they give rise to various discrete integrable systems, \cite{SBH08, BS08}. Second, they have obvious applications in the architectural design \cite{PW08}: curved facades and roofs made of quadrangular glass panels can be found in abundance in the modern cities. Third, flexible Kokotsakis polyhedra (a special case called Miura-ori) found application in industry as solar panels for spacecraft \cite{Ni12}.

In Section \ref{sec:AlgReform} we survey recent works of Karpenkov \cite{Kar10}, Stachel \cite{Sta10} and Nawratil \cite{Naw11,Naw12} that inspired the present paper.

%

\subsection{Organization of the paper}
Section \ref{sec:Prelim} introduces the notation, provides more details on the previous work on the subject as well as on the approach used in the present work. It also explains the terminology used in Section \ref{sec:List} to list all flexible Kokotsakis polyhedra and describe their properties. In Section \ref{sec:ConfSpace} the configuration spaces of spherical four-bar linkages are classified, and their equations and parametrizations are written down. Section \ref{sec:Couplings} does a similar job for couplings of two four-bar linkages, with an emphasis on two classes: involutive and reducible couplings. Finally, in Section \ref{sec:Proof} the diagram of branched covers associated with a flexible Kokotsakis polyhedron is subject to a thorough analysis based on the results of two previous Sections, which results in the proof of Main Theorem.

\section{Notation and preliminaries}
\label{sec:Prelim}
Vertices of a Kokotsakis polyhedron and the values of its planar angles at the interior vertices are denoted as on Figure \ref{fig:NotPlanAngles}. Clearly, it is only a neighborhood of the base face $A_1A_2A_3A_4$ that matters: replacing, say, the vertex $B_1$ by any other point on the half-line $A_1B_1$ does not affect the flexibility or rigidity of the polyhedron. We will assume that all planar angles lie between $0$ and $\pi$: $0 < \alpha_i, \beta_i, \gamma_i, \delta_i < \pi$.

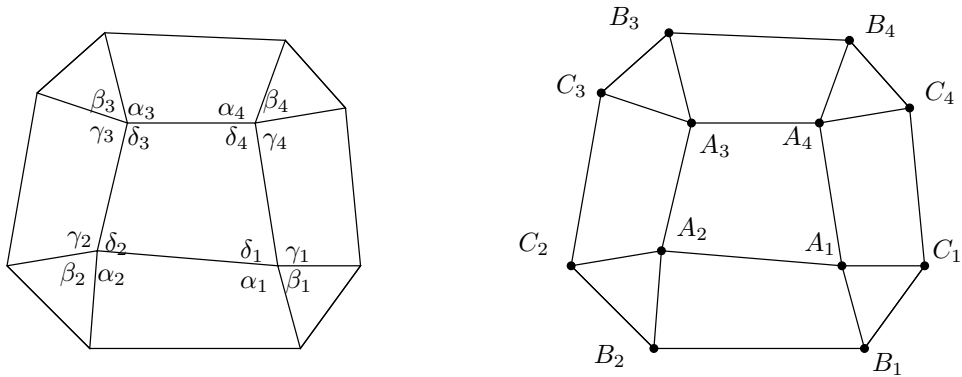
\begin{figure}[ht]
\centering
\begin{picture}(0,0)%
\includegraphics{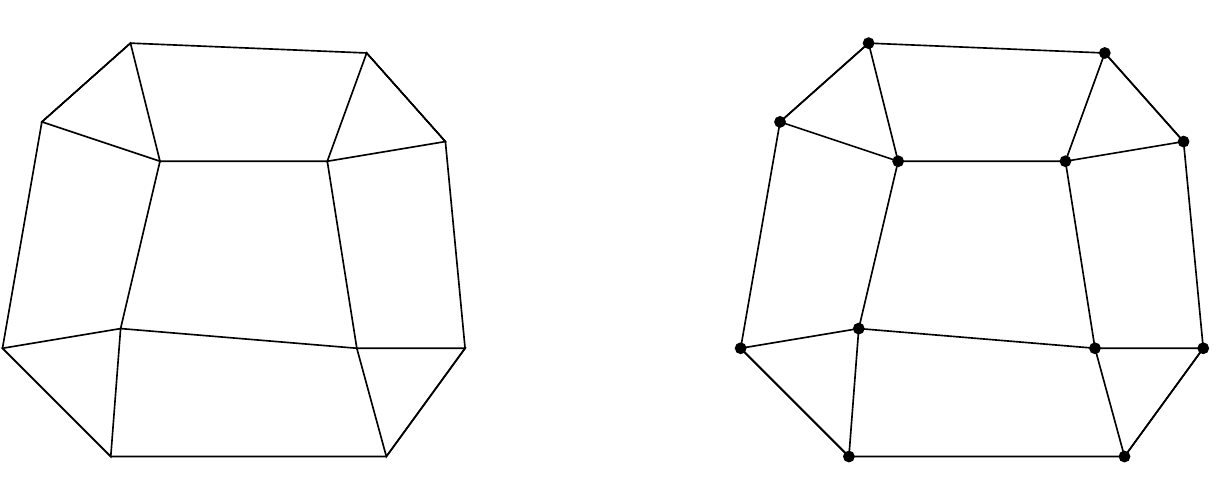}%
\end{picture}%
\setlength{\unitlength}{4144sp}%
\begingroup\makeatletter\ifx\SetFigFont\undefined%
\gdef\SetFigFont#1#2#3#4#5{%
  \reset@font\fontsize{#1}{#2pt}%
  \fontfamily{#3}\fontseries{#4}\fontshape{#5}%
  \selectfont}%
\fi\endgroup%
\begin{picture}(5562,2281)(-506,-1025)
\put(4276,-241){\makebox(0,0)[lb]{\smash{{\SetFigFont{10}{12.0}{\rmdefault}{\mddefault}{\updefault}{\color[rgb]{0,0,0}$A_1$}%
}}}}
\put(3511,-151){\makebox(0,0)[lb]{\smash{{\SetFigFont{10}{12.0}{\rmdefault}{\mddefault}{\updefault}{\color[rgb]{0,0,0}$A_2$}%
}}}}
\put(3646,344){\makebox(0,0)[lb]{\smash{{\SetFigFont{10}{12.0}{\rmdefault}{\mddefault}{\updefault}{\color[rgb]{0,0,0}$A_3$}%
}}}}
\put(4141,389){\makebox(0,0)[lb]{\smash{{\SetFigFont{10}{12.0}{\rmdefault}{\mddefault}{\updefault}{\color[rgb]{0,0,0}$A_4$}%
}}}}
\put(4681,-961){\makebox(0,0)[lb]{\smash{{\SetFigFont{10}{12.0}{\rmdefault}{\mddefault}{\updefault}{\color[rgb]{0,0,0}$B_1$}%
}}}}
\put(5041,-286){\makebox(0,0)[lb]{\smash{{\SetFigFont{10}{12.0}{\rmdefault}{\mddefault}{\updefault}{\color[rgb]{0,0,0}$C_1$}%
}}}}
\put(3016,-916){\makebox(0,0)[lb]{\smash{{\SetFigFont{10}{12.0}{\rmdefault}{\mddefault}{\updefault}{\color[rgb]{0,0,0}$B_2$}%
}}}}
\put(2566,-241){\makebox(0,0)[lb]{\smash{{\SetFigFont{10}{12.0}{\rmdefault}{\mddefault}{\updefault}{\color[rgb]{0,0,0}$C_2$}%
}}}}
\put(2791,704){\makebox(0,0)[lb]{\smash{{\SetFigFont{10}{12.0}{\rmdefault}{\mddefault}{\updefault}{\color[rgb]{0,0,0}$C_3$}%
}}}}
\put(3106,1109){\makebox(0,0)[lb]{\smash{{\SetFigFont{10}{12.0}{\rmdefault}{\mddefault}{\updefault}{\color[rgb]{0,0,0}$B_3$}%
}}}}
\put(4636,1064){\makebox(0,0)[lb]{\smash{{\SetFigFont{10}{12.0}{\rmdefault}{\mddefault}{\updefault}{\color[rgb]{0,0,0}$B_4$}%
}}}}
\put(4996,659){\makebox(0,0)[lb]{\smash{{\SetFigFont{10}{12.0}{\rmdefault}{\mddefault}{\updefault}{\color[rgb]{0,0,0}$C_4$}%
}}}}
\put(1171,-286){\makebox(0,0)[lb]{\smash{{\SetFigFont{10}{12.0}{\rmdefault}{\mddefault}{\updefault}{\color[rgb]{0,0,0}$\gamma_1$}%
}}}}
\put(1171,-466){\makebox(0,0)[lb]{\smash{{\SetFigFont{10}{12.0}{\rmdefault}{\mddefault}{\updefault}{\color[rgb]{0,0,0}$\beta_1$}%
}}}}
\put(901,-286){\makebox(0,0)[lb]{\smash{{\SetFigFont{10}{12.0}{\rmdefault}{\mddefault}{\updefault}{\color[rgb]{0,0,0}$\delta_1$}%
}}}}
\put(901,-466){\makebox(0,0)[lb]{\smash{{\SetFigFont{10}{12.0}{\rmdefault}{\mddefault}{\updefault}{\color[rgb]{0,0,0}$\alpha_1$}%
}}}}
\put( 46,-421){\makebox(0,0)[lb]{\smash{{\SetFigFont{10}{12.0}{\rmdefault}{\mddefault}{\updefault}{\color[rgb]{0,0,0}$\alpha_2$}%
}}}}
\put( 91,-241){\makebox(0,0)[lb]{\smash{{\SetFigFont{10}{12.0}{\rmdefault}{\mddefault}{\updefault}{\color[rgb]{0,0,0}$\delta_2$}%
}}}}
\put(-179,-421){\makebox(0,0)[lb]{\smash{{\SetFigFont{10}{12.0}{\rmdefault}{\mddefault}{\updefault}{\color[rgb]{0,0,0}$\beta_2$}%
}}}}
\put(226,389){\makebox(0,0)[lb]{\smash{{\SetFigFont{10}{12.0}{\rmdefault}{\mddefault}{\updefault}{\color[rgb]{0,0,0}$\delta_3$}%
}}}}
\put(  1,434){\makebox(0,0)[lb]{\smash{{\SetFigFont{10}{12.0}{\rmdefault}{\mddefault}{\updefault}{\color[rgb]{0,0,0}$\gamma_3$}%
}}}}
\put(  1,614){\makebox(0,0)[lb]{\smash{{\SetFigFont{10}{12.0}{\rmdefault}{\mddefault}{\updefault}{\color[rgb]{0,0,0}$\beta_3$}%
}}}}
\put(1036,614){\makebox(0,0)[lb]{\smash{{\SetFigFont{10}{12.0}{\rmdefault}{\mddefault}{\updefault}{\color[rgb]{0,0,0}$\beta_4$}%
}}}}
\put(1036,389){\makebox(0,0)[lb]{\smash{{\SetFigFont{10}{12.0}{\rmdefault}{\mddefault}{\updefault}{\color[rgb]{0,0,0}$\gamma_4$}%
}}}}
\put(811,389){\makebox(0,0)[lb]{\smash{{\SetFigFont{10}{12.0}{\rmdefault}{\mddefault}{\updefault}{\color[rgb]{0,0,0}$\delta_4$}%
}}}}
\put(766,569){\makebox(0,0)[lb]{\smash{{\SetFigFont{10}{12.0}{\rmdefault}{\mddefault}{\updefault}{\color[rgb]{0,0,0}$\alpha_4$}%
}}}}
\put(-134,-196){\makebox(0,0)[lb]{\smash{{\SetFigFont{10}{12.0}{\rmdefault}{\mddefault}{\updefault}{\color[rgb]{0,0,0}$\gamma_2$}%
}}}}
\put(226,569){\makebox(0,0)[lb]{\smash{{\SetFigFont{10}{12.0}{\rmdefault}{\mddefault}{\updefault}{\color[rgb]{0,0,0}$\alpha_3$}%
}}}}
\end{picture}%
\caption{Notation in a Kokotsakis polyhedron.}
\label{fig:NotPlanAngles}
\end{figure}

\subsection{Spherical linkage associated with a Kokotsakis polyhedron}
\label{sec:SpherLink}

For each of the four interior vertices $A_1$, $A_2$, $A_3$, $A_4$ consider its spherical image, that is the intersection of the cone of adjacent faces with
a unit sphere centered at the vertex. This yields four spherical quadrilaterals $Q_i$ with side
lengths $\alpha_i, \beta_i, \gamma_i, \delta_i$ in this cyclic order. Spherical images of two adjacent vertices are \emph{coupled} by means of a common dihedral angle. Equivalently, for every edge $A_iA_{i+1}$ we have a scissors-like coupling of two spherical quadrilaterals, see Figure \ref{fig:Scissors}.

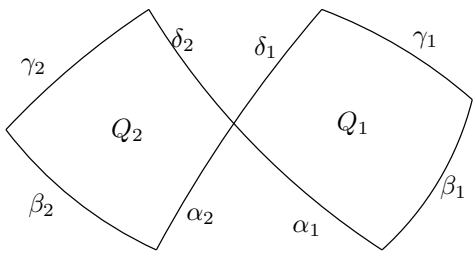
\begin{figure}[ht]
\centering
\begin{picture}(0,0)%
\includegraphics{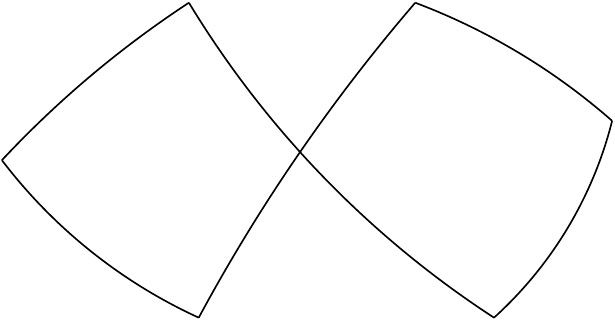}%
\end{picture}%
\setlength{\unitlength}{4144sp}%
\begingroup\makeatletter\ifx\SetFigFont\undefined%
\gdef\SetFigFont#1#2#3#4#5{%
  \reset@font\fontsize{#1}{#2pt}%
  \fontfamily{#3}\fontseries{#4}\fontshape{#5}%
  \selectfont}%
\fi\endgroup%
\begin{picture}(2806,1456)(-727,-1239)
\put(1879,-897){\makebox(0,0)[lb]{\smash{{\SetFigFont{10}{12.0}{\rmdefault}{\mddefault}{\updefault}{\color[rgb]{0,0,0}$\beta_1$}%
}}}}
\put(1716,  5){\makebox(0,0)[lb]{\smash{{\SetFigFont{10}{12.0}{\rmdefault}{\mddefault}{\updefault}{\color[rgb]{0,0,0}$\gamma_1$}%
}}}}
\put(997,-1125){\makebox(0,0)[lb]{\smash{{\SetFigFont{10}{12.0}{\rmdefault}{\mddefault}{\updefault}{\color[rgb]{0,0,0}$\alpha_1$}%
}}}}
\put(361,-1051){\makebox(0,0)[lb]{\smash{{\SetFigFont{10}{12.0}{\rmdefault}{\mddefault}{\updefault}{\color[rgb]{0,0,0}$\alpha_2$}%
}}}}
\put(-629,-151){\makebox(0,0)[lb]{\smash{{\SetFigFont{10}{12.0}{\rmdefault}{\mddefault}{\updefault}{\color[rgb]{0,0,0}$\gamma_2$}%
}}}}
\put(271,-16){\makebox(0,0)[lb]{\smash{{\SetFigFont{10}{12.0}{\rmdefault}{\mddefault}{\updefault}{\color[rgb]{0,0,0}$\delta_2$}%
}}}}
\put(766,-61){\makebox(0,0)[lb]{\smash{{\SetFigFont{10}{12.0}{\rmdefault}{\mddefault}{\updefault}{\color[rgb]{0,0,0}$\delta_1$}%
}}}}
\put(-584,-1006){\makebox(0,0)[lb]{\smash{{\SetFigFont{10}{12.0}{\rmdefault}{\mddefault}{\updefault}{\color[rgb]{0,0,0}$\beta_2$}%
}}}}
\put(1261,-511){\makebox(0,0)[lb]{\smash{{\SetFigFont{10}{12.0}{\rmdefault}{\mddefault}{\updefault}{\color[rgb]{0,0,0}$Q_1$}%
}}}}
\put(-89,-556){\makebox(0,0)[lb]{\smash{{\SetFigFont{10}{12.0}{\rmdefault}{\mddefault}{\updefault}{\color[rgb]{0,0,0}$Q_2$}%
}}}}
\end{picture}%
\caption{Two coupled spherical quadrilaterals associated with the edge $A_1A_2$.}
\label{fig:Scissors}
\end{figure}

On Figure \ref{fig:Couple4} two different ways to couple all four quadrilaterals $Q_i$ are shown. The left one is a closed chain of quadrilaterals on a sphere (due to $\delta_1 + \delta_2 + \delta_3 + \delta_4 = 2\pi$), so that it incorporates all four couplings. The right one fits better with Figure \ref{fig:NotPlanAngles}, but is in general not closed, so that the coupling between $Q_1$ and $Q_4$ must be encoded by requiring the equality of marked angles.

\begin{figure}[ht]
\centering
\begin{picture}(0,0)%
\includegraphics{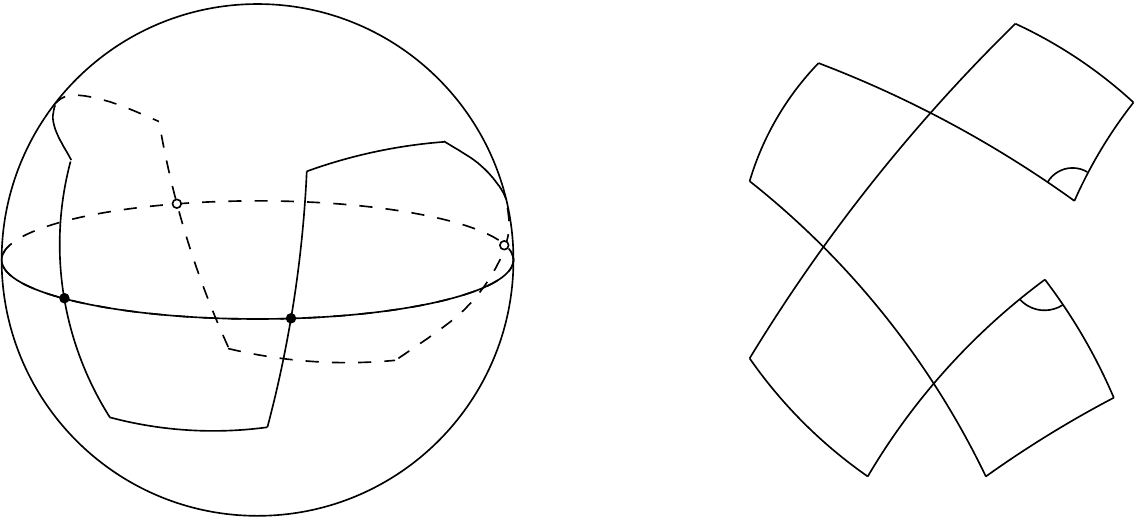}%
\end{picture}%
\setlength{\unitlength}{4144sp}%
\begingroup\makeatletter\ifx\SetFigFont\undefined%
\gdef\SetFigFont#1#2#3#4#5{%
  \reset@font\fontsize{#1}{#2pt}%
  \fontfamily{#3}\fontseries{#4}\fontshape{#5}%
  \selectfont}%
\fi\endgroup%
\begin{picture}(5191,2354)(-3427,-1328)
\put(916,-466){\makebox(0,0)[lb]{\smash{{\SetFigFont{10}{12.0}{\rmdefault}{\mddefault}{\updefault}{\color[rgb]{0,0,0}$\delta_1$}%
}}}}
\put(581,114){\makebox(0,0)[lb]{\smash{{\SetFigFont{10}{12.0}{\rmdefault}{\mddefault}{\updefault}{\color[rgb]{0,0,0}$\delta_3$}%
}}}}
\put(1063,209){\makebox(0,0)[lb]{\smash{{\SetFigFont{10}{12.0}{\rmdefault}{\mddefault}{\updefault}{\color[rgb]{0,0,0}$\delta_4$}%
}}}}
\put(586,-331){\makebox(0,0)[lb]{\smash{{\SetFigFont{10}{12.0}{\rmdefault}{\mddefault}{\updefault}{\color[rgb]{0,0,0}$\delta_2$}%
}}}}
\put(-1782,-342){\makebox(0,0)[lb]{\smash{{\SetFigFont{10}{12.0}{\rmdefault}{\mddefault}{\updefault}{\color[rgb]{0,0,0}$\delta_1$}%
}}}}
\put(-2826,-351){\makebox(0,0)[lb]{\smash{{\SetFigFont{10}{12.0}{\rmdefault}{\mddefault}{\updefault}{\color[rgb]{0,0,0}$\delta_2$}%
}}}}
\put(-2979,102){\makebox(0,0)[lb]{\smash{{\SetFigFont{10}{12.0}{\rmdefault}{\mddefault}{\updefault}{\color[rgb]{0,0,0}$\delta_3$}%
}}}}
\put(-1682,106){\makebox(0,0)[lb]{\smash{{\SetFigFont{10}{12.0}{\rmdefault}{\mddefault}{\updefault}{\color[rgb]{0,0,0}$\delta_4$}%
}}}}
\end{picture}%
\caption{Spherical linkages with scissors-like joints associated with a Kokotsakis polyhedron.}
\label{fig:Couple4}
\end{figure}
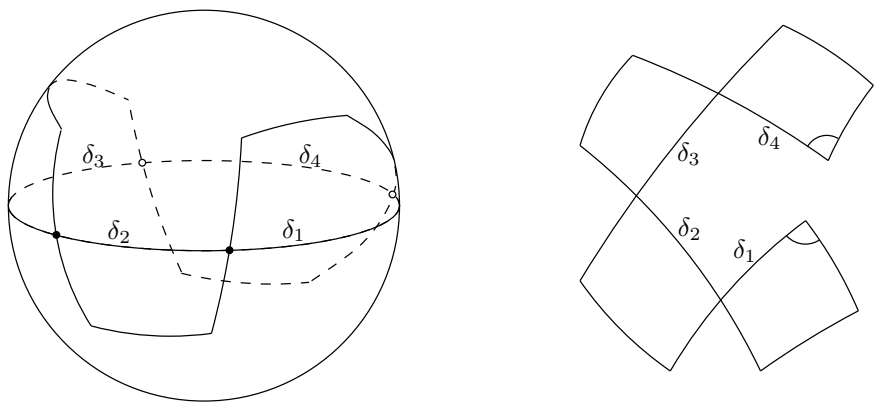

\begin{lem}
Every Kokotsakis polyhedron gives rise to spherical linkages as shown on Figure \ref{fig:Couple4}. Vice versa, every spherical linkage of this form corresponds to some Kokotsakis polyhedron, if we require $\delta_1 + \delta_2 + \delta_3 + \delta_4 = 2\pi$ for the linkage on the right.

A Kokotsakis polyhedron is flexible if and only if either of the corresponding spherical linkages on Figure \ref{fig:Couple4} is flexible. For the linkage on the right, the two marked angles are required to stay equal during the deformation.
\end{lem}
\begin{proof}
How to associate with a Kokotsakis polyhedron a spherical linkage was explained above.
To construct a Kokotsakis polyhedron for a given chain of quadrilaterals, start by choosing as the base any euclidean quadrilateral with the angles $\delta_1, \delta_2, \delta_3, \delta_4$. At the vertices of the base, construct tetrahedral angles that have spherical quadrilaterals $Q_i$ as their spherical images. The fact that the corresponding angles of $Q_i$ are equal in pairs means that these tetrahedral angles fit to form a Kokotsakis polyhedron.

During an isometric deformation of a Kokotsakis polyhedron the planar angles of its faces remain
constant. Thus isometric deformations of the polyhedron correspond to motions of the linkages on
Figure \ref{fig:Couple4} (with the equal angles condition for the right one).
\end{proof}

Thus, in order to classify all flexible Kokotsakis polyhedra, it suffices to classify all flexible spherical linkages as on Figure \ref{fig:Couple4} with $\delta_1 + \delta_2 + \delta_3 + \delta_4 = 2\pi$.


%

\subsection{Algebraic reformulation of the problem}
\label{sec:AlgReform}
For fixed lengths of the bars, the shape of the spherical linkage introduced in Section \ref{sec:SpherLink} is uniquely determined by the four angles $\phi$, $\psi_1$, $\psi_2$, and $\theta$ on Figure \ref{fig:SpherLinkAngles}. If we introduce the variables
\begin{equation}
\label{eqn:Variables}
z = \tan\frac{\phi}2, \quad w_1 = \tan\frac{\psi_1}2, \quad w_2 = \tan\frac{\psi_2}2, \quad u = \tan\frac{\theta}2,
\end{equation}
then the relation between each pair of adjacent angles can be expressed as a polynomial equation in the corresponding variables:
\begin{equation}
\label{eqn:PolSystem}
\begin{aligned}
P_1(z, w_1) &= 0 \quad & P_2(z, w_2) &= 0 \\
P_3(u, w_1) &= 0 \quad & P_4(u, w_2) &= 0
\end{aligned}
\end{equation}

\begin{figure}[ht]
\centering
\begin{picture}(0,0)%
\includegraphics{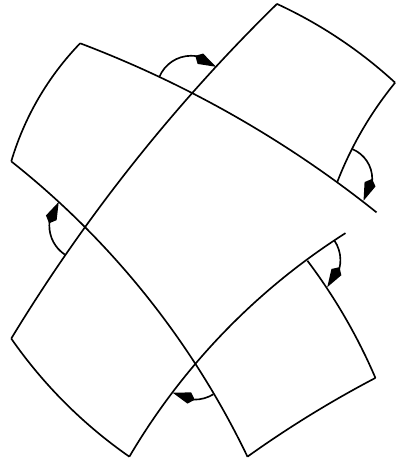}%
\end{picture}%
\setlength{\unitlength}{4144sp}%
\begingroup\makeatletter\ifx\SetFigFont\undefined%
\gdef\SetFigFont#1#2#3#4#5{%
  \reset@font\fontsize{#1}{#2pt}%
  \fontfamily{#3}\fontseries{#4}\fontshape{#5}%
  \selectfont}%
\fi\endgroup%
\begin{picture}(1814,2086)(2740,-1149)
\put(4304,-376){\makebox(0,0)[lb]{\smash{{\SetFigFont{10}{12.0}{\rmdefault}{\mddefault}{\updefault}{\color[rgb]{0,0,0}$\psi_1$}%
}}}}
\put(2755,-145){\makebox(0,0)[lb]{\smash{{\SetFigFont{10}{12.0}{\rmdefault}{\mddefault}{\updefault}{\color[rgb]{0,0,0}$\psi_2$}%
}}}}
\put(3514,-1011){\makebox(0,0)[lb]{\smash{{\SetFigFont{10}{12.0}{\rmdefault}{\mddefault}{\updefault}{\color[rgb]{0,0,0}$\phi$}%
}}}}
\put(3466,756){\makebox(0,0)[lb]{\smash{{\SetFigFont{10}{12.0}{\rmdefault}{\mddefault}{\updefault}{\color[rgb]{0,0,0}$\theta$}%
}}}}
\put(4464,134){\makebox(0,0)[lb]{\smash{{\SetFigFont{10}{12.0}{\rmdefault}{\mddefault}{\updefault}{\color[rgb]{0,0,0}$\psi_1$}%
}}}}
\end{picture}%
\caption{Angles determining the shape of a spherical linkage.}
\label{fig:SpherLinkAngles}
\end{figure}
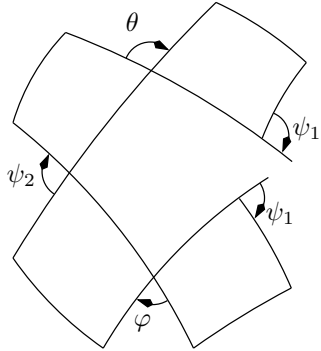

Thus we have the following lemma.
\begin{lem}
A Kokotsakis polyhedron is flexible if and only if the system of polynomial equations \eqref{eqn:PolSystem} has a 1-parameter family of solutions over the reals.
\end{lem}

Karpenkov \cite{Kar10} wrote down a similar system of equations, only with the cosines of angles in place of the tangents of half-angles as variables. He suggested to study the system \eqref{eqn:PolSystem} by the elimination method: let $R_{12}(w_1, w_2)$ be the resultant of $P_1$ and $P_2$ as polynomials in $z$, and let $R_{34}(w_1,w_2)$ be the resultant of $P_3$ and $P_4$ as polynomials in $u$. Then the system has a one-parameter set of solutions if and only if the algebraic sets $R_{12} = 0$ and $R_{34} = 0$ have a common irreducible component. This, in turn, is equivalent to the vanishing of the resultant of $R_{12}$ and $R_{34}$.

Stachel \cite{Sta10} used the substitution \eqref{eqn:Variables} that leads to simpler equations than in the case of cosines. All possible factorizations of the resultant $R_{12}$ were studied in two works \cite{Naw11, Naw12} of Nawratil. The case of irreducible $R_{12}$ and $R_{34}$ remained open.

\subsection{Branched covers between configuration spaces}
Our approach to the system \eqref{eqn:PolSystem} is as follows. Put
\begin{gather*}
Z_i := \{(z,w_i) \in (\CP^1)^2 \mid P_i(z,w_i) = 0\},\quad i = 1,2\\
Z_{12} := \{(z,w_1,w_2) \in (\CP^1)^3 \mid P_1(z,w_1) = 0 = P_2(z,w_2)\}
\end{gather*}
Then $Z_i$ is the complexified configuration space of the quadrilateral $Q_i$, and $Z_{12}$ is the complexified configuration space of the spherical linkage formed by coupling $Q_1$ and $Q_2$. (Note that the projection of $Z_{12}$ to the $(w_1,w_2)$-plane is the zero set of the resultant $R_{12}$ defined in Section \ref{sec:AlgReform}.)

A component of $Z_i$ is called \emph{trivial}, if it has the form $z = \const$ or $w_i = \const$.
For non-trivial components, the restrictions of the maps $Z_{12} \to Z_i$, $i = 1,2$ are branched covers between Riemann surfaces. If the solution set $Z_{all}$ of the system \eqref{eqn:PolSystem} is one-dimensional, then the map $Z_{all} \to Z_{12}$ is also a branched cover, and we obtain a diagram of branched covers shown on Figure \ref{fig:BigDiagram}, Section \ref{sec:Proof}. All maps in this diagram are at most two-fold, and the configuration spaces $Z_i$ can be classified as in Section \ref{sec:ConfSpace}. However, the analysis of the diagram is still complicated enough, and we bring more structure in it by distinguishing certain sorts of couplings.

First, a coupling $(Q_1,Q_2)$ is called \emph{involutive}, if its configuration space $Z_{12}$ carries an involution that changes the value of $z$ but preserves $w_i$:
$$
j_{12} \colon (w_1, z, w_2) \mapsto (w_1, z', w_2)
$$
This means that the involutions
\[
j_1 \colon (w_1, z) \mapsto (w_1, z') \quad \text{and} \quad j_2 \colon (w_2, z) \mapsto (w_2, z'')
\]
on $Z_1$ and $Z_2$ are compatible: $z' = z''$. Geometrically, $j_1$ is a ``folding'' of the quadrilateral $Q_1$. Therefore a coupling is involutive if and only if the two marked angles on Figure \ref{fig:InvCoupling}, left, remain equal during the deformation. (The coupling between $Q_1$ and $Q_2$ on Figure \ref{fig:InvCoupling} is chosen in the same way as on Figure \ref{fig:Couple4}, right; the involutivity has a more nice geometric meaning if the coupling is chosen as on Figure \ref{fig:Couple4}, left.)

\begin{figure}[ht]
\centering
\includegraphics{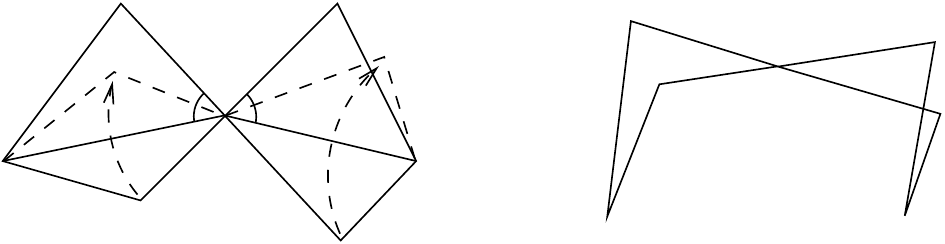}
\caption{An involutive coupling and the action of $j_{12}$ on its configuration space.}
\label{fig:InvCoupling}
\end{figure}

Involutive couplings are classified in Lemma \ref{lem:InvolLift}.

Second, a coupling $(Q_1, Q_2)$ is called \emph{reducible}, if the algebraic set $Z_{12}$ is reducible, while $Z_1$ and $Z_2$ are not. This property is hard to visualize: the real configuration space of a coupling may have several connected components all of which are parts of one complex component.

Reducibility of $(Q_1,Q_2)$ is related to the reducibility of the resultant $R_{12}(w_1,w_2)$ defined in Section \ref{sec:AlgReform}. Moreover, a coupling is involutive if and only if the resultant is a square. For more details see Section \ref{sec:Pullback}.

\subsection{Types of spherical quadrilaterals and associated parameters}
\label{sec:QuadParam}
Let $Q$ be a spherical quadrilateral with side lengths $\alpha$, $\beta$, $\gamma$, $\delta$, in this cyclic order.
The form of the configuration space of $Q$ depends on the number and the type of solutions of the equation
\begin{equation}
\label{eqn:Grashof}
\alpha \pm \beta \pm \gamma \pm \delta = 0 (\mod 2\pi)
\end{equation}

\begin{dfn}
\label{dfn:QuadTypes}
A spherical quadrilateral $Q$ is said to be
\begin{itemize}
\item
of \emph{elliptic} type, if equation \eqref{eqn:Grashof} has no solutions;
\item
of \emph{conic} type, if equation \eqref{eqn:Grashof} has exactly one solution;
\item
a \emph{deltoid}, if it has two pairs of equal adjacent sides, and an \emph{antideltoid}, if it has two pairs of adjacent sides complementing each other to $\pi$;
\item
an \emph{isogram}, if pairs of opposite sides have equal lengths, and an \emph{antiisogram}, if lengths of opposite sides complement each other to $\pi$.
\end{itemize}
\end{dfn}
Note that circumscribable spherical quadrilaterals are characterized by $\alpha + \gamma = \beta + \delta$ and thus ar of conic type according to our terminology (possibly degenerating to a deltoid). The link of a vertex of a quad surface satisfies this condition if and only if the faces adjacent to this vertex are tangent to a circular cone. Because of this, surfaces with this property are called \emph{conical meshes} by Pottman and Wallner \cite{PW08} and find an application in the freeform architecture. We call these quadrilaterals conic for a different reason: their configuration spaces are described by quadratic equations.

At the beginning of Section \ref{sec:ConfSpace} we show that every spherical quadrilateral belongs to one of the types above. The theorems below, also proved in Section \ref{sec:ConfSpace}, describe the configuration spaces of all types of quadrilaterals. They provide a vocabulary for the description of flexible Kokotsakis polyhedra that follows in Section \ref{sec:List}.

\begin{thm}
\label{thm:ParamLin}
The configuration space $Z$ of an (anti)isogram with side lengths $\alpha, \beta, \gamma, \delta$ has the following form.
\begin{enumerate}
\item
If $\alpha = \beta = \gamma = \delta = \frac{\pi}2$, then all components of $Z$ are trivial:
\[
Z = \{z = 0\} \cup \{z = \infty\} \cup \{w = 0\} \cup \{w = \infty\}
\]
\item
If $Q$ is an antiisogram and at the same time (anti)deltoid, then $Z$ has one non-trivial component with the equation
\begin{equation}
\label{eqn:ConfIsoOne}
z = \kappa w, \text{ where } \kappa =
\begin{cases}
\frac{1}{\cos \alpha}, &\text{if } \alpha = \beta = \pi - \gamma = \pi - \delta\\
-\cos\alpha, &\text{if } \alpha = \delta = \pi - \beta = \pi - \gamma
\end{cases}
\end{equation}
\item
If $Q$ is an antiisogram and not an (anti)deltoid, then $Z$ has two non-trivial components of the form
\[
z = \kappa w
\]
where
\begin{equation}
\label{eqn:ConfIsoTwo}
\kappa \in \left\{ \frac{\sin\frac{\alpha-\beta}2}{\sin\frac{\alpha+\beta}2}, \frac{\cos\frac{\alpha-\beta}2}{\cos\frac{\alpha+\beta}2} \right\}
\end{equation}
\item
If $Q$ is an isogram: $\alpha = \gamma$, $\beta = \delta$ with not all sides equal to $\frac{\pi}{2}$, then the non-trivial components of $Z$ are described by the equation
\begin{equation}
\label{eqn:ConfIso}
z = \frac1{\kappa w}
\end{equation}
where $\kappa$ taking one or two values according to whether $Q$ is a deltoid and given by the same formulas as in the antiisogram case.
\end{enumerate}

\end{thm}

\begin{thm}
\label{thm:ParamDegenCon}
Let $Q$ be a deltoid or an antideltoid that is neither isogram nor antiisogram. Then the affine part of the non-trivial component of its configuration space has the following parametrization.
\begin{enumerate}
\item
If $\alpha = \delta,\ \beta = \gamma \quad \text{or} \quad \alpha + \delta = \pi = \beta + \gamma$,
then
\begin{equation}
\label{eqn:ParamDelt1}
z^m = p\sin t, \qquad w = \epsilon \sqrt{-\mu} e^{it},
\end{equation}
where
$$
\begin{aligned}
p &= \sqrt{\frac{\sin^2\delta}{\sin^2\gamma} - 1}, &\quad
m &=
\begin{cases}
1, &\text{if } \alpha = \delta\\
-1, &\text{if } \alpha + \delta = \pi
\end{cases}\\
\mu &= 
\frac{\tan\delta + \tan\gamma}{\tan\delta - \tan\gamma}, &\quad
\epsilon &=
\begin{cases}
m, &\text{if } \gamma + \delta > \pi,\\
-m, &\text{if } \gamma + \delta < \pi
\end{cases}
\end{aligned}
$$
\item
If $\alpha = \beta,\ \gamma = \delta \quad \text{or} \quad \alpha + \beta = \pi = \gamma + \delta$,
then
\begin{equation}
\label{eqn:ParamDelt2}
z = \iota \sqrt{-\lambda} e^{it}, \qquad w^n = q\sin t
\end{equation}
where
$$
\begin{aligned}
q &= \sqrt{\frac{\sin^2\delta}{\sin^2\alpha} - 1}, &\quad
n &=
\begin{cases}
1, &\text{if } \gamma = \delta\\
-1, &\text{if } \gamma + \delta = \pi
\end{cases}\\
\lambda &= 
\frac{\tan\delta + \tan\alpha}{\tan\delta - \tan\alpha}, &\quad
\iota &=
\begin{cases}
n, &\text{if } \alpha + \delta > \pi,\\
-n, &\text{if } \alpha + \delta < \pi
\end{cases}
\end{aligned}
$$
\end{enumerate}
For square roots we adopt the convention $\sqrt{x} \in \R_+ \cup i\R_+$.
\end{thm}

\begin{thm}
\label{thm:ParamCon}
The affine part of the configuration space of a spherical quadrilateral of the conic type has the parametrization
\begin{equation}
\label{eqn:ParamCon}
z^m = p \sin t, \qquad w^n = q \sin(t + t_0),
\end{equation}
Here the exponents $m$ and $n$ are determined according to the table
\begin{center}
{\renewcommand{\arraystretch}{1.5}
\begin{tabular}{c|cc}
& $n = 1$& $n = -1$\\
\hline
$m = 1$ & $\alpha + \gamma = \beta + \delta$ & $\alpha + \beta = \gamma + \delta$\\
$m = -1$ & $\alpha + \delta = \beta + \gamma$ & $\alpha + \beta + \gamma + \delta = 2\pi$
\end{tabular}
}
\end{center}
The amplitudes are given by
$$
p = \sqrt{\frac{\sin\alpha \sin\delta}{\sin\beta \sin\gamma} - 1} \in \R_{>0} \cup i\R_{>0}, \qquad q = \sqrt{\frac{\sin\gamma \sin\delta}{\sin\alpha \sin\beta} - 1} \in \R_{>0} \cup i\R_{>0};
$$
The phase shift satisfies
$$
\tan t_0 = i \sqrt{\frac{\sin\beta \sin\delta}{\sin\alpha \sin\gamma}}
$$
The indeterminacy in $t_0$ up to the summand $\pi$ is resolved in the table below: it gives an interval in which $t_0$ lies, depending on the values of $pq$ and~$\sigma$.
\begin{center}
{\renewcommand{\arraystretch}{1.5}
\begin{tabular}{c|ccc}
& $pq \in \R_{>0}$ & $pq \in i\R_{>0}$ & $pq \in R_{<0}$\\
\hline
$\sigma < \pi$ & $i\R_{>0}$ & $\frac{\pi}2 + i\R_{>0}$ & $\pi + i\R_{>0}$\\
$\sigma > \pi$ & $\pi + i\R_{>0}$ & $\frac{3\pi}2 + i\R_{>0}$ & $i\R_{>0}$
\end{tabular}
}
\end{center}
where
$$
\sigma =
\begin{cases}
\frac{\alpha + \beta + \gamma + \delta}2, &\text{if } \alpha + \gamma = \beta + \delta\\
\frac{-\alpha + \beta + \gamma - \delta}2 + \pi, &\text{if } \alpha + \beta = \gamma + \delta\\
\frac{\alpha + \beta - \gamma - \delta}2 + \pi, &\text{if } \alpha + \delta = \beta + \gamma\\
\frac{-\alpha + \beta - \gamma + \delta}2 + \pi, &\text{if } \alpha + \beta + \gamma + \delta = 2\pi
\end{cases}
$$



\end{thm}

In the elliptic case introduce the notation
$$
\sigma = \frac{\alpha + \beta + \gamma + \delta}2, \qquad \overline\alpha = \sigma - \alpha = \frac{-\alpha + \beta + \gamma + \delta}2,
$$
and similarly $\overline\beta = \sigma - \beta$, $\overline\gamma = \sigma - \gamma$, $\overline\delta = \sigma - \delta$. Note that applying this transformation twice yields the initial quadruple of numbers. Denote
$$
M = \frac{\sin\alpha \sin\beta \sin\gamma \sin\delta}{\sin\overline\alpha \sin\overline\beta \sin\overline\gamma \sin\overline\delta}, 
$$
Due to the last equation of Lemma \ref{lem:BarRelations}, if equation \eqref{eqn:Grashof} has no solution, then we have $M \ne 1$. More exactly, we have
\begin{gather*}
M < 1 \Leftrightarrow \alpha_{\min} + \alpha_{\max} < \sigma < \pi \text{ or }\alpha_{\min} + \alpha_{\max} > \sigma > \pi\\
M > 1 \Leftrightarrow \alpha_{\min} + \alpha_{\max} > \sigma < \pi \text{ or }\alpha_{\min} + \alpha_{\max} < \sigma > \pi
\end{gather*}
where $\alpha_{\min}$, respectively $\alpha_{\max}$ denotes the smallest, respectively the biggest of the numbers $\alpha$, $\beta$, $\gamma$, $\delta$.

\begin{thm}
\label{thm:ParamSnCn}
The configuration space of a spherical quadrilateral of the elliptic type has the following parametrization.
\begin{enumerate}
\item
If $M < 1$, then
\begin{equation}
\label{eqn:ParamSn}
z = p \sn t, \quad w = q \sn(t + t_0),
\end{equation}
where $\sn$ is the elliptic sine function with modulus $k = \sqrt{1-M}$,
and the phase shift $t_0$ satisfies
\begin{equation*}
\dn t_0 = \sqrt{\frac{\sin\alpha \sin\gamma}{\sin\overline\alpha \sin\overline\gamma}}
\end{equation*}
\item
If $M > 1$, then
\begin{equation}
\label{eqn:ParamCn}
z = p \cn t, \quad w = q \cn(t + t_0)
\end{equation}
where $\cn$ is the elliptic cosine function with modulus $k = \sqrt{1-M^{-1}}$,
and the phase shift $t_0$ satisfies
\begin{equation*}
\dn t_0 = \sqrt{\frac{\sin\overline\alpha \sin\overline\gamma}{\sin\alpha \sin\gamma}}
\end{equation*}
\end{enumerate}
In both cases the amplitudes $p$ and $q$ are given by
$$
p = \sqrt{\frac{\sin\alpha \sin\delta}{\sin\overline\alpha \sin\overline\delta} - 1} \in \R_{>0} \cup i\R_{>0}, \qquad q = \sqrt{\frac{\sin\gamma \sin\delta}{\sin\overline\gamma \sin\overline\delta} - 1} \in \R_{>0} \cup i\R_{>0}
$$
The value of $\dn t_0$ determines the phase shift up to a real half-period $2K$. This indeterminacy is resolved in the table below: it gives an interval in which $t_0$ lies, depending on the values of $pq$ and $\sigma$.
\begin{center}
{\renewcommand{\arraystretch}{1.5}
\begin{tabular}{c|ccc}
& $pq \in \R_{>0}$ & $pq \in i\R_{>0}$ & $pq \in R_{<0}$\\
\hline
$\sigma < \pi$ & $(0,iK')$ & $(K,K+iK')$ & $(2K,2K+iK')$\\
$\sigma > \pi$ & $(2K,2K+iK')$ & $(3K,3K+iK')$ & $(0,iK')$
\end{tabular}
}
\end{center}
\end{thm}


To finish setting up the notation, consider orthodiagonal quadrilaterals, i.~e. those whose side lengths satisfy the relation
$$
\cos\alpha\cos\gamma = \cos\beta\cos\delta
$$
An orthodiagonal quadrilateral is either (anti)deltoid, or of elliptic type, see Section \ref{sec:Orthodiag}. We exclude the case when $\alpha = \beta = \gamma = \delta = \frac\pi{2}$, as it leads only to trivial deformations, see Theorem \ref{thm:ParamLin} and Section \ref{sec:TrivClass}.

We refer to a vertex of a quadrilateral by naming the two sides incident to it. We say that an (anti)deltoid has apices $\alpha\delta$ and $\beta\gamma$, if $\alpha = \delta$, $\beta = \gamma$ or $\alpha+\delta = \pi = \beta+\gamma$.

\begin{dfn}
\label{dfn:InvFactors}
Let $Q$ be an orthodiagonal quadrilateral. We define the \emph{involution factors} at each of its vertices, excluding the apices if $Q$ is an (anti)deltoid, as follows.

The involution factor at the vertex $\alpha\delta$ is
$$
\lambda :=
\begin{cases}
\frac{\tan\delta + \tan\alpha}{\tan\delta - \tan\alpha}, &\text{if } \alpha \ne \frac{\pi}2 \text{ or } \delta \ne \frac{\pi}2\\
\frac{\cos\beta + \cos\gamma}{\cos\beta - \cos\gamma}, &\text{if } \alpha = \delta = \frac{\pi}2
\end{cases}
$$
Similarly, the involution factor at the vertex $\gamma\delta$ is
$$
\mu :=
\begin{cases}
\frac{\tan\delta + \tan\gamma}{\tan\delta - \tan\gamma}, &\text{if } \gamma \ne \frac{\pi}2 \text{ or } \delta \ne \frac{\pi}2\\
\frac{\cos\beta + \cos\alpha}{\cos\beta - \cos\alpha}, &\text{if } \gamma = \delta = \frac{\pi}2
\end{cases}
$$
Besides, for an orthodiagonal quadrilateral of elliptic type we put
$$
\nu :=
\begin{cases}
\frac{(\lambda - 1)(\mu - 1)}{\cos\delta}, &\text{if } \delta \ne \frac\pi{2}\\
2(\mu -1) \tan\alpha, &\text{if } \delta = \gamma = \frac\pi{2}\\
2(\lambda - 1) \tan\gamma, &\text{if } \delta = \alpha = \frac\pi{2}
\end{cases}
$$
for an (anti)deltoid with apex $\alpha\delta$ we put
$$
\zeta := 
\begin{cases}
\frac{\mu - 1}{\cos\delta}, &\text{if } \delta \ne \frac\pi{2}\\
2 \tan\gamma, &\text{if } \delta = \alpha = \frac\pi{2}
\end{cases}
$$
and for an (anti)deltoid with apex $\gamma\delta$ we put
$$
\xi :=
\begin{cases}
\frac{\lambda - 1}{\cos\delta}, &\text{if } \delta \ne \frac\pi{2}\\
2 \tan\alpha, &\text{if } \delta = \gamma = \frac\pi{2}
\end{cases}
$$
\end{dfn}
Note that in the (anti)deltoid case the values of $\lambda$ and $\mu$ coincide with those given in Theorem \ref{thm:ParamDegenCon}. Also there are the identities
\[
p^2 \zeta^2 = 4\mu, \quad q^2 \xi^2 = 4\lambda
\]

The involution factors are well-defined real numbers different from $0$. For example, if $\alpha = \frac{\pi}2$ and $\delta \ne \frac{\pi}2$, then $\lambda = \frac{\infty}{-\infty} = -1$. If $\alpha = \delta$ or $\alpha + \delta = \pi$, so that the denominator or numerator in the first formula for $\lambda$ vanish, then either $Q$ is an (anti)deltoid and has no involution factor at the vertex $\delta\alpha$, or $\alpha = \delta = \frac{\pi}2$, and $\lambda$ is computed by the second formula. The second formula makes sense since $\beta \ne \gamma$ by assumption that $\delta\alpha$ is not an apex.


\subsection{Switching a boundary strip}
\label{sec:Switching}
Here we describe an operation that transforms one flexible Kokotsakis polyhedron to another flexible one.
\emph{Switching the right boundary strip} consists in replacing on Figure \ref{fig:NotPlanAngles} the vertex $C_1$ by its mirror image with respect to $A_1$, and $C_4$ by its mirror image with respect to $A_4$. Switching the left, lower, and upper boundary strips is defined similarly.

In the intrinsic terms, switching the right boundary strip consists in replacing $\beta_1, \beta_4, \gamma_1, \gamma_4$ by their complements to $\pi$. It can transform the spherical quadrilateral $Q_i$, $i = 1, 4$ from an isogram to an antiisogram, or from deltoid to antideltoid, or change $n_i$ to $-n_i$ if $Q_i$ was of conic type. We will use this repeatedly to simplify the case distinction. The action of switching on the equation of the configuration space is described in Lemma~\ref{lem:Switching}.

\section{The list of flexible Kokotsakis polyhedra}
\label{sec:List}

We list all flexible polyhedra modulo the switching operations described in Section \ref{sec:Switching}. The description goes in terms of conditions on the angles $\alpha_i, \beta_i, \gamma_i, \delta_i$.

\begin{mainthm}
\label{thm:Class}
Every flexible Kokotsakis polyhedron with quadrangular base belongs to one of the following classes.
\begin{enumerate}
\item
Orthodiagonal (or T-surfaces, Graf-Sauer surfaces).

\item
Isogonal (or V-surfaces, discrete Voss surfaces).

\item
Equimodular.

\item
Conjugate-modular.

\item
Linear compounds.

\item
Linearly conjugate.

\item
Chimeras.

\item
Trivial.
\end{enumerate}
\end{mainthm}
Linear compounds include previously known polyhedra with reflectional or transitional symmetry. Trivial flexible polyhedra are characterized by constance of some of the dihedral angles during the deformation. Finally, chimeras are called so because they are based on spherical quadrilaterals of different types.

%

\subsection{Orthodiagonal types}
\label{sec:OrthType}
\subsubsection{Orthodiagonal involutive type}
\label{sec:OrthoInv}
A Kokotsakis polyhedron belongs to the \emph{orthodiagonal type}, if its planar angles satisfy the following conditions.

\begin{enumerate}
\item
All spherical quadrilaterals $Q_i$ are orthodiagonal:
$$
\cos\alpha_i \cos\gamma_i = \cos\beta_i \cos\delta_i
$$
Geometrically this means that the plane $C_1A_1A_2$ is orthogonal to the plane $B_1A_1A_4$ (this property is preserved during a deformation of a tetrahedral angle), and the same holds for the corresponding pairs of planes through the vertices $A_2$, $A_3$, $A_4$. We exclude the case $\alpha_i = \beta_i = \gamma_i = \delta_i = \frac\pi{2}$, as it leads to trivial deformations only, which are described in Section \ref{sec:TrivType}.
\item
The couplings of adjacent quadrilaterals are compatible, see Definition \ref{dfn:Compatible}. Geometrically this means that each of the polygonal lines $C_1A_1A_2C_2$, $B_2A_2A_3B_3$, $C_3A_3A_4C_4$, $B_4A_4A_1B_1$ on Figure \ref{fig:NotPlanAngles} remains planar during the deformation. Due to 1) each of the $AB$-planes is orthogonal to each of the $AC$-planes.
\item
The angles of the base quadrilateral satisfy the condition
$$
\cos\delta_1 \cos\delta_3 = \cos\delta_2 \cos\delta_4
$$
Together with $\delta_1 + \delta_2 + \delta_3 + \delta_4$ this implies that the base quadrilateral is a trapezoid.
\end{enumerate}

The compatibility condition 2) can be made more explicit. For example if $\delta_1 + \delta_4 = \delta_2 + \delta_3 = \pi$ and $\alpha_i, \delta_i \ne \frac{\pi}{2}$ for all $i$, then the compatibility is equivalent to
\[
\gamma_1 + \gamma_4 = \pi, \quad \gamma_2 + \gamma_3 = \pi, \quad
\frac{\tan\alpha_1}{\tan\alpha_2} = \frac{\tan\alpha_4}{\tan\alpha_3} = \frac{\tan\delta_1}{\tan\delta_2}
\]
so that, in particular, the $AC$-planes are parallel.

This class of flexible Kokotsakis polyhedra was described by Sauer and Graf in \cite{SG31} and called ``\emph{T-Flache}'' (from ``Trapezflache''). Sauer and Graf proved their flexibility by a geometric argument.


\subsubsection{Orthodiagonal antiinvolutive type}
\label{sec:OrthoAnti}

\begin{enumerate}
\item
All quadrilaterals $Q_i$ are orthodiagonal and elliptic:
$$
\cos\alpha_i \cos\gamma_i = \cos\beta_i \cos\delta_i, \quad \alpha_i \pm \beta_i \pm \gamma_i \pm \delta_i \ne 0 (\mod 2\pi)
$$
\item
The involution factors at common vertices are opposite:
$$
\lambda_1 = -\lambda_2, \quad \mu_1 = -\mu_4, \quad \mu_2 = -\mu_3, \quad \lambda_3 = -\lambda_4
$$
\item
The following relations hold:
$$
\frac{\nu_1^2}{\lambda_1\mu_1} = \frac{\nu_3^2}{\lambda_3\mu_3}, \quad \frac{\nu_2^2}{\lambda_2\mu_2} = \frac{\nu_4^2}{\lambda_4\mu_4}, \quad \frac{\nu_1^2}{\lambda_1\mu_1} + \frac{\nu_2^2}{\lambda_2\mu_2} = 1
$$
\end{enumerate}

\subsection{Isogonal type}
\label{sec:Voss}
A polyhedron of \emph{basic} isogonal type is characterized by the following conditions.
\begin{enumerate}
\item
All quadrilaterals $Q_i$ are antiisograms:
\[
\alpha_i + \gamma_i = \pi = \beta_i + \delta_i, \quad i = 1, 2, 3, 4
\]
\item
One of the following equalities hold:
$$
\kappa_1 \kappa_3 = \kappa_2 \kappa_4,
$$
where $\kappa_i$ are as in Theorem \ref{thm:ParamLin}. Note that $\kappa_i$ may take two values if $Q_i$ is not an (anti)deltoid.
\end{enumerate}

A general polyhedron of isogonal type can be obtained from a basic one by switching some of the boundary strips. For example, switching the left and the right or the upper and the lower boundary strips transforms all antiisograms to isograms.

Flexible polyhedra of isogonal type were described in \cite{SG31} and named \emph{discrete Voss surfaces}. They have two flat realizations corresponding to $\alpha_i + \beta_i + \gamma_i + \delta_i = 2\pi$ and to $\alpha_i - \beta_i + \gamma_i - \delta_i = 0$.
Miura-ori \cite{Ni12} is of this type.

\subsection{Equimodular type}
Here we have two subcases:

\subsubsection{Elliptic case}
\label{sec:EllCase}
Assume that
\begin{equation}
\label{eqn:GenericAlphas}
\alpha_i \pm \beta_i \pm \gamma_i \pm \delta_i \neq
0 (\mod 2\pi)
\end{equation}
for all choices of $\pm$. Introduce the notation
$$
\sigma_i := \frac{\alpha_i + \beta_i + \gamma_i + \delta_i}2,
\quad \overline{\alpha}_i := \sigma_i - \alpha_i = \frac{-\alpha_i +
\beta_i + \gamma_i + \delta_i}2
$$
Similarly, $\overline{\beta}_i := \sigma_i - \beta_i$,
$\overline{\gamma}_i := \sigma_i - \gamma_i$, $\overline{\delta}_i
:= \sigma_i - \delta_i$. Denote
$$
a_i := \frac{\sin\alpha_i}{\sin\overline\alpha_i}, \quad b_i :=
\frac{\sin\beta_i}{\sin\overline\beta_i}, \quad c_i :=
\frac{\sin\gamma_i}{\sin\overline\gamma_i}, \quad d_i :=
\frac{\sin\delta_i}{\sin\overline\delta_i}, \quad
M_i := a_ib_ic_id_i
$$

The polyhedron is of equimodular elliptic type if the following three conditions
are satisfied.
\begin{enumerate}
\item
Quadrilaterals have equal moduli:
$$
M_1 = M_2 = M_3 = M_4
$$
\item
Amplitudes at common vertices are equal:
$$
a_1d_1 = a_2d_2, \quad b_2c_2 = b_3c_3, \quad a_3d_3 = a_4d_4, \quad b_4c_4 =
b_1c_1
$$
\item
The sum of shifts is a period:
$$
t_1 \pm t_2 \pm t_3 \pm t_4 \in \Lambda,
$$
where $t_i$ is determined up to the real half-period $2K$ by
$$
\dn t_i =
\begin{cases}
\sqrt{a_ic_i}, &\text{if } M < 1\\
\frac1{\sqrt{a_ic_i}}, &\text{if } M > 1
\end{cases}
$$
and the indeterminacy is resolved at the end of Theorem \ref{thm:ParamSnCn}.
\end{enumerate}

Together with the restriction $\delta_1 + \delta_2 + \delta_3 + \delta_4 = 2\pi$ this gives 9 conditions on 16 angles of a Kokotsakis polyhedron. However, as follows from \cite[\S 9, 10]{Kok33}, these conditions are not independent, so that a polyhedron of elliptic equimodular type depends on 8 parameters instead of 7.

\subsubsection{Conic case}
\label{sec:ConEquimod}
The polyhedron is of \emph{basic} equimodular conic type if the following
conditions are satisfied.
\begin{enumerate}
\item
\label{it:CycQuad}
All spherical quadrilaterals $Q_i$ are circumscribed: $\alpha_i + \gamma_i = \beta_i +
\delta_i$.
\item
Their amplitudes at common vertices are equal:
\begin{equation*}
\begin{aligned}
\frac{\sin\alpha_1\sin\delta_1}{\sin\beta_1\sin\gamma_1} =
\frac{\sin\alpha_2\sin\delta_2}{\sin\beta_2\sin\gamma_2} &\quad&
\frac{\sin\alpha_2\sin\beta_2}{\sin\gamma_2\sin\delta_2} =
\frac{\sin\alpha_3\sin\beta_3}{\sin\gamma_3\sin\delta_3}\\
\frac{\sin\alpha_3\sin\delta_3}{\sin\beta_3\sin\gamma_3} =
\frac{\sin\alpha_4\sin\delta_4}{\sin\beta_4\sin\gamma_4} &\quad&
\frac{\sin\alpha_4\sin\beta_4}{\sin\gamma_4\sin\delta_4} =
\frac{\sin\alpha_1\sin\beta_1}{\sin\gamma_1\sin\delta_1}
\end{aligned}
\end{equation*}
\item
The sum of shifts is a multiple of $2\pi$:
$$
t_1 \pm t_2 \pm t_3 \pm t_4 \in 2\pi\Z
$$
\end{enumerate}

Here $t_i \in \C$ is determined up to $\pi$ by
$$
\tan t_i = i \sqrt{\frac{\sin\beta_i\sin\delta_i}{\sin\alpha_i\sin\gamma_i}},
$$
(one of the $i$'s on the right hand side is the imaginary unit), with the indeterminacy resolved as in Theorem \ref{thm:ParamCon}.

Switching all four boundary strips results in replacing condition 1) by
\begin{enumerate}
\item[\ref{it:CycQuad}')]
All spherical quadrilaterals $Q_i$ have perimeter $2\pi$:
$\alpha_i + \beta_i + \gamma_i + \delta_i$.
\end{enumerate}
This is an origami.

\subsection{Conjugate-modular type}
\label{sec:ConjMod}
\subsubsection{First elliptic conjugate-modular type}
\label{sec:ConjMod1}
\begin{enumerate}
\item
All quadrilaterals $Q_i$ are elliptic of $\cn$-type: $M_i > 1$ for all $i$.
\item
The moduli of $Q_1$ and $Q_3$ are equal and conjugate to those of $Q_2$ and $Q_4$:
$$
M_1 = M_3, \quad M_2 = M_4, \quad \frac{1}{M_1} + \frac{1}{M_2} = 1
$$
\item
The amplitudes satisfy the following relations:
$$
\frac{p_1}{p_2} = \pm i \frac{k}{k'}, \quad \frac{q_3}{q_2} = \pm i \frac{k}{k'}, \quad \frac{p_3}{p_4} = \pm i \frac{k}{k'}, \quad \frac{q_1}{q_4} = \pm i \frac{k}{k'},
$$
where $k = \sqrt{1 - M_1^{-1}}$ is the Jacobi modulus of $Q_1$.
\item
For some of the eight choices of the $\pm$-signs on the left hand side, the following equation hold:
$$
t_1 \pm it_2 \pm t_3 \pm it_4 =
\begin{cases}
0 (\mod \Lambda), &\text{if } \frac{p_1}{p_2} = \frac{q_3}{q_2}, \frac{q_1}{q_4} = \frac{p_3}{p_4} \text{ or } \frac{p_1}{p_2} = -\frac{q_3}{q_2}, \frac{q_1}{q_4} = -\frac{p_3}{p_4}\\
2K (\mod \Lambda), &\text{if } \frac{p_1}{p_2} = \frac{q_3}{q_2}, \frac{q_1}{q_4} = -\frac{p_3}{p_4} \text{ or } \frac{p_1}{p_2} = -\frac{q_3}{q_2}, \frac{q_1}{q_4} = \frac{p_3}{p_4}
\end{cases}
$$
\end{enumerate}
The last condition implies that $t_1 \pm t_3 \in K\Z + iK'\Z$ and $t_2 \pm t_4 \in K'\Z + iK\Z$.

\subsubsection{Second elliptic conjugate-modular type}
\label{sec:ConjMod2}
\begin{enumerate}
\item
All quadrilaterals $Q_i$ are elliptic of $\cn$-type: $M_i > 1$ for all $i$.
\item
The moduli of $Q_1$ and $Q_4$ are equal and conjugate to those of $Q_2$ and $Q_3$:
$$
M_1 = M_4, \quad M_2 = M_3, \quad \frac{1}{M_1} + \frac{1}{M_2} = 1
$$
\item
The amplitudes satisfy the following relations:
$$
\frac{p_1}{p_2} = \pm i \frac{k}{k'}, \quad q_2 = q_3, \quad \frac{p_3}{p_4} = \pm i \frac{k'}{k}, \quad q_4 = q_1,
$$
where $k = \sqrt{1 - M_1^{-1}}$ is the Jacobi modulus of $Q_1$.
\item
For some of the eight choices of the $\pm$-signs on the left hand side, the following equation hold:
$$
t_1 \pm it_2 \pm it_3 \pm t_4 =
\begin{cases}
0 (\mod \Lambda), &\text{if } \frac{p_1}{p_2} = \frac{p_4}{p_3}\\
2K (\mod \Lambda), &\text{if } \frac{p_1}{p_2} = -\frac{p_4}{p_3}
\end{cases}
$$
\end{enumerate}

The last condition is rather restrictive. It implies that $t_1 \pm t_4 \in K\Z + iK'\Z$ and $t_2 \pm t_3 \in K'\Z + iK\Z$. Note that if $t_1 \pm t_4$ is a half-period of $\Lambda$, then the polyhedron is of linear compound type described in Section \ref{sec:LinComp}.

\subsection{Linear compound type}
\label{sec:LinComp}
A coupling $(Q_1, Q_2)$ is called \emph{linear} if it results in a linear dependence between $\tan \frac{\psi_1}2$ and $\tan \frac{\psi_2}2$:
\begin{equation}
\label{eqn:LinCoupling}
w_1 = cw_2
\end{equation}
(Strictly speaking, if one of the components of the configuration space of the coupling has this equation.) If the coupling $(Q_3, Q_4)$ is also linear with the same value of $c$, then the polyhedron is flexible, and we say that it belongs to the \emph{linear compound type}.
Note that switching the right or the left boundary strip transforms the linear dependence to $w_1w_2 = c'$.

Planar-symmetric and translational types from \cite{Sta10} are special cases of linear compounds.

Below we list all linear couplings $(Q_1,Q_2)$ together with the corresponding values of $c$.

\subsubsection{Coupling of (anti)isograms}
\label{sec:LinIso}
Both $Q_1$ and $Q_2$ are isograms or both are antiisograms. The coefficient $c$ in \eqref{eqn:LinCoupling} equals
\[
c = \frac{\kappa_2}{\kappa_1}
\]
with $\kappa_i$ as in Theorem \ref{thm:ParamLin}. Note that switching the lower boundary strip transforms isograms to antiisograms.

\subsubsection{Linear lateral (anti)deltoid coupling}
\label{sec:LinLatDelt}
\begin{enumerate}
\item
$Q_1$ and $Q_2$ are both deltoids or both antideltoids, coupled laterally:
$$
\begin{aligned}
&\alpha_i = \beta_i, \quad \gamma_i = \delta_i, \quad \text{ or}\\
&\alpha_i + \beta_i = \pi = \gamma_i + \delta_i
\end{aligned}
$$
\item
The coupling is involutive, that is $\lambda_1 = \lambda_2$. If $\alpha_i, \delta_i \ne \frac{\pi}2$, then this is equivalent to
$$
\frac{\tan\alpha_1}{\tan\delta_1} = \frac{\tan\alpha_2}{\tan\delta_2}
$$
\end{enumerate}
The coefficient $c$ equals
$$
c =
\begin{cases}
\frac{\xi_2}{\xi_1}, &\text{if }Q_1 \text{ and }Q_2 \text{ are deltoids},\\
\frac{\xi_1}{\xi_2}, &\text{if }Q_1 \text{ and }Q_2 \text{ are antideltoids}
\end{cases}
$$
with $\xi_i$ as in Definition \ref{dfn:InvFactors}.
In particular, if $\delta_i \ne \frac{\pi}2$, then $\frac{\xi_2}{\xi_1} = \frac{\cos\delta_1}{\cos\delta_2}$.

\subsubsection{Linear frontal (anti)deltoid coupling}
\label{sec:LinFrontDelt}
\begin{enumerate}
\item
$Q_1$ and $Q_2$ are both deltoids or both antideltoids, coupled frontally:
$$
\begin{aligned}
&\alpha_i = \delta_i, \quad \beta_i = \gamma_i, \quad &&i = 1,2, \text{ or}\\
&\alpha_i + \delta_i = \pi = \beta_i + \gamma_i, \quad &&i = 1,2
\end{aligned}
$$
\item
The coupling is reducible, which means that
$$
\frac{\sin\alpha_1}{\sin\beta_1} = \frac{\sin\alpha_2}{\sin\beta_2}
$$
\end{enumerate}
The coefficient $c$ equals
$$
c = \frac{\epsilon_1 \sqrt{-\mu_1}}{\epsilon_2 \sqrt{-\mu_2}} = \pm \sqrt{\frac{\sin(\alpha_1+\beta_1) \sin(\alpha_2-\beta_2)}{\sin(\alpha_2+\beta_2) \sin(\alpha_1-\beta_1)}},
$$
see Theorem \ref{thm:ParamDegenCon}. Switching the lower boundary strip transforms deltoids to antideltoids while preserving the coefficient $c$.

\subsubsection{Linear elliptic coupling}
\label{sec:LinEll}
Configuration spaces of $Q_1$ and $Q_2$ are elliptic curves. Besides, in the notation of Section \ref{sec:EllCase} we have
\begin{enumerate}
\item
$Q_1$ and $Q_2$ have equal moduli $M_1 = M_2$, as well as equal amplitudes at the common vertex:
$$
a_1d_1 = a_2d_2, \quad b_1c_1 = b_2c_2
$$
\item
The shift difference is a (real) half-period: $t_1-t_2 \in \{0,2K\}$. Or, equivalently,
$$
a_1c_1 = a_2c_2
$$
\end{enumerate}
Then we have
$$
c =
\begin{cases}
\sqrt{\frac{c_1d_1 - 1}{c_2d_2 - 1}}, & \text{if }\sin\sigma_1\sin\sigma_2 > 0,\\
-\sqrt{\frac{c_1d_1 - 1}{c_2d_2 - 1}}, & \text{if }\sin\sigma_1\sin\sigma_2 < 0
\end{cases}
$$

\subsubsection{Linear conic coupling}
\label{sec:LinCon}
\begin{enumerate}
\item
Quadrilaterals $Q_1$ and $Q_2$ are either both circumscribed or both have perimeter $2\pi$:
\begin{subequations}
\begin{equation}
\label{eqn:BothCyc}
\alpha_i + \gamma_i = \beta_i + \delta_i, \quad i = 1,2
\end{equation}
\begin{equation}
\label{eqn:BothPer2Pi}
\alpha_i + \beta_i + \gamma_i + \delta_i = 2\pi, \quad i = 1,2
\end{equation}
\end{subequations}
\item
Their amplitudes at the common vertex are equal, and the shifts difference is a half-period:
$$
\frac{\sin\alpha_1}{\sin\beta_1} = \frac{\sin\alpha_2}{\sin\beta_2},
\quad \frac{\sin\gamma_1}{\sin\delta_1} =
\frac{\sin\gamma_2}{\sin\delta_2}
$$
\end{enumerate}
The value of $c$ is given by the following table
\begin{center}
{\renewcommand{\arraystretch}{1.5}
\begin{tabular}{c|cc}
& \eqref{eqn:BothCyc}& \eqref{eqn:BothPer2Pi}\\
\hline
$\sin\sigma_1 \sin\sigma_2 > 0$ & $\bar c$ & $\bar c^{-1}$\\
$\sin\sigma_1 \sin\sigma_2 < 0$ & $-\bar c$ & $- \bar c^{-1}$
\end{tabular}
}
\quad where $\bar c =\sqrt{\frac{\frac{\sin\gamma_1\sin\delta_1}{\sin\alpha_1\sin\beta_1}-1}{\frac{\sin\gamma_2\sin\delta_2}{\sin\alpha_2\sin\beta_2}-1}}$
\end{center}

\subsection{Linearly conjugate types}
\label{sec:LinConj}
These are polyhedra, where $Q_2$ and $Q_4$ are antiisograms, so that they result in linear dependencies $z = \kappa_2 w_2$ and $u = \kappa_4 w_1$. Polyhedra $Q_1$ and $Q_3$ must be of the same type, and their equations are related by the above linear substitutions.

\subsubsection{Linearly conjugate antideltoids}
\label{sec:LinDelt}
\begin{enumerate}
\item
$Q_1$ and $Q_3$ are ``parallel'' antideltoids:
$$
\alpha_1 + \delta_1 = \pi = \beta_1 + \gamma_1, \quad \alpha_3 + \beta_3 = \pi = \gamma_3 + \delta_3
$$
\item
$Q_2$ and $Q_4$ are antiisograms:
$$
\alpha_2 + \gamma_2 = \pi = \beta_2 + \gamma_2, \quad \alpha_4 + \gamma_4 = \pi = \beta_4 + \delta_4
$$
\item
The following conditions are satisfied:
$$
\kappa_4^2 \mu_1 = \lambda_3, \quad \kappa_4 \zeta_1 = \kappa_2 \xi_3,
$$
where $\mu_1, \zeta_1, \lambda_3, \xi_3$ are as in Definition \ref{dfn:InvFactors}, and $\kappa_2, \kappa_4$ as in Theorem \ref{thm:ParamLin}.
\end{enumerate}

\subsubsection{Linearly conjugate conics}
\label{sec:LinConjCon}
\begin{enumerate}
\item
$Q_1$ and $Q_3$ are conic quadrilaterals of perimeter $2\pi$:
$$
\alpha_1 + \beta_1 + \gamma_1 + \delta_1 = 2\pi = \alpha_3 + \beta_3 + \gamma_3 + \delta_3,
$$
and equations $\alpha_i \pm \beta_i \pm \gamma_i \pm \delta_i \equiv 0 (\mod\, 2\pi)$ have for $i = 1, 3$ no other solutions.
\item
$Q_2$ and $Q_4$ are antiisograms:
$$
\alpha_2 + \gamma_2 = \pi = \beta_2 + \gamma_2, \quad \alpha_4 + \gamma_4 = \pi = \beta_4 + \delta_4
$$
\item
The following relations hold:
$$
q_3 = |\kappa_2| p_1, \quad q_1 = |\kappa_4| p_3, \quad t_1 =
\begin{cases}
t_3, &\text{if }c_2c_4 > 0\\
t_3 + \pi, &\text{if }c_2c_4 < 0
\end{cases}
$$
\end{enumerate}

\subsubsection{Linearly conjugate elliptics}
\label{sec:LinConjEll}
\begin{enumerate}
\item
Quadrilaterals $Q_1$ and $Q_3$ are elliptic and have the same modulus:
$$
M_1 = M_3
$$
\item
$Q_2$ and $Q_4$ are antiisograms:
$$
\alpha_2 + \gamma_2 = \pi = \beta_2 + \gamma_2, \quad \alpha_4 + \gamma_4 = \pi = \beta_4 + \delta_4
$$
\item
The following relations hold:
$$
p_1 = |\kappa_2| q_3, \quad p_3 = |\kappa_4| q_1, \quad t_1 =
\begin{cases}
t_3, &\text{if }\kappa_2\kappa_4 > 0\\
t_3 + 2K, &\text{if }\kappa_2\kappa_4 < 0
\end{cases}
$$
\end{enumerate}

\subsection{Chimeras}
\label{sec:TransType}
In the Borges' taxonomy of animals these would correspond to the ``innumerable ones'' or ``et cetera''.

Kokotsakis polyhedra listed here contain quadrilaterals of different types. Potentially they could be used to join pieces of flexible quad surfaces of equimodular, orthodiagonal, or isogonal types.

\subsubsection{Conic-deltoid}
\label{sec:3e}
\begin{enumerate}
\item
Quadrilaterals $Q_2$ and $Q_3$ have perimeter $2\pi$:
$$
\alpha_2 + \beta_2 + \gamma_2 + \delta_2 = 2\pi, \quad \alpha_3 + \beta_3 + \gamma_3 + \delta_3 = 2\pi
$$
and form a reducible coupling: $q_2 = q_3$.
\item
Quadrilaterals $Q_1$ and $Q_4$ are antideltoids reducibly coupled to $Q_2$ and $Q_3$:
$$
\begin{aligned}
\alpha_1 + \delta_1 = \pi = \beta_1 + \gamma_1, &\quad p_1 = p_2\\
\alpha_4 + \delta_4 = \pi = \beta_4 + \gamma_4, &\quad p_4 = p_3
\end{aligned}
$$
\item
The number $t_0$ determined by
$$
\epsilon_1 \sqrt{-\mu_1} = \epsilon_4 \sqrt{-\mu_4} e^{it_0}
$$
is related to the shifts of $Q_2$ and $Q_3$ through
$$
t_0 = \pm t_1 \pm t_2,
$$
where any of the four combinations of signs are allowed.
\end{enumerate}

\subsubsection{First orthodiagonal-isogram}
\label{sec:TransType1}
\begin{enumerate}
\item
The quadrilateral $Q_1$ is orthodiagonal:
$$
\cos\alpha_1 \cos\gamma_1 = \cos\beta_1 \cos\delta_1
$$
\item
$Q_2$ and $Q_4$ are antideltoids that form involutive couplings with $Q_1$:
\begin{gather*}
\alpha_2 + \beta_2 = \pi = \gamma_2 + \delta_2, \quad \lambda_1 = \lambda_2\\
\alpha_4 + \delta_4 = \pi = \beta_4 + \gamma_4, \quad \mu_1 = \mu_4
\end{gather*}
(In particular, this implies that $Q_1$ is not an (anti)deltoid.)
\item
$Q_3$ is an isogram: $\alpha_3 = \gamma_3$, $\beta_3 = \delta_3$.
\item
the parameters of $Q_i$ must satisfy the relation
\[
\nu_1 = \kappa_3 \xi_2 \zeta_4
\]
with $\kappa$, $\nu$, $\xi$, and $\zeta$ as in Theorem \ref{thm:ParamLin} and Definition \ref{dfn:InvFactors}. In particular, if $\delta_i \ne \frac{\pi}2$ for $i = 1,2,4$, then this condition becomes
$$
\kappa_3 \cos\delta_1 = \cos\delta_2 \cos\delta_4
$$
\end{enumerate}

\subsubsection{Second orthodiagonal-isogram}
\label{sec:TransType2}
\begin{enumerate}
\item
The quadrilateral $Q_1$ is orthodiagonal:
$$
\cos\alpha_1 \cos\gamma_1 = \cos\beta_1 \cos\delta_1
$$
Besides, $Q_1$ is neither deltoid nor antideltoid.
\item
$Q_2$ is an antideltoid forming an involutive coupling with $Q_1$:
$$
\alpha_2 + \beta_2 = \pi = \gamma_2 + \delta_2, \quad \lambda_1 = \lambda_2
$$
\item
$Q_3$ is a deltoid coupled with $Q_2$ frontally, and $Q_4$ is an antiisogram:
$$
\alpha_3 = \beta_3, \quad \gamma_3 = \delta_3, \quad \alpha_4 + \gamma_4 = \beta_4 + \delta_4
$$
\item
The following two equations hold:
$$
\kappa_4^2 \mu_1 = \lambda_3, \quad \kappa_4 \nu_1 = \xi_2 \xi_3,
$$
with $\kappa$, $\lambda$, $\mu$, $\nu$, $\xi$ as in Theorem \ref{thm:ParamLin} and Definition \ref{dfn:InvFactors}.
\end{enumerate}

\subsubsection{Conic-isogram}
\label{sec:IConIso}
\begin{enumerate}
\item
The quadrilateral $Q_2$ is a conic quadrilateral of perimeter $2\pi$:
$$
\alpha_2 + \beta_2 + \gamma_2 + \delta_2 = 2\pi,
$$
equation $\alpha_2 \pm \beta_2 \pm \gamma_2 \pm \delta_2 \equiv 0 (\mod 2\pi)$ has no other solutions.
\item
Quadrilaterals $Q_1$ and $Q_3$ are antideltoids reducibly coupled with $Q_2$:
$$
p_1 = p_2, \quad q_2 = q_3
$$
\item
$Q_4$ is an isogram with
\[
2\epsilon_1 i \sqrt{-\mu_1} e^{\pm it_2} \kappa_4 = q_3 \xi_3
\]
\end{enumerate}

\subsubsection{Conic-antiisogram}
\label{sec:IConAntiiso}
\begin{enumerate}
\item
The quadrilateral $Q_2$ is a conic quadrilateral of perimeter $2\pi$:
$$
\alpha_2 + \beta_2 + \gamma_2 + \delta_2 = 2\pi,
$$
equation $\alpha_2 \pm \beta_2 \pm \gamma_2 \pm \delta_2 \equiv 0 (\mod 2\pi)$ has no other solutions.
\item
Quadrilaterals $Q_1$ and $Q_3$ are antideltoids reducibly coupled with $Q_2$:
$$
p_1 = p_2, \quad q_2 = q_3
$$
\item
$Q_4$ is an antiisogram with
\[
\epsilon_1 i \sqrt{-\mu_1} e^{\pm it_2} q_3 \xi_3 \kappa_4 = 2
\]
\end{enumerate}

\subsubsection{Frontally coupled deltoid and antideltoid vs. reducibly coupled elliptics}
\label{sec:DeltAntideltEll}
\begin{enumerate}
\item
$(Q_1,Q_2)$ is a frontal coupling of a deltoid with an antideltoid:
$$
\alpha_1 = \delta_1, \quad \beta_1 = \gamma_1, \quad \alpha_2 + \delta_2 = \pi = \beta_2 + \gamma_2,
$$
which is irreducible: $\frac{\sin\alpha_1}{\sin\beta_1} \ne \frac{\sin\alpha_2}{\sin\beta_2}$.
\item
Quadrilaterals $Q_3$ and $Q_4$ are elliptic with configuration spaces para\-met\-rized by $\sn$ and form a reducible coupling:
$$
M_3 = M_4 < 1, \quad p_3 = p_4
$$
\item
either the sum or the difference of shifts equals a quarter-period with the imaginary part $\frac{K'}2$:
$$
\pm t_3 \pm t_4 = lK + i\frac{K'}2
$$
\item
The following relations hold:
$$
\begin{aligned}
\mu_1 &= \frac{q_1^2}k, &\mu_2 &= \frac{q_2^2}k, &\zeta_1\zeta_2 &= \frac{2(1+k)}{k\sqrt{k}} q_1q_2, &\text{ if } l &= 0\\
\mu_1 &= -\frac{q_1^2}k, &\mu_2 &= -\frac{q_2^2}k, &\zeta_1\zeta_2 &= \frac{2i(1-k)}{k\sqrt{k}} q_1q_2, &\text{ if } l &= 1\\
\mu_1 &= \frac{q_1^2}k, &\mu_2 &= \frac{q_2^2}k, &\zeta_1\zeta_2 &= -\frac{2(1+k)}{k\sqrt{k}} q_1q_2, &\text{ if } l &= 2\\
\mu_1 &= -\frac{q_1^2}k, &\mu_2 &= -\frac{q_2^2}k, &\zeta_1\zeta_2 &= -\frac{2i(1-k)}{k\sqrt{k}} q_1q_2, &\text{ if } l &= 3
\end{aligned}
$$
\end{enumerate}

\subsubsection{Reducible conic-deltoid coupling vs. isogram-deltoid coupling}
\label{sec:2b}
\begin{enumerate}
\item
The quadrilateral $Q_3$ is a conic quadrilateral of perimeter $2\pi$:
$$
\alpha_3 + \beta_3 + \gamma_3 + \delta_3 = 2\pi,
$$
equation $\alpha_3 \pm \beta_3 \pm \gamma_3 \pm \delta_3 \equiv 0 (\mod 2\pi)$ has no other solutions.
\item
The quadrilateral $Q_4$ is an antideltoid:
$$
\alpha_4 + \delta_4 = \pi = \beta_4 + \gamma_4
$$
that forms a reducible coupling with $Q_3$:
$$
p_3 = p_4
$$
\item
The quadrilateral $Q_1$ is an antideltoid:
$$
\alpha_1 + \delta_1 = \pi = \beta_1 + \gamma_1
$$
\item
The quadrilateral $Q_2$ is an antiisogram:
$$
\alpha_2 + \gamma_2 = \pi = \beta_2 + \delta_2
$$
\item
The following relations hold:
$$
\mu_1 = \mu_4 e^{\pm 2it_3}, \quad \zeta_1 = c_2 \frac{2i \epsilon_4\sqrt{-\mu_4} e^{\pm it_3}}{q_2},
$$
where $\pm$ in the first equation must match the $\pm$ in the second equation, and $z = c_2 w_2$ is the equation of an irreducible component of~$Z_2$.
\end{enumerate}

\subsubsection{Three reducibly coupled conics vs. an isogram}
\label{sec:2ciiCon}
\begin{enumerate}
\item
Quadrilaterals $Q_3$, $Q_4$, $Q_1$ are conic with perimeter $2\pi$:
$$
\alpha_i + \beta_i + \gamma_i + \delta_i = 2\pi, \quad \text{for }i = 3, 4, 1,
$$
and equations $\alpha_i \pm \beta_i \pm \gamma_i \pm \delta_i \equiv 0 (\mod\, 2\pi)$ have no other solutions.
\item
Couplings $(Q_3, Q_4)$ and $(Q_4, Q_1)$ are reducible:
$$
p_3 = p_4, \quad q_4 = q_1
$$
\item
Quadrilateral $Q_2$ is an antiisogram:
$$
\alpha_2 + \gamma_2 = \pi = \beta_2 + \delta_2
$$
\item
The following relations hold:
$$
q_3 = |\kappa_2|p_1, \quad t_1 \pm t_3 \pm t_4 =
\begin{cases}
0, \text{if }\kappa_2 > 0\\
\pi, \text{if }\kappa_2 < 0
\end{cases}
$$
\end{enumerate}

\subsubsection{Three reducibly coupled elliptics vs. an isogram}
\label{sec:2ciiEll}
\begin{enumerate}
\item
Quadrilaterals $Q_3$, $Q_4$, $Q_1$ are elliptic and have equal moduli:
$$
M_1 = M_3 = M_4
$$
\item
Couplings $(Q_3, Q_4)$ and $(Q_4, Q_1)$ are reducible:
$$
p_3 = p_4, \quad q_4 = q_1
$$
\item
Quadrilateral $Q_2$ is an antiisogram:
$$
\alpha_2 + \gamma_2 = \pi = \beta_2 + \delta_2
$$
\item
The following relations hold:
$$
p_1 = |\kappa_2|q_3, \quad t_1 \pm t_3 \pm t_4 =
\begin{cases}
0, \text{if }\kappa_2 > 0\\
2K, \text{if }\kappa_2 < 0
\end{cases}
$$
\end{enumerate}

\subsubsection{Involutively coupled orthodiagonal and antideltoid vs. reducibly coupled conic and deltoid}
\label{sec:3bii}
\begin{enumerate}
\item
The quadrilateral $Q_1$ is orthodiagonal:
$$
\cos\alpha_1 \cos\gamma_1 = \cos\beta_1 \cos\delta_1,
$$
but is neither deltoid nor antideltoid.
\item
$Q_2$ is an antideltoid forming an involutive coupling with $Q_1$:
$$
\alpha_2 + \beta_2 = \pi = \gamma_2 + \delta_2, \quad \lambda_1 = \lambda_2
$$
\item
$Q_3$ is circumscribed:
$$
\alpha_3 = \gamma_3, \quad \beta_3 = \delta_3
$$
\item
$Q_4$ is a deltoid forming with $Q_3$ a reducible coupling:
$$
\alpha_4 = \delta_4, \quad \beta_4 = \gamma_4, \quad
\frac{\sin^2\alpha_4}{\sin^2\beta_4} = \frac{\sin\alpha_3\sin\delta_3}{\sin\beta_3\sin\gamma_3}
$$
\item
The following relations hold:
$$
\mu_1 = \mu_4 e^{\pm 2it_3}, \quad \frac{\nu_1}{\xi_2} =  \frac{2i \epsilon_4 \sqrt{-\mu_4} e^{\pm it_3}}{q_3}
$$
(the $\pm$ in the first equation must match the $\pm$ in the second equation).
\end{enumerate}

\subsection{Trivial types}
\label{sec:TrivType}
We call an isometric deformation of a Kokotsakis polyhedron \emph{trivial}, if it preserves one of the dihedral angles at the central face. In Section \ref{sec:TrivClass} trivial deformations are classified, which leads to the four types shown on Figure \ref{fig:TrivFlex}.
During an isometric deformation, shaded faces move while white faces stay fixed in $\R^3$. Any other trivially flexible polyhedron is obtained from one of these by switching some of the immobile boundary strips.

\begin{figure}[ht] 
\centering
\includegraphics{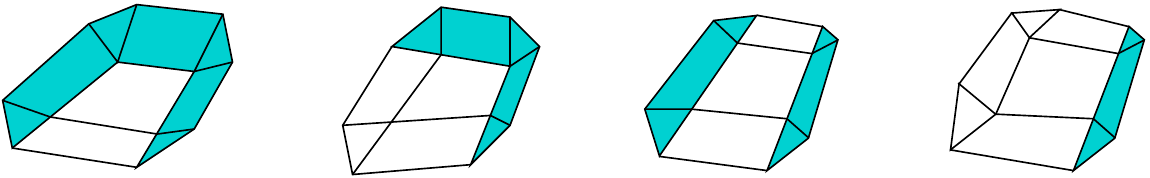}
\caption{Trivially flexible Kokotsakis polyhedra.}
\label{fig:TrivFlex}
\end{figure}

\section{Configuration space of a spherical four-bar linkage}
\label{sec:ConfSpace}
\subsection{Side lengths of a spherical quadrilateral}
A spherical quadrilateral is a collection of four points (vertices) on the unit sphere, together with a cyclic order such that no two consecutive vertices form a pair of antipodes. Thus for any two consecutive vertices there is a unique great circle passing through them; the shortest of its arcs joining the vertices is called a side of the quadrilateral. We allow the edges of a quadrilateral to intersect and overlap, except in one of the ways shown on Figure \ref{fig:DegenQuad}.

\begin{figure}[ht]
\centering
\includegraphics{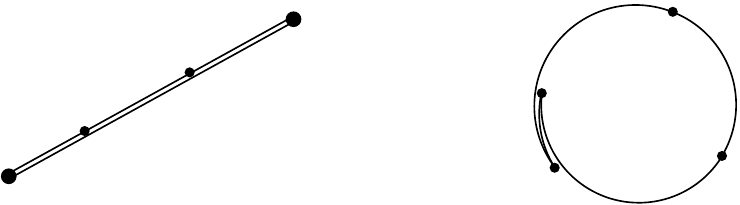}
\caption{The only degenerate configurations that don't count as quadrilaterals.}
\label{fig:DegenQuad}
\end{figure}

\begin{lem}
\label{lem:SpherQuadLengths}
There exists a spherical quadrilateral with side lengths $(\alpha, \beta, \gamma, \delta)$ if and only if the inequalities
\begin{equation}
\label{eqn:QuadSides}
0 < \alpha < \pi, \quad \alpha < \beta + \gamma + \delta < \alpha + 2\pi
\end{equation}
are fulfilled, as are all those obtained by exchanging $\alpha$ with $\beta$, $\gamma$, or $\delta$.
\end{lem}
\begin{proof}
The necessity of the first three inequalities in \eqref{eqn:QuadSides} follows from our definition of a spherical quadrilateral. To prove the necessity of the fourth one, replace a vertex by its antipode and apply the third inequality. Note that the equality cases in the third and the fourth inequalities correspond to degenerate quadrilaterals shown on Figure \ref{fig:DegenQuad}.

To prove the sufficiency, note that a spherical triangle with side lengths $\alpha, \beta, \epsilon$ exists for all $\epsilon \in [|\alpha-\beta|, \min\{\alpha+\beta, 2\pi-\alpha-\beta\}]$. The inequalities \eqref{eqn:QuadSides} together with those obtained by permutations imply that
$$
[|\alpha-\beta|, \min\{\alpha+\beta, 2\pi-\alpha-\beta\}] \cap [|\gamma-\delta|, \min\{\gamma+\delta, 2\pi-\gamma-\delta\}] \ne \emptyset,
$$
therefore we can construct a quadrilateral by putting two triangles together.
\end{proof}

Recall the notation
$$
\sigma = \frac{\alpha + \beta + \gamma + \delta}2, \qquad \overline\alpha = \sigma - \alpha = \frac{-\alpha + \beta + \gamma + \delta}2,
$$
$\overline\beta = \sigma - \beta$, $\overline\gamma = \sigma - \gamma$, $\overline\delta = \sigma - \delta$.
Inequalities \eqref{eqn:QuadSides} imply similar inequalities for $\overline\alpha, \overline\beta, \overline\gamma, \overline\delta$, so that the latter are also side lengths of a quadrilateral.

The following lemma is proved by using standard trigonometric identities.
\begin{lem}
\label{lem:BarRelations}
For any $\alpha, \beta, \gamma, \delta \in \R$ and for $\sigma$, $\overline\alpha, \overline\beta, \overline\gamma, \overline\delta$ as defined above the following identities (and all obtained from them by permutations) hold.
$$
\overline\alpha + \overline\beta = \gamma + \delta
$$
$$
\sin\overline\alpha \sin\overline\beta - \sin\alpha \sin\beta = \sin\sigma \sin(\sigma-\alpha-\beta) = \sin\gamma\sin\delta - \sin\overline\gamma \sin\overline\delta
$$
\begin{multline*}
\sin\alpha \sin\beta - \sin\overline\gamma \sin\overline\delta = \sin(\sigma-\alpha-\gamma)\sin(\sigma-\beta-\gamma)\\
= \sin \overline \alpha \sin\overline\beta - \sin\gamma \sin\delta
\end{multline*}
\begin{multline*}
\sin\sigma \sin(\sigma-\alpha-\beta) \sin(\sigma-\beta-\gamma) \sin(\sigma-\alpha-\gamma)\\
= \sin\alpha \sin\beta \sin\gamma \sin\delta - \sin\overline\alpha \sin\overline\beta \sin\overline\gamma \sin\overline\delta
\end{multline*}
\end{lem}

\subsection{The configuration space as an algebraic curve}
\label{sec:ConfSpaceDef}
For any numbers $\alpha, \beta, \gamma, \delta$ that satisfy inequalities \eqref{eqn:QuadSides}, the associated \emph{configuration space} is the set of all spherical quadrilaterals with side lengths $\alpha, \beta, \gamma, \delta$ in this cyclic order, up to an orientation-preserving isometry. A quadrilateral with given side lengths is uniquely determined by the values of two adjacent angles; on the other hand, these angles satisfy a certain relation. By performing the substitution
\begin{equation}
\label{eqn:ZWSubst}
z = \tan \frac{\phi}2, \quad w = \tan \frac{\psi}2,
\end{equation}
where $\phi$ and $\psi$ are as on Figure \ref{fig:FourbarZW},
we arrive at a polynomial equation in $z$ and $w$.

\begin{figure}[ht]
\centering
\begin{picture}(0,0)%
\includegraphics{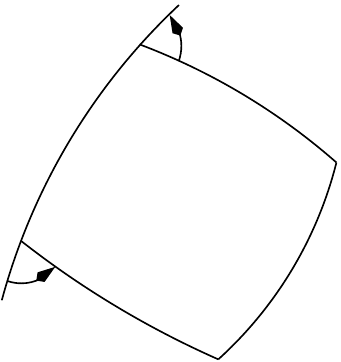}%
\end{picture}%
\setlength{\unitlength}{4144sp}%
\begingroup\makeatletter\ifx\SetFigFont\undefined%
\gdef\SetFigFont#1#2#3#4#5{%
  \reset@font\fontsize{#1}{#2pt}%
  \fontfamily{#3}\fontseries{#4}\fontshape{#5}%
  \selectfont}%
\fi\endgroup%
\begin{picture}(1546,1639)(533,-1239)
\put(616,-1028){\makebox(0,0)[lb]{\smash{{\SetFigFont{10}{12.0}{\rmdefault}{\mddefault}{\updefault}{\color[rgb]{0,0,0}$\phi$}%
}}}}
\put(997,-1125){\makebox(0,0)[lb]{\smash{{\SetFigFont{10}{12.0}{\rmdefault}{\mddefault}{\updefault}{\color[rgb]{0,0,0}$\alpha$}%
}}}}
\put(687,-190){\makebox(0,0)[lb]{\smash{{\SetFigFont{10}{12.0}{\rmdefault}{\mddefault}{\updefault}{\color[rgb]{0,0,0}$\delta$}%
}}}}
\put(1438,253){\makebox(0,0)[lb]{\smash{{\SetFigFont{10}{12.0}{\rmdefault}{\mddefault}{\updefault}{\color[rgb]{0,0,0}$\psi$}%
}}}}
\put(1716,  5){\makebox(0,0)[lb]{\smash{{\SetFigFont{10}{12.0}{\rmdefault}{\mddefault}{\updefault}{\color[rgb]{0,0,0}$\gamma$}%
}}}}
\put(1879,-897){\makebox(0,0)[lb]{\smash{{\SetFigFont{10}{12.0}{\rmdefault}{\mddefault}{\updefault}{\color[rgb]{0,0,0}$\beta$}%
}}}}
\end{picture}%
\caption{Angles $\phi$ and $\psi$ determine the shape of the quadrilateral.}
\label{fig:FourbarZW}
\end{figure}
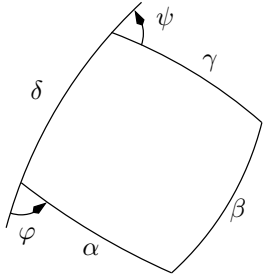

\begin{lem}
\label{lem:EqConfSpace}
The configuration space of quadrilaterals with side lengths $\alpha$, $\beta$, $\gamma$, $\delta$ in this cyclic order is the solution set of the equation
\begin{equation}
\label{eqn:AdjZ}
c_{22} z^2 w^2 + c_{20} z^2 + c_{02} w^2 + 2 c_{11} zw + c_{00} = 0,
\quad \text{where}
\end{equation}
\begin{equation*}
\begin{aligned}
c_{22} &= \sin\frac{\alpha + \beta + \gamma - \delta}2
\sin\frac{\alpha - \beta + \gamma - \delta}2 = \sin\overline\delta \sin(\sigma-\beta-\delta)\\
c_{20} &= \sin\frac{\alpha - \beta - \gamma - \delta}2
\sin\frac{\alpha + \beta - \gamma - \delta}2 = \sin\overline\alpha \sin(\sigma-\beta-\alpha)\\
c_{02} &= \sin\frac{\alpha + \beta - \gamma + \delta}2
\sin\frac{\alpha - \beta - \gamma + \delta}2 = \sin\overline\gamma \sin(\sigma-\beta-\gamma)\\
c_{11} &= -\sin \alpha \sin \gamma\\
c_{00} &= \sin\frac{\alpha - \beta + \gamma + \delta}2
\sin\frac{\alpha + \beta + \gamma + \delta}2 = \sin\overline\beta \sin\sigma
\end{aligned}
\end{equation*}
\end{lem}
Lemma \ref{lem:EqConfSpace} is proved in \cite{Sta10}. See also \cite{Izm_Conf4Bar}.

The substitution \eqref{eqn:ZWSubst} identifies $\R/2\pi\Z$ with $\R \cup \{\infty\} = \RP^1$. Therefore equation \eqref{eqn:AdjZ} must be viewed as an equation in two projective variables. This is achieved by \emph{bihomogenization}
\begin{equation}
\label{eqn:Bihom}
c_{22} z_1^2 w_1^2 + c_{20} z_1^2 w_0^2 + c_{02} z_0^2 w_1^2 + 2 c_{11} z_1 w_1 z_0 w_0 + c_{00} z_0^2 w_0^2 = 0,
\end{equation}
where $z = \frac{z_1}{z_0}$, $w = \frac{w_1}{w_0}$,
see \cite{Izm_Conf4Bar} for more details.

\begin{dfn}
The solution set of equation \eqref{eqn:Bihom} in $(\CP^1)^2$ is called the \emph{complexified configuration space} of quadrilaterals with side lengths $\alpha$, $\beta$, $\gamma$, $\delta$ and is denoted by $Z(\alpha,\beta,\gamma,\delta)$ or just briefly by $Z$.
\end{dfn}


The following lemma will be useful in the next section.
\begin{lem}
\label{lem:Switching}
Let $\alpha, \beta, \gamma, \delta$ be a quadruple of numbers satisfying inequalities \eqref{eqn:QuadSides}. Then the map
$$
\begin{aligned}
(\CP^1)^2 &\to (\CP^1)^2\\
(z,w) &\mapsto (-z^{-1}, w)
\end{aligned}
$$
restricts to a bijection $Z(\alpha,\beta,\gamma,\delta) \to Z(\pi-\alpha,\pi-\beta,\gamma,\delta)$, and the map
$$
\begin{aligned}
(\CP^1)^2 &\to (\CP^1)^2\\
(z,w) &\mapsto (z, -w^{-1})
\end{aligned}
$$
restricts to a bijection $Z(\alpha,\beta,\gamma,\delta) \to Z(\alpha,\pi-\beta,\pi-\gamma,\delta)$.
\end{lem}
\begin{proof}
It suffices to prove only the first part.
This can be done by making the substitutions $\alpha \to \pi-\alpha$, $\beta \to \pi-\beta$ in \eqref{eqn:AdjZ}. The bijection between the real parts of the configuration spaces is established by replacing the vertex $\alpha\beta$ by its antipode and using $\tan\frac{\phi + \pi}2 = - (\tan\frac{\phi}2)^{-1}$, see Figure~\ref{fig:Switching}.
\end{proof}

\begin{figure}[ht]
\centering
\begin{picture}(0,0)%
\includegraphics{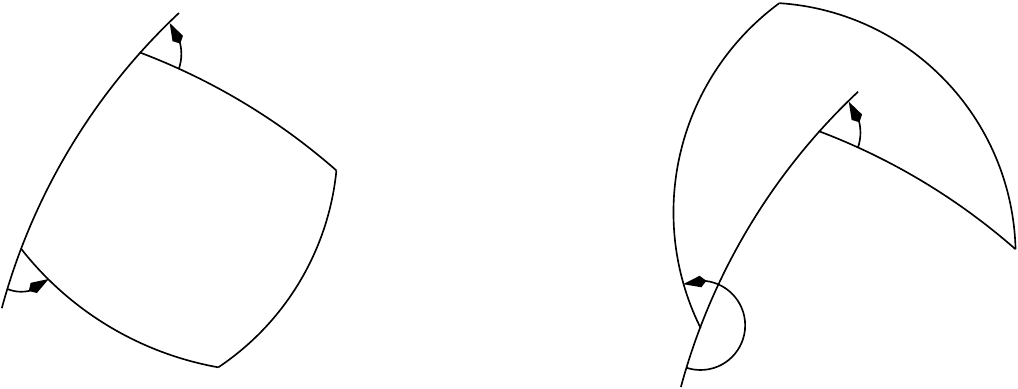}%
\end{picture}%
\setlength{\unitlength}{4144sp}%
\begingroup\makeatletter\ifx\SetFigFont\undefined%
\gdef\SetFigFont#1#2#3#4#5{%
  \reset@font\fontsize{#1}{#2pt}%
  \fontfamily{#3}\fontseries{#4}\fontshape{#5}%
  \selectfont}%
\fi\endgroup%
\begin{picture}(4651,1771)(533,-1329)
\put(616,-1028){\makebox(0,0)[lb]{\smash{{\SetFigFont{10}{12.0}{\rmdefault}{\mddefault}{\updefault}{\color[rgb]{0,0,0}$\phi$}%
}}}}
\put(1716,  5){\makebox(0,0)[lb]{\smash{{\SetFigFont{10}{12.0}{\rmdefault}{\mddefault}{\updefault}{\color[rgb]{0,0,0}$\gamma$}%
}}}}
\put(1438,253){\makebox(0,0)[lb]{\smash{{\SetFigFont{10}{12.0}{\rmdefault}{\mddefault}{\updefault}{\color[rgb]{0,0,0}$\psi$}%
}}}}
\put(687,-190){\makebox(0,0)[lb]{\smash{{\SetFigFont{10}{12.0}{\rmdefault}{\mddefault}{\updefault}{\color[rgb]{0,0,0}$\delta$}%
}}}}
\put(3241,-196){\makebox(0,0)[lb]{\smash{{\SetFigFont{10}{12.0}{\rmdefault}{\mddefault}{\updefault}{\color[rgb]{0,0,0}$\pi-\alpha$}%
}}}}
\put(4543,-107){\makebox(0,0)[lb]{\smash{{\SetFigFont{10}{12.0}{\rmdefault}{\mddefault}{\updefault}{\color[rgb]{0,0,0}$\psi$}%
}}}}
\put(4771,164){\makebox(0,0)[lb]{\smash{{\SetFigFont{10}{12.0}{\rmdefault}{\mddefault}{\updefault}{\color[rgb]{0,0,0}$\pi-\beta$}%
}}}}
\put(1899,-920){\makebox(0,0)[lb]{\smash{{\SetFigFont{10}{12.0}{\rmdefault}{\mddefault}{\updefault}{\color[rgb]{0,0,0}$\beta$}%
}}}}
\put(957,-1141){\makebox(0,0)[lb]{\smash{{\SetFigFont{10}{12.0}{\rmdefault}{\mddefault}{\updefault}{\color[rgb]{0,0,0}$\alpha$}%
}}}}
\put(3937,-1207){\makebox(0,0)[lb]{\smash{{\SetFigFont{10}{12.0}{\rmdefault}{\mddefault}{\updefault}{\color[rgb]{0,0,0}$\phi + \pi$}%
}}}}
\put(3943,-730){\makebox(0,0)[lb]{\smash{{\SetFigFont{10}{12.0}{\rmdefault}{\mddefault}{\updefault}{\color[rgb]{0,0,0}$\delta$}%
}}}}
\put(4590,-497){\makebox(0,0)[lb]{\smash{{\SetFigFont{10}{12.0}{\rmdefault}{\mddefault}{\updefault}{\color[rgb]{0,0,0}$\gamma$}%
}}}}
\end{picture}%
\caption{Isomorphism between the spaces $Z(\alpha,\beta,\gamma,\delta)$ and $Z(\alpha,\pi-\beta,\pi-\gamma,\delta)$.}
\label{fig:Switching}
\end{figure}
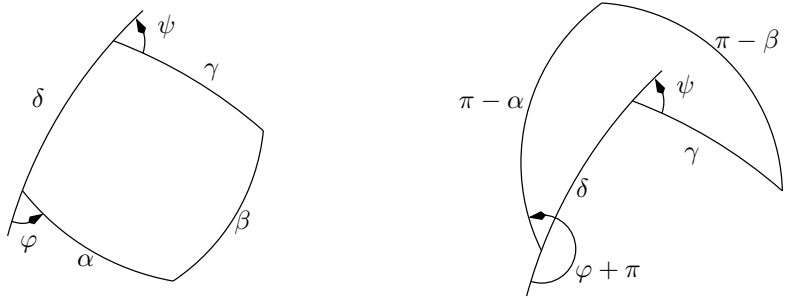




\subsection{Classifying configuration spaces}
\label{sec:ClassConfSp}
Here we prove Theorems \ref{thm:ParamLin}--\ref{thm:ParamSnCn}.

Let $Q$ be a spherical quadrilateral with side lengths $\alpha$, $\beta$, $\gamma$, $\delta$ in this cyclic order.
The shape of the configuration space will depend on the number of solutions of the equation
\begin{equation}
\label{eqn:Grashof2}
\alpha \pm \beta \pm \gamma \pm \delta \equiv 0 (\mod 2\pi)
\end{equation}
Because of \eqref{eqn:QuadSides}, there is no solution with one or three minus signs. Thus, every solution of \eqref{eqn:Grashof2} corresponds either to the sum of two sides being equal to the sum of two others or to the sum of all sides being equal to $2\pi$.

If equation \eqref{eqn:Grashof2} has at least two solutions, then it is easy to show that $Q$ has either two pairs of sides of equal lengths or two pairs of sides, lengths in each pair complementing each other to $\pi$. If these are pairs of opposite sides, then $Q$ is an isogram or antiisogram; if these are pairs of adjacent sides, then $Q$ is a deltoid or antideltoid. It follows that every quadrilateral belongs to one of the types described in Definition \ref{dfn:QuadTypes}.


\begin{proof}[Proof of Theorem \ref{thm:ParamLin}]
If $Q$ is an antiisogram, then equation \eqref{eqn:AdjZ} becomes
\begin{equation}
\label{eqn:IsoAdjZ}
\sin(\alpha-\delta) z^2 + 2\sin\alpha zw + \sin(\alpha+\delta) w^2 = 0
\end{equation}
If $\alpha = \delta = \frac{\pi}2$, then the bihomogenization \eqref{eqn:Bihom} yields $z_1z_0w_1w_0 = 0$. That is, the configuration space consists of four trivial components: $z = 0$, $z = \infty$, $w = 0$, and $w = \infty$. Similarly, if $\alpha = \beta$ or $\alpha + \beta = \pi$, then there are two trivial components and one non-trivial of the form $z = \kappa w$ with $\kappa$ given by \eqref{eqn:ConfIsoOne}. Finally, if $\alpha \ne \beta$ and $\alpha + \beta \ne \pi$, then by solving the quadratic equation \eqref{eqn:IsoAdjZ} (for which the identity $\sin(\alpha-\beta)\sin(\alpha+\beta) = \sin^2\alpha - \sin^2\beta$ might be useful) we find two non-trivial components with $\kappa$ as given in \eqref{eqn:ConfIsoTwo}.

If $Q$ is an isogram, then equation \eqref{eqn:AdjZ} becomes
$$
\sin(\alpha-\beta) z^2w^2 - 2\sin\alpha zw + \sin(\alpha+\beta) = 0,
$$
and the argument is similar. Alternatively, one can use the first switching isomorphism of Lemma \ref{lem:Switching}.
\end{proof}

To deal with the (anti)deltoid case, we need the following lemma.
\begin{lem}
\label{lem:ParamDegCon}
The affine algebraic curve $aw^2 + 2bzw + c=0$ with $a,b,c \ne 0$ has the parametrization $z = p \sin t$, $w = re^{it}$, where
$$
p = \sqrt{\frac{ac}{b^2}}, \quad r =
\begin{cases}
\sqrt{-\frac{c}{a}}, &\text{if } bc > 0\\
-\sqrt{-\frac{c}{a}}, &\text{if }bc < 0
\end{cases}
$$
\end{lem}
The proof is straightforward. Recall our convention $\sqrt{x} \in i\R_+$ for $x<0$.

Note that the $(1,2)$-bihomogenization of the curve $aw^2 + 2bzw + c = 0$ contains two additional points $(\infty, 0)$ and $(\infty, \infty)$.

\begin{proof}[Proof of Theorem \ref{thm:ParamDegenCon}]
If $Q$ satisfies $\alpha = \delta$, $\beta = \gamma$, then equation \eqref{eqn:AdjZ} becomes
\begin{equation}
\label{eqn:AffDelt}
\sin(\delta-\gamma) w^2 - 2\sin\gamma zw + \sin(\delta+\gamma) = 0
\end{equation}
The $(2,2)$-bihomogenization \eqref{eqn:Bihom} contains a trivial component $z = \infty$. The affine part of the non-trivial component can be parametrized according to Lemma \ref{lem:ParamDegCon}, which yields parametrization \eqref{eqn:ParamDelt1} for the case $m = 1$. The case $m = 1$ follows by the first switching isomorphism of Lemma \ref{lem:Switching}. Finally, parametrizations \eqref{eqn:ParamDelt2} are obtained by exchanging $z$ with $w$ and $\alpha$ with~$\gamma$.
\end{proof}

Let now $Q$ be a quadrilateral of conic type. Consider first the case when the unique solution of equation \eqref{eqn:Grashof2} is $\alpha + \gamma = \beta + \delta$.
Then we have
\begin{gather*}
\sigma = \alpha + \gamma = \beta + \delta,\\
\overline\alpha = \gamma, \quad \overline\beta = \delta, \quad \overline\gamma = \alpha, \quad \overline\delta = \beta,\\
\sigma - \beta - \alpha = \delta - \alpha, \quad \sigma - \beta - \gamma = \delta - \gamma
\end{gather*}
It follows that $c_{22}=0$ in \eqref{eqn:AdjZ}, and that the other coefficients are
\begin{equation}
\label{eqn:CoeffCon}
\begin{split}
c_{20} = \sin\gamma \sin(\delta-\alpha), \quad c_{02} = \sin\alpha \sin(\delta-\gamma),\\
c_{11} = -\sin\alpha \sin\gamma, \quad c_{00} = \sin\sigma \sin\delta
\end{split}
\end{equation}

\begin{lem}
\label{lem:ConicParam}
The affine algebraic curve
\begin{equation}
\label{eqn:ConNonDeg}
c_{20}z^2 + c_{02}w^2 + 2c_{11}zw + c_{00} = 0
\end{equation}
with $c_{20} \ne 0, c_{02} \ne 0, c_{00} \ne 0, c_{11}^2 - c_{20}c_{02} \ne 0$ has the parametrization
$$
z = p \sin t, \quad w = q \sin(t+t_0),
$$
where the amplitudes and the phase shift are given by
$$
p = \sqrt{\frac{c_{02}c_{00}}{c_{11}^2 - c_{20}c_{02}}}, \quad q = \sqrt{\frac{c_{02}c_{00}}{c_{11}^2 - c_{20}c_{02}}}, \quad \cos t_0 = -\frac{c_{11}c_{00}}{(c_{11}^2 - c_{20}c_{02})pq}
$$
(the phase shift is determined only up to the sign).
\end{lem}
\begin{proof}
It is easy to show that the functions $x = \sin t$, $y = \sin(t+t_0)$ parametrize the curve
\begin{equation}
\label{eqn:ConParPrim}
x^2 + y^2 - 2\cos t_0 xy - \sin^2 t_0 = 0
\end{equation}
Substitution $x = \frac{z}{p}$, $y = \frac{w}{q}$ transforms this to
$$
q^2z^2 + p^2w^2 - 2pq \cos t_0 zw - p^2q^2 \sin^2 t_0 = 0
$$
and the formulas for $p$, $q$, $\cos t_0$ are found by solving the proportion
$$
q^2 : p^2 : -pq \cos t_0 : -p^2q^2 \sin^2 t_0 = c_{20} : c_{02} : c_{11} : c_{00}
$$
\end{proof}

In order to apply Lemma \ref{lem:ConicParam} to the equation \eqref{eqn:AdjZ}, we need the following specialization of Lemma \ref{lem:BarRelations}.
\begin{lem}
\label{lem:BarIdentCon}
If $\alpha + \gamma = \beta + \delta =: \sigma$, then the following identities hold.
$$
\begin{aligned}
\sin\alpha\sin\gamma - \sin\beta\sin\delta &= \sin(\delta-\alpha)\sin(\delta-\gamma)\\
\sin\gamma\sin\delta - \sin\alpha\sin\beta &= \sin\sigma \sin(\delta-\alpha)\\
\sin\alpha\sin\delta - \sin\beta\sin\gamma &= \sin\sigma \sin(\delta-\gamma)
\end{aligned}
$$
\end{lem}

\begin{proof}[Proof of Theorem \ref{thm:ParamCon}]
If $\alpha + \gamma = \beta + \delta$, then equation \eqref{eqn:AdjZ} takes the form \eqref{eqn:ConNonDeg}
with $c_{ij}$ as in \eqref{eqn:CoeffCon}. If equation \eqref{eqn:Grashof2} has no other solutions, then $c_{20}, c_{02}, c_{00} \ne 0$. Besides, due to Lemma \ref{lem:BarIdentCon} we have
$$
c_{11}^2 - c_{20}c_{02} = \sin\alpha \sin\beta \sin\gamma \sin\delta \ne 0
$$
Thus we are in a position to apply Lemma \ref{lem:ConicParam}. It yields
$$
p = \sqrt{\frac{\sin\sigma\sin(\delta-\gamma)}{\sin\beta\sin\gamma}}, \quad q = \sqrt{\frac{\sin\sigma\sin(\delta-\alpha)}{\sin\alpha\sin\beta}},
$$
which coincides with the formulas given in Theorem \ref{thm:ParamCon} due to Lemma \ref{lem:BarIdentCon}. Furthermore we have
\begin{equation}
\label{eqn:t0Con}
\cos t_0 = \frac{\sin\sigma}{pq \sin \beta}, \quad
\cos^2 t_0 = \frac{\sin\alpha\sin\gamma}{\sin\alpha\sin\gamma - \sin\beta\sin\delta} \in \R \setminus [0,1]
\end{equation}
It follows that $\tan^2 t_0 = -\frac{\sin\beta\sin\delta}{\sin\alpha\sin\gamma}$. Since the phase shift is determined only up to the sign, we choose $t_0$ with $\Im t_0 > 0$, that is $\Im \tan t_0 > 0$, which leads to the formula in Theorem \ref{thm:ParamCon}. Finally, the range of cosine
$$
i\R_+ \mapsto [0,+\infty), \quad \frac{\pi}2 + i\R_+ \mapsto (-\infty, 1]
$$
and equations \eqref{eqn:t0Con} lead to the table determining $\Re t_0$.

The other cases can be reduced to $\alpha + \gamma = \beta + \delta$ with the help of switching isomorphisms of Lemma \ref{lem:Switching}. For example, if $\alpha + \beta = \gamma + \delta$, then we denote
$$
\alpha' := \alpha, \beta' := \pi - \beta, \gamma' := \pi - \gamma, \delta' := \delta,
$$
so that $\alpha' + \gamma' = \beta' + \delta'$. Hence the configuration space $Z(\alpha', \beta', \gamma', \delta')$ has the parametrization $z = p'\sin t$, $w = q' \sin(t \pm t'_0)$, where $p'$, $q'$, and $t'_0$ are computed by applying the formulas of Theorem \ref{thm:ParamCon} to the angles $\alpha'$, $\beta'$, $\gamma'$, $\delta'$. By composing this with the second switching isomorphism of Lemma \ref{lem:Switching}, we obtain the parametrization
$$
z = p' \sin t, \quad w^{-1} = -q' \sin(t + t'_0) = q' \sin(t + (t'_0 + \pi))
$$
of the space $Z(\alpha,\beta,\gamma,\delta)$. It remains to note that $p'=p$, $q'=q$, $\tan t_0 = \tan t'_0$. On the other hand, we have 
$$
\sigma' = \frac{\alpha'+\beta'+\gamma'+\delta'}2 = \frac{\alpha - \beta - \gamma + \delta}2 + \pi,
$$
which is $<\pi$ if and only if $\sigma > \pi$ with $\sigma$ given by the table in Theorem \ref{thm:ParamCon}. This accounts for $t_0 = t'_0 + \pi$.
\end{proof}

In the elliptic case we use the following lemma.

\begin{lem}
\label{lem:EllParam}
For any given $k \in (0,1)$ and $t_0 \in \C$, the functions
$$
x(t) = \sn(t;k), \quad y(t) = \sn(t+t_0;k)
$$
parametrize the affine algebraic curve
\begin{equation}
\label{eqn:EllParSn}
x^2 + y^2 - k^2 \sn^2 t_0 x^2 y^2 - 2 \cn t_0 \dn t_0 xy - \sn^2 t_0 = 0
\end{equation}
(all elliptic functions in the formula have Jacobi modulus $k$).

Similarly, the functions
$$
x(t) = \cn(t;k), \quad y(t) = \cn(t+t_0;k)
$$
parametrize the affine algebraic curve
\begin{equation}
\label{eqn:EllParCn}
x^2 + y^2 + k^2 \frac{\sn^2 t_0}{\dn^2 t_0} x^2y^2 - 2 \frac{\cn t_0}{\dn^2 t_0} xy - (k')^2 \frac{\sn^2 t_0}{\dn^2 t_0} = 0,
\end{equation}
where $k' = \sqrt{1-k^2}$ is the conjugate modulus.
\end{lem}
See \cite{Izm_Conf4Bar} for a proof and historical references.

\begin{proof}[Proof of Theorem \ref{thm:ParamSnCn}]
If $Q$ is of elliptic type, then neither of the coefficients in equation \eqref{eqn:AdjZ} vanishes. We need to show that the substitution $z = px$, $w = qy$ transforms \eqref{eqn:AdjZ} into one of the equations \eqref{eqn:EllParSn} or \eqref{eqn:EllParCn}, where $p$, $q$, $k$, and $t_0$ are as described in Theorem \ref{thm:ParamSnCn}.

We don't give an analog of Lemma \ref{lem:ConicParam}, because the formulas expressing $p$, $q$, $k$, and $t_0$ through the coefficients $c_{ij}$ are rather complicated. On the other hand, in \cite{Izm_Conf4Bar} it is explained how the values of $p$ and $q$ can be guessed. So let's just substitute
$$
z = \sqrt{\frac{\sin\sigma \sin(\sigma-\beta-\gamma)}{\sin\overline\alpha \sin\overline\delta}} x, \quad w = \sqrt{\frac{\sin\sigma \sin(\sigma-\beta-\alpha)}{\sin\overline\gamma \sin\overline\delta}}
$$
(due to Lemma \ref{lem:BarRelations}, these coefficients are $p$ and $q$ of Theorem \ref{thm:ParamSnCn}) into
\begin{multline*}
\sin\overline\alpha \sin(\sigma-\beta-\alpha) z^2 + \sin\overline\gamma \sin(\sigma-\beta-\gamma) w^2\\
+ \sin\overline\delta \sin(\sigma - \beta - \delta) z^2 w^2 - 2\sin\alpha\sin\gamma zw + \sin\overline\beta \sin\sigma = 0
\end{multline*}
We obtain
\begin{multline}
\label{eqn:AdjZEll}
x^2 + y^2 - \frac{\sin\sigma \sin(\sigma-\alpha-\gamma)}{\sin\overline\alpha \sin\overline\gamma} x^2y^2\\
- \frac{2pq \sin\alpha \sin\gamma \sin\overline\delta}{\sin\sigma \sin(\sigma-\alpha-\beta) \sin(\sigma-\beta-\gamma)} xy + \frac{\sin\overline\beta \sin\overline\delta}{\sin(\sigma-\alpha-\beta) \sin(\sigma-\beta-\gamma)} = 0
\end{multline}

Two cases must be distinguished: $M := \frac{\sin\alpha\sin\beta\sin\gamma\sin\delta}{\sin\overline\alpha \sin\overline\beta \sin\overline\gamma \sin\overline\delta} < 1$ and $M > 1$.

If $M < 1$, then we have to show that equation \eqref{eqn:AdjZEll} arises from \eqref{eqn:EllParSn} by substituting
$$
k = \sqrt{1-M} = \sqrt{\frac{\sin\sigma \sin(\sigma - \alpha - \beta) \sin(\sigma - \beta - \gamma) \sin(\sigma - \beta - \delta)}{\sin\overline\alpha \sin\overline\beta \sin\overline\gamma \sin\overline\delta}}
$$
and $t_0$ as described in Theorem \ref{thm:ParamSnCn}. From
$$
\dn t_0 = \sqrt{\frac{\sin\alpha \sin\gamma}{\sin\overline\alpha \sin\overline\gamma}}
$$
and $\sn^2 t_0 + k^2 \dn^2 t_0 = 1$ we compute
$$
\sn^2 t_0 = \frac{-\sin\overline\beta \sin\overline\delta}{\sin(\sigma-\alpha-\beta)\sin(\sigma-\beta-\gamma)} = \frac{-\sin\overline\beta \sin\overline\delta}{\sin\alpha\sin\gamma - \sin\overline\beta\sin\overline\delta}
$$
It follows that the constant term and the coefficient at $x^2y^2$ in \eqref{eqn:AdjZEll} are equal to $-\sn^2 t_0$ and $-k^2 \sn^2 t_0$, respectively. Further, $\sn^2 t_0 + \cn^2 t_0 = 1$ implies that
$$
\cn^2 t_0 = \frac{\sin\alpha\sin\gamma}{\sin(\sigma-\alpha-\beta)\sin(\sigma-\beta-\gamma)}  \in (-\frac{(k')^2}{k^2}, 0) \cup (1,+\infty)
$$
It follows that the square of the coefficient at $xy$ in \eqref{eqn:AdjZEll} equals $4 \cn^2 t_0 \dn^2 t_0$. The sign of the coefficient resolves the $+2K$-indeterminacy of $t_0$. Indeed, equating the coefficient with $-2 \cn t_0 \dn t_0$ leads to
\begin{equation}
\label{eqn:Cnt0}
\cn t_0 = \frac{\sin\sigma}{pq \sin\overline\delta}
\end{equation}
Together with the following information about the range of elliptic cosine:
$$
\begin{aligned}
(0,iK') &\mapsto (1,+\infty), &\quad (K,K+iK') &\mapsto (0, -i\frac{k'}{k}),\\
(2K,2K+iK') &\mapsto (-\infty, -1), &\quad (3K,3K+iK') &\mapsto (0, i\frac{k'}{k})
\end{aligned}
$$
this yields the table in Theorem \ref{thm:ParamSnCn}.

If $M>1$, then a similar argument shows that equation \eqref{eqn:AdjZEll} arises from \eqref{eqn:EllParCn} by substituting
$$
k = \sqrt{\frac{\sin\sigma \sin(\sigma - \alpha - \beta) \sin(\sigma - \beta - \gamma) \sin(\sigma - \alpha - \gamma)}{\sin\alpha \sin\beta \sin\gamma \sin\delta}}
$$
and $t_0$ as described in Theorem \ref{thm:ParamSnCn}. The value of $\dn t_0$ implies this time that
$$
\sn^2 t_0 = \frac{-\sin\beta \sin\delta}{\sin\overline\alpha \sin\overline\gamma - \sin\beta \sin\delta}, \quad \cn^2 t_0 = \frac{\sin\overline\alpha \sin\overline\gamma}{\sin\overline\alpha \sin\overline\gamma - \sin\beta \sin\delta}
$$
The coefficient at $xy$ yields the same formula \eqref{eqn:Cnt0} for $\cn t_0$.
\end{proof}

\subsection{Branch points and involutions}

Let $Z = Z(\alpha,\beta,\gamma,\delta) \subset (\CP^1)^2$ be the complexified configuration space of a quadrilateral with side lengths $\alpha$, $\beta$, $\gamma$, $\delta$, see Section \ref{sec:ConfSpaceDef}.
An irreducible component of $Z$ is called \emph{trivial}, if it is described by an equation of the form $z = \const$ or $w = \const$.

If $Z^0$ is a non-trivial component of $Z$, then both projections
$$
\begin{aligned}
f \colon Z^0 &\to \CP^1 \qquad &\qquad g \colon Z^0 &\to \CP^1\\
(z,w) &\mapsto z &\qquad (z,w) &\mapsto w
\end{aligned}
$$
are branched covers. Since equation \eqref{eqn:AdjZ} has degree~2 both in $z$ and $w$, the degrees of $f$ and $g$ are at most 2.
Let $A \subset \CP^1$ and $B \subset \CP^1$ be the branch sets of the maps $f$ and $g$ respectively.
Denote by
$$
\begin{aligned}
i \colon Z^0 &\to Z^0 \qquad &\qquad j \colon Z^0 &\to Z^0\\
(z,w) &\mapsto (z,w') &\qquad (z,w) &\mapsto (z',w)
\end{aligned}
$$
the deck transformations of $f$ and $g$ (defined only if the corresponding branched cover is two-fold). Geometrically, involutions $i$ and $j$ act by folding the quadrilateral along one of its diagonals, see Figure \ref{fig:Involutions}.

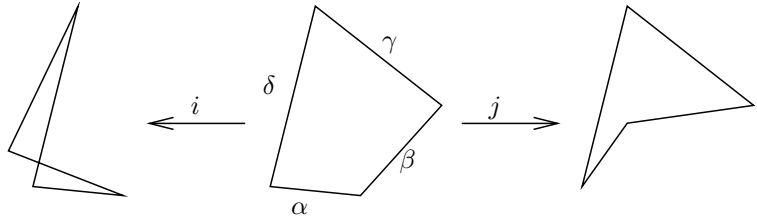
\begin{figure}[ht]
\centering
\begin{picture}(0,0)%
\includegraphics{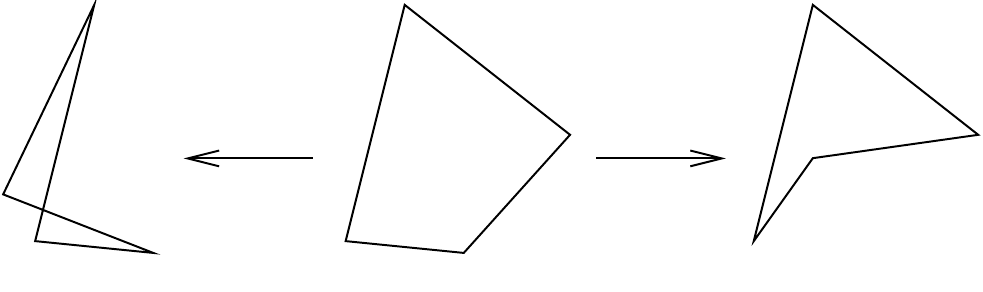}%
\end{picture}%
\setlength{\unitlength}{4972sp}%
\begingroup\makeatletter\ifx\SetFigFont\undefined%
\gdef\SetFigFont#1#2#3#4#5{%
  \reset@font\fontsize{#1}{#2pt}%
  \fontfamily{#3}\fontseries{#4}\fontshape{#5}%
  \selectfont}%
\fi\endgroup%
\begin{picture}(3739,1105)(-1101,-479)
\put(-179, 74){\makebox(0,0)[lb]{\smash{{\SetFigFont{10}{12.0}{\rmdefault}{\mddefault}{\updefault}{\color[rgb]{0,0,0}$i$}%
}}}}
\put(179,173){\makebox(0,0)[lb]{\smash{{\SetFigFont{10}{12.0}{\rmdefault}{\mddefault}{\updefault}{\color[rgb]{0,0,0}$\delta$}%
}}}}
\put(774,404){\makebox(0,0)[lb]{\smash{{\SetFigFont{10}{12.0}{\rmdefault}{\mddefault}{\updefault}{\color[rgb]{0,0,0}$\gamma$}%
}}}}
\put(1306, 74){\makebox(0,0)[lb]{\smash{{\SetFigFont{10}{12.0}{\rmdefault}{\mddefault}{\updefault}{\color[rgb]{0,0,0}$j$}%
}}}}
\put(860,-193){\makebox(0,0)[lb]{\smash{{\SetFigFont{10}{12.0}{\rmdefault}{\mddefault}{\updefault}{\color[rgb]{0,0,0}$\beta$}%
}}}}
\put(321,-424){\makebox(0,0)[lb]{\smash{{\SetFigFont{10}{12.0}{\rmdefault}{\mddefault}{\updefault}{\color[rgb]{0,0,0}$\alpha$}%
}}}}
\end{picture}%
\caption{Involutions $i$ and $j$ on the configuration space of a quadrilateral.}
\label{fig:Involutions}
\end{figure}

\begin{lem}
\label{lem:BranchInvol}
Branch sets and involutions of the configuration space $Z$ of a spherical quadrilateral $Q$ have the following form.
\begin{enumerate}
\item
If $Q$ is an (anti)isogram, then for each of the non-trivial components of $Z$ (there can be $0$, $1$ or $2$ of them) the maps $f$ and $g$ are homeomorphisms.
\item
If $Q$ is an (anti)deltoid with apex $\alpha\delta$ (that is if either $\alpha = \delta$ and $\beta = \gamma$ or $\alpha + \delta = \pi = \beta + \gamma$), then the map $g$ is a homeomorphism, and the map $f$ is two-fold with the branch set and involution
$$
A = \{\pm p^m\}, \quad i \colon t \mapsto \pi - t
$$
Similarly, if $Q$ is an (anti)deltoid with apex $\alpha\beta$, then $f$ is a homeomorphism, and $g$ is a two-fold branched cover with
$$
B = \{\pm q^n\}, \quad j \colon t \mapsto \pi - t
$$
\item
If $Q$ is a conic quadrilateral, then the curve $Z \subset (\CP^1)^2$ is irreducible and has a unique singular point $(0^{-m}, 0^{-n})$ (where $0^{-1} := \infty$). On the regular part of $Z$, both maps $f$ and $g$ are two-fold branched covers with branch sets and involutions
$$
\begin{aligned}
A &= \{\pm p^m\}, &\quad i \colon t &\mapsto \pi - t\\
B &= \{\pm q^n\}, &\quad j \colon t &\mapsto (\pi - 2t_0) - t
\end{aligned}
$$
\item
If $Q$ is an elliptic quadrilateral with $M<1$ (that is $Z$ is parametrized by $\sn$), then both maps $f$ and $g$ are two-fold branched covers with branch sets and involutions
$$
\begin{aligned}
A &= \{\pm p, \pm \frac{p}{k}\}, &\quad i \colon t &\mapsto 2K - t\\
B &= \{\pm q, \pm \frac{q}{k}\}, &\quad j \colon t &\mapsto (2K - 2t_0) - t
\end{aligned}
$$
If $Q$ has $M>1$ (that is $Z$ parametrized by $\cn$), then the branch sets and involutions are given by
$$
\begin{aligned}
A &= \{\pm p, \pm ip \frac{k'}{k}\}, &\quad i \colon t &\mapsto - t\\
B &= \{\pm q, \pm iq \frac{k'}{k}\}, &\quad j \colon t &\mapsto - 2t_0 - t
\end{aligned}
$$
\end{enumerate}
\end{lem}
\begin{proof}
In the (anti)isogram case, the assertion follows from the fact that each non-trivial component has the form $z = c w$ or $zw = c w^{-1}$, see Theorem \ref{thm:ParamLin}. The affine curve $z = c w$ contains in its biprojective completion the point $(\infty,\infty)$, and the curve $z = cw^{-1}$ contains $(0,\infty)$ and $(\infty,0)$, which ensures that both maps $f$ and $g$ are homeomorphisms.

If $Q$ is a deltoid with apex $\alpha\delta$, then $Z^0$ is the $(1,2)$-biprojective completion of the affine curve \eqref{eqn:AffDelt}. It is readily seen that $g$ is a homeomorphism (we have $g^{-1}(0) = (\infty, 0)$ and $g^{-1}(\infty) = (\infty,\infty)$). Formulas for the branch points of $f$ and for the involution $i$ follow from the parametrization \eqref{eqn:ParamDelt1} of the affine part of $Z^0$.

If $Q$ is a conic quadrilateral with $\alpha + \gamma = \beta + \delta$, then the affine part of $Z$ has the form \eqref{eqn:ConNonDeg} which is a nonsingular affine curve. The biprojective completion adds a single point $(\infty,\infty)$, whose neighborhood is equivalent to $z_0^2w_0^2 = 0$. Formulas for branch sets $A$ and $B$ and involutions $i$ and $j$ follow from the parametrization \eqref{eqn:ParamCon}.
For other types of conic quadrilaterals the situation is similar.

Finally, if $Q$ is elliptic, then the parametrization \eqref{eqn:ParamSn}, respectively \eqref{eqn:ParamCn}, defines an analytic diffeomorphism $\C/\Lambda \to \Z$. The branch sets $A$ and $B$ can be computed as the sets of critical values of the functions $z = z(t)$ and $w = w(t)$, respectively. Formulas for the involutions follow from the properties of the functions $\sn$ and $\cn$.
\end{proof}



\subsection{Orthodiagonal quadrilaterals}
\label{sec:Orthodiag}
A quadrilateral is called \emph{orthodiagonal}, if its diagonals are orthogonal to each other.

\begin{lem}
\label{lem:OrthodiagCosSin}
A spherical quadrilateral with side lengths $\alpha, \beta, \gamma, \delta$ in this cyclic order is orthodiagonal if and only if any of the following equivalent conditions is fulfilled.
\begin{subequations}
\begin{equation}
\label{eqn:OrthodiagCos}
\cos\alpha \cos\gamma = \cos\beta \cos\delta
\end{equation}
\begin{equation}
\label{eqn:OrthodiagSin}
\sin\overline{\alpha} \sin\overline{\gamma} = \sin\overline{\beta} \sin\overline{\delta}
\end{equation}
\end{subequations}
In particular, isometric deformations preserve the orthodiagonality.
\end{lem}
\begin{proof}
Condition \eqref{eqn:OrthodiagCos} is equivalent to the orthodiagonality due to the spherical Pythagorean theorem. Equivalence between \eqref{eqn:OrthodiagCos} and \eqref{eqn:OrthodiagSin} follows by simple trigonometry.
\end{proof}

Clearly, deltoids and antideltoids are orthodiagonal. Also, an (anti)isogram is orthodiagonal only if it is an (anti)deltoid at the same time.

%
%

\begin{lem}
\label{lem:QuartPeriod}
If an orthodiagonal quadrilateral is not an (anti)deltoid, then it is of elliptic type, its configuration space has the $\sn$-parametrization (case~1 of Theorem \ref{thm:ParamSnCn}), and the phase shift $t_0$ is a quarter-period: $\Im t_0 = \frac{K'}2$. Conversely, every quadrilateral with these properties is orthodiagonal.
\end{lem}
\begin{proof}
By the remark preceding the theorem, if $Q$ is not an (anti)deltoid, then it is either conic or elliptic. Thus its configuration space has one of the parametrizations from Theorems \ref{thm:ParamCon} and \ref{thm:ParamSnCn}, and hence by Lemma \ref{lem:BranchInvol} carries the involutions $i$ and $j$.

A simple geometric argument shows that $Q$ is orthodiagonal if and only if $i$ and $j$ commute or, equivalently, if and only if $(i \circ j)^2 = \id$. On the other hand, since
$$
i \circ j(t) = t+2t_0,
$$
the involutions $i$ and $j$ commute if and only if $t_0$ is a quarter-period. In the conic case, $t_0$ cannot be a quarter-period since $\Im t_0 > 0$. Neither can it in the $\cn$ case, since there we have $\frac14\Lambda = K\Z + \frac{K+iK'}2\Z$, which has an empty intersection with $K\Z + (0,iK') \ni t_0$. Thus the configuration space is parametrized by $\sn$, and the lemma is proved.
\end{proof}


If in an orthodiagonal quadrilateral we have $\delta = \frac{\pi}2$, then due to \eqref{eqn:OrthodiagCos} we must have either $\alpha = \frac{\pi}2$ or $\gamma = \frac{\pi}2$.

\begin{lem}
The configuration space of an elliptic orthodiagonal quadrilateral is described by one of the following equations.
\begin{enumerate}
\item
If $\delta \ne \frac{\pi}2$, then
\begin{equation*}
\left( \sin(\delta-\alpha) z + \frac{\sin(\delta+\alpha)}{z} \right)
\left( \sin(\delta-\gamma) w + \frac{\sin(\delta+\gamma)}{w} \right) 
= 4 \sin\alpha\sin\gamma\cos\delta,
\end{equation*}
\item
If $\alpha = \delta = \frac{\pi}2$, then
\begin{equation*}
\left( (\cos\beta - \cos\gamma)z + \frac{\cos\beta + \cos\gamma}{z} \right) \left(w + \frac1w\right) = 4\sin\gamma
\end{equation*}
\item
If $\gamma = \delta = \frac{\pi}2$, then
$$
\left(z + \frac1z\right)\left( (\cos\beta - \cos\alpha) w + \frac{\cos\beta + \cos\alpha}{w}\right) = 4\sin\alpha
$$
\end{enumerate}

\end{lem}
\begin{proof}
Equation $\cos\alpha \cos\gamma = \cos\beta \cos\delta$ can be shown to imply
$$
\begin{aligned}
\sin(\delta - \alpha)\sin(\delta - \gamma) &= 2\cos\delta \sin\overline{\delta} \sin(\sigma - \beta - \delta)\\
\sin(\delta - \alpha)\sin(\delta + \gamma) &= 2\cos\delta \sin\overline{\alpha} \sin(\sigma - \beta - \alpha)\\
\sin(\delta + \alpha)\sin(\delta - \gamma) &= 2\cos\delta \sin\overline{\gamma} \sin(\sigma - \beta - \gamma)\\
\sin(\delta + \alpha)\sin(\delta + \gamma) &= 2\cos\delta \sin\overline{\beta} \sin\sigma
\end{aligned}
$$
(the proof can start with $\sin\overline{\delta} \sin(\sigma - \beta - \delta) = \frac12 (\cos\beta - \cos(\alpha+\gamma-\delta))$).
This shows that the equation in the first part of the lemma is equivalent to \eqref{eqn:AdjZ}. The second and the third part are proved by similar transformations.
\end{proof}

\begin{rem}
With the help of the identities from Lemma \ref{lem:BarRelations} one can show that
$$
\det
\begin{pmatrix}
c_{22} & c_{20}\\
c_{02} & c_{00}
\end{pmatrix}
= \sin\alpha \sin\gamma (\sin\overline{\beta} \sin\overline{\delta} - \sin\overline{\alpha} \sin\overline{\gamma}),
$$
where $c_{ij}$ are the coefficients from \eqref{eqn:AdjZ}. By Lemma \ref{lem:OrthodiagCosSin}, the right hand side vanishes if and only if the quadrilateral is orthodiagonal. It follows that equation \eqref{eqn:AdjZ} takes the form $(az^2 + b)(cw^2 + d) = zw$ if and only if the quadrilateral is orthodiagonal. This fact is also expressed and proved in a different way in Lemma \ref{lem:DescendInvol} below.
\end{rem}

The involution factors and other parameters introduced in Definition \ref{dfn:InvFactors} allow to abbreviate the equation of the configuration space of an orthodiagonal quadrilateral in the following way.

\begin{cor}
\label{cor:ConfOrthInv}
The configuration space of an orthodiagonal quadrilateral has the equation
\begin{subequations}
\label{eqn:ConfOrthodiag}
\begin{equation}
\label{eqn:LambdaMuNu}
(z + \lambda z^{-1})(w + \mu w^{-1}) = \nu, \quad\text{ if } Q \text{ is not an (anti)deltoid}
\end{equation}
\begin{equation}
z + \lambda z^{-1} = \xi w^n, \quad\text{ if } Q \text{ is an (anti)deltoid with apex }\alpha\beta
\end{equation}
\begin{equation}
w + \mu w^{-1} = \zeta z^m, \quad\text{if } Q \text{ is an (anti)deltoid with apex }\alpha\delta
\end{equation}
\end{subequations}
Here $m = 1$, respectively $n = 1$, if $Q$ is a deltoid, and $m = -1$, respectively $n = -1$, if $Q$ is antideltoid; $\lambda$, $\mu$, $\nu$, $\xi$, $\zeta$ are as in Definition \ref{dfn:InvFactors}.
\end{cor}

The involution factor $\lambda$ is defined if and only if the projection $g \colon (z,w) \mapsto w$ restricted to the non-trivial component $Z^{0}$ is a two-fold branched cover of $\CP^1$. Recall that $j \colon (z,w) \mapsto (z',w)$ denotes the deck transformation of $g$. In general, $z'$ depends both on $z$ and $w$. But equations \eqref{eqn:ConfOrthodiag} show that in the orthodiagonal case $z'$ depends only on $z$. In other words, the involution $j$ descends to an involution $f_*(j) \colon \CP^1 \to \CP^1$.

The next lemma explains the term ``involution factor'' and shows that the pushout $f_*(j)$ exists only in the orthodiagonal case.

\begin{lem}
\label{lem:DescendInvol}
The involution $j \colon (z,w) \mapsto (z',w)$ on the configuration space of an orthodiagonal quadrilateral acts by $z' = \lambda z^{-1}$, where $\lambda$ is as in Definition \ref{dfn:InvFactors}.

Conversely, if the involution $j$ descends to an involution on $\CP^1$, then $Q$ is an orthodiagonal quadrilateral.
\end{lem}
\begin{proof}
The first part is immediate from equations \eqref{eqn:ConfOrthodiag}.

If $f$ is two-fold, then the map $f_*(j) \colon z \mapsto z'$ is well-defined if and only if $f \circ j = f \circ j \circ i$. The last equation implies that either $j = j \circ i$ or $i \circ j = j \circ i$ holds. As $j \ne j \circ i$, we conclude that $i$ and $j$ must commute. This implies that $t_0$ is a quarter-period, and thus by Lemma \ref{lem:QuartPeriod} the quadrilateral is orthodiagonal.
\end{proof}

\begin{lem}
\label{lem:InvFactAmpl}
Coefficients $\lambda$, $\mu$, $\nu$ in the equation \eqref{eqn:LambdaMuNu} are expressed in terms of the modulus and amplitudes as follows:
$$
(\lambda, \mu, \nu) =
\begin{cases}
(\frac{p^2}{k}, \frac{q^2}k, \frac{2(1+k)}{k\sqrt{k}}pq), &\text{if } t_0 = \frac{iK'}2\\
(\frac{p^2}{k}, \frac{q^2}k, -\frac{2(1+k)}{k\sqrt{k}}pq), &\text{if } t_0 = 2K + \frac{iK'}2\\
(-\frac{p^2}{k}, -\frac{q^2}k, \frac{2i(1-k)}{k\sqrt{k}}pq), &\text{if } t_0 = K + \frac{iK'}2\\
(-\frac{p^2}{k}, -\frac{q^2}k, -\frac{2i(1-k)}{k\sqrt{k}}pq), &\text{if } t_0 = 3K + \frac{iK'}2
\end{cases}
$$
\end{lem}
\begin{proof}
The formulas are obtained by substituting $z = px$ and $w = qy$ in the equations
$$
\left(x + \frac1{kx}\right)\left(y + \frac1{ky}\right) = \pm \frac{2(1+k)}{k\sqrt{k}}, \quad \left(x - \frac1{kx}\right)\left(y - \frac1{ky}\right) = \pm \frac{2i(1-k)}{k\sqrt{k}}
$$
describing the curves $x = \sn t$, $y = \sn(t+t_0)$ for the values of $t_0$ given above.

\end{proof}



\begin{dfn}
\label{dfn:Compatible}
Orthodiagonal quadrilaterals $Q_1$ and $Q_2$ are called \emph{compatible} if one of the following holds:
\begin{itemize}
\item
the involution factors of $Q_1$ and $Q_2$ at their common vertex are equal: $\lambda_1 = \lambda_2$;
\item
$Q_1$ and $Q_2$ are frontally coupled deltoids;
\item
$Q_1$ and $Q_2$ are frontally coupled antideltoids.
\end{itemize}
\end{dfn}

Geometrically, a coupling of compatible orthodiagonal quadrilaterals is characterized by the property that during a deformation the angles marked on Figure \ref{fig:CompatCoupling} remain equal.

\begin{figure}[htb]
\includegraphics{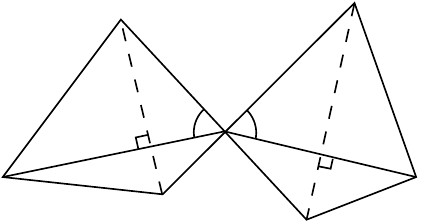}
\caption{A compatible coupling of orthodiagonal quadrilaterals.}
\label{fig:CompatCoupling}
\end{figure}

The coupling $(Q_1,Q_2)$ on Figure \ref{fig:Scissors} is compatible orthodiagonal if and only if during a deformation of the polyhedron on Figure \ref{fig:NotPlanAngles} the points $A_1$, $A_2$, $C_1$, $C_2$ remain in one plane, and this plane is always orthogonal to the planes $B_1A_1A_4$ and $B_2A_2A_3$.

\section{Configuration space of two coupled four-bar linkages}
\label{sec:Couplings}
In Section \ref{sec:ConfSpace} we studied the complexified configuration space $Z$ of a spherical four-bar linkage on Figure \ref{fig:FourbarZW}, together with the projections $f \colon Z \to \CP^1$ and $g \colon Z \to \CP^1$ that output the tangents of half the angles $\phi$ and $\psi$. In this Section we study the configuration space of two four-bar linkages coupled as shown on Figure \ref{fig:CoupledQuad}.

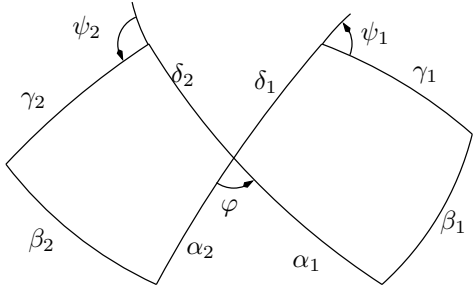
\begin{figure}[htb]
\begin{picture}(0,0)%
\includegraphics{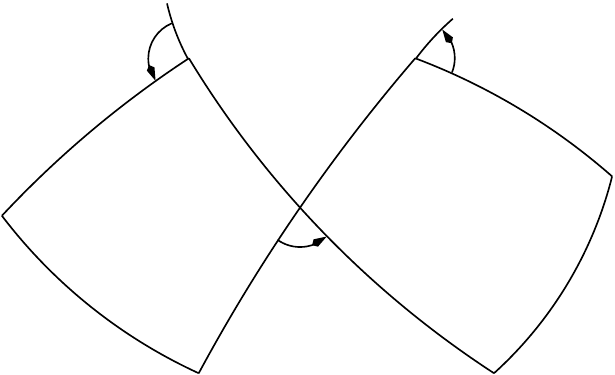}%
\end{picture}%
\setlength{\unitlength}{4144sp}%
\begingroup\makeatletter\ifx\SetFigFont\undefined%
\gdef\SetFigFont#1#2#3#4#5{%
  \reset@font\fontsize{#1}{#2pt}%
  \fontfamily{#3}\fontseries{#4}\fontshape{#5}%
  \selectfont}%
\fi\endgroup%
\begin{picture}(2806,1707)(-727,-1239)
\put(1879,-897){\makebox(0,0)[lb]{\smash{{\SetFigFont{10}{12.0}{\rmdefault}{\mddefault}{\updefault}{\color[rgb]{0,0,0}$\beta_1$}%
}}}}
\put(1716,  5){\makebox(0,0)[lb]{\smash{{\SetFigFont{10}{12.0}{\rmdefault}{\mddefault}{\updefault}{\color[rgb]{0,0,0}$\gamma_1$}%
}}}}
\put(997,-1125){\makebox(0,0)[lb]{\smash{{\SetFigFont{10}{12.0}{\rmdefault}{\mddefault}{\updefault}{\color[rgb]{0,0,0}$\alpha_1$}%
}}}}
\put(361,-1051){\makebox(0,0)[lb]{\smash{{\SetFigFont{10}{12.0}{\rmdefault}{\mddefault}{\updefault}{\color[rgb]{0,0,0}$\alpha_2$}%
}}}}
\put(-629,-151){\makebox(0,0)[lb]{\smash{{\SetFigFont{10}{12.0}{\rmdefault}{\mddefault}{\updefault}{\color[rgb]{0,0,0}$\gamma_2$}%
}}}}
\put(271,-16){\makebox(0,0)[lb]{\smash{{\SetFigFont{10}{12.0}{\rmdefault}{\mddefault}{\updefault}{\color[rgb]{0,0,0}$\delta_2$}%
}}}}
\put(766,-61){\makebox(0,0)[lb]{\smash{{\SetFigFont{10}{12.0}{\rmdefault}{\mddefault}{\updefault}{\color[rgb]{0,0,0}$\delta_1$}%
}}}}
\put(568,-792){\makebox(0,0)[lb]{\smash{{\SetFigFont{10}{12.0}{\rmdefault}{\mddefault}{\updefault}{\color[rgb]{0,0,0}$\phi$}%
}}}}
\put(1405,233){\makebox(0,0)[lb]{\smash{{\SetFigFont{10}{12.0}{\rmdefault}{\mddefault}{\updefault}{\color[rgb]{0,0,0}$\psi_1$}%
}}}}
\put(-321,267){\makebox(0,0)[lb]{\smash{{\SetFigFont{10}{12.0}{\rmdefault}{\mddefault}{\updefault}{\color[rgb]{0,0,0}$\psi_2$}%
}}}}
\put(-584,-1006){\makebox(0,0)[lb]{\smash{{\SetFigFont{10}{12.0}{\rmdefault}{\mddefault}{\updefault}{\color[rgb]{0,0,0}$\beta_2$}%
}}}}
\end{picture}%
\caption{Coupled four-bar linkages.}
\label{fig:CoupledQuad}
\end{figure}



\subsection{Coupled four-bar linkages and fiber products of branched coverings}
\label{sec:Pullback}
The configuration space $Z_{12}$ of the coupling on Figure \ref{fig:CoupledQuad} is the solution set of a system of equations:
$$
Z_{12} := \{(w_1, z, w_2) \in (\CP^1)^3 \mid P_1(z, w_1) = 0, P_2(z, w_2) = 0\}
$$
Here $P_1$ and $P_2$ are polynomials whose zero sets are the configuration spaces $Z_1$ and $Z_2$. Thus we have a commutative diagram
\begin{equation}
\label{eqn:Z12CD}
\xymatrix {
\CP^1 && Z_{12} \ar[dl]_{\widetilde{f_1}} \ar[dr]^{\widetilde{f_2}} && \CP^1 \\
& Z_2 \ar[ul]^{g_2} \ar[dr]_{f_2} && Z_1 \ar[ur]_{g_1} \ar[dl]^{f_1} \\
&& \CP^1 &&
}
\end{equation}
where each of the maps is the restriction of a projection $(\CP^1)^3 \to (\CP^1)^2$ or $(\CP^1)^2 \to \CP^1$.
It is easy to see that the commutative square in \eqref{eqn:Z12CD} is a fiber product diagram:
$$
Z_{12} = Z_1 \times_{\CP^1} Z_2 := \{(x_1, x_2) \in Z_1 \times Z_2 \mid f_1(x_1) = f_2(x_2)\}
$$

The set $Z_{12}$ is an algebraic set and can have several irreducible components. In particular, this is the case when $Z_1$ or $Z_2$ is reducible. We are not interested in trivial components of $Z_i$ given by $z_i = \const$ or $w = \const$, which will be dealt with in Section \ref{sec:TrivClass}. Therefore let us restrict the diagram \eqref{eqn:Z12CD} to some non-trivial components $Z_1^0$ and $Z_2^0$ of $Z_1$ and $Z_2$.

Each of the maps $f_i \colon Z_i^0 \to \CP^1$ is either a homeomorphism or a two-fold branched cover. If $f_1$ is a homeomorphism, then so is $\widetilde{f_1}$, and the branched cover $\widetilde{f_2}$ is equivalent to $f_2 \colon Z_2^0 \to \CP^1$.

The following lemma describes the fiber product $Z_1^0 \times_{\CP^1} Z_2^0$ in the case when both maps $f_1|_{Z_1^0}$ and $f_2|_{Z_2^0}$ are two-fold branched covers. This is a simple special case of the general description of the components of a fibre product over $\CP^1$, see \cite[Section 2]{Pak09}.

\begin{lem}
\label{lem:FibProd}
Let $f_1 \colon Z_1^0 \to \CP^1$ and $f_2 \colon Z_2^0 \to \CP^1$ be two-fold branched covers with branch sets $A_1, A_2 \subset \CP^1$, respectively. Then their fiber product $Z_{12}$ has one of the following forms.
\begin{enumerate}
\item 
If $A_1 \ne A_2$, then both $\widetilde{f_1}$ and $\widetilde{f_2}$ are two-fold branched covers with branch sets $f_1^{-1}(A_2)$ and $f_2^{-1}(A_1)$, respectively. The composition $f_1 \circ \widetilde{f_2} = f_2 \circ \widetilde{f_1}$ is four-fold, and its monodromy has the form $(12)(34)$ around $A_1 \setminus A_2$, $(13)(24)$ around $A_1 \cap A_2$, $(14)(23)$ around $A_2 \setminus A_1$. Here $1, 2, 3, 4$ is an appropriate labeling of the preimage of some point.
\item
If $A_1 = A_2$ (so that the covers $f_1$ and $f_2$ are equivalent), then $Z_{12}$ is equivalent to the union of two copies of $Z_1^0$:
$$
Z_{12} = Z_1^{0,+} \sqcup Z_1^{0,-}
$$
The restrictions of $\widetilde{f_1}$ and $\widetilde{f_2}$ to $Z_1^{0,+}$ and $Z_1^{0,-}$ are homeomorphisms, and the compositions
$$
Z_1^0 \stackrel{\widetilde{f_2}^{-1}}{\longrightarrow} Z_1^{0,+} \stackrel{\widetilde{f_1}}{\longrightarrow} Z_2^0 \quad \text{and} \quad Z_1^0 \stackrel{\widetilde{f_2}^{-1}}{\longrightarrow} Z_1^{0,-} \stackrel{\widetilde{f_1}}{\longrightarrow} Z_2^0
$$
differ by postcomposition with the deck transformation of $Z_2^0$ (equivalently, by precomposition with the deck transformation of $Z_1^0$).
\end{enumerate}
\end{lem}

\subsection{Reducible couplings}
\label{sec:ReduceCoupled}
\begin{dfn}
A coupling of spherical four-bar linkages on Figure \ref{fig:CoupledQuad} is called \emph{reducible}, if there exist non-trivial irreducible components $Z_1^0$ and $Z_2^0$ of $Z_1$ and $Z_2$, respectively, such that the fiber product $Z_1^0 \times_{\CP^1} Z_2^0$ is reducible.
\end{dfn}
By Lemma \ref{lem:FibProd}, a coupling is reducible if and only if the maps $f_1|_{Z_1^0}$ and $f_2|_{Z_2^0}$ are two-fold branched covers and their sets of branch points coincide: $A_1 = A_2$. Each of the components $Z_i^0$ can have one of the following forms, see Section \ref{sec:ConfSpace}:
\begin{itemize}
\item
Equation \eqref{eqn:Grashof2} has no solutions; then
$Z_i^0 = Z_i$ is an elliptic curve with the parametrization
\begin{equation}
\label{eqn:EllParam}
Z_i = \{(z,w_i) \mid z = p_iF_i(t), w_i = q_iF_i(t+t_i)\},
\end{equation}
where $F_i(t) = \sn(t,k_i)$ or $F_i(t) = \cn(t,k_i)$.
\item
Equation \eqref{eqn:Grashof2} has exactly one solution; then
$Z_i^0 = Z_i$ is a conic with the parametrization
\begin{equation}
\label{eqn:ConParam}
Z_i = \{(z,w_i) \mid z^{m_i} = p_i \sin t, w_i^{n_i} = q_i \sin(t+t_i)\},
\end{equation}
where $m_i, n_i = \pm 1$ depending on the form of the solution of \eqref{eqn:Grashof2}, as described in Theorem \ref{thm:ParamCon}.
\item
The quadrilateral $Q_i$ is either deltoid with $\alpha_i = \delta_i, \beta_i = \gamma_i$ or antideltoid with $\alpha_i + \delta_i = \pi = \beta_i + \gamma_i$, and
\begin{equation}
\label{eqn:DegConParam}
Z_i^0 = \{(z,w_i) \mid z_i^{m_i} = p_i\sin t, w_i = \epsilon_i \sqrt{-\mu_i} e^{it},\},
\end{equation}
where $m_i = 1$ if $Q_i$ is a deltoid, and $m_i = -1$ if $Q_i$ is anti-deltoid, and with $\epsilon_i$ and $\mu_i$ as in Theorem \ref{thm:ParamDegenCon}.
\end{itemize}

We have $|A_i| = 4$ when $Q_i$ is elliptic, and $|A_i|=2$ when $Q_i$ is a conic quadrilateral or an (anti)deltoid. Therefore a reducible coupling is only possible between two elliptics, or between two conics, or between (anti)deltoid and conic, or between two (anti)deltoids.


\subsubsection{Reducible couplings of elliptic quadrilaterals}
\label{sec:RedEll}
\begin{lem}
\label{lem:RedCoupEll}
A coupling of two elliptic quadrilaterals $Q_1$ and $Q_2$ is reducible if and only if the parametrizations \eqref{eqn:EllParam} of their configuration spaces satisfy one of the following sets of conditions:
\begin{enumerate}
\item
\label{it:EllPlain}
$F_1$ and $F_2$ is the same elliptic function ($\sn$ or $\cn$), and
$$
k_1 = k_2, \quad p_1 = p_2 =: p
$$
In terms of the side lengths of $Q_1$ and $Q_2$ this is equivalent to
\begin{equation}
\label{eqn:EqAmpEll}
\frac{\sin\alpha_1 \sin\delta_1}{\sin\overline{\alpha}_1 \sin\overline{\delta}_1} = \frac{\sin\alpha_2 \sin\delta_2}{\sin\overline{\alpha}_2  \sin\overline{\delta}_2}, \quad
\frac{\sin\beta_1 \sin\gamma_1}{\sin\overline{\beta}_1  \sin\overline{\gamma}_1} = \frac{\sin\beta_2 \sin\gamma_2}{\sin\overline{\beta}_2 \sin\overline{\gamma}_2},
\end{equation}

The components of the configuration space $Z_{12}$ of the coupling have the following parametrizations:
\begin{equation}
\label{eqn:ParamEllPlain}
\begin{aligned}
\{w_1 = q_1 F(t-t_1), z = p F(t), w_2 = q_2 F(t + t_2) &\} \\
\cup \{w_1 = q_1 F(t-t_1), z = p F(t), w_2 = q_2 F(t - t_2) &\}
\end{aligned}
\end{equation}
\item
\label{it:EllSkew}
$F_1 = \cn(\,\cdot\,,k_1)$ and $F_2 = \cn(\,\cdot\,,k_2)$ and
$$
k_1^2 + k_2^2 = 1, \qquad \frac{p_1}{p_2} = \pm i \frac{k_1}{k_2}
$$
The components of the configuration space of the coupling have the following parametrizations:
\begin{equation}
\label{eqn:ParamEllSkew}
\begin{aligned}
\{w_1 = q_1 \cn(t-t_1, k_1), z = p_1 \cn (t, k_1), w_2 = \frac{p_1q_2}{p_2} \cn(t + it_2, k_1) &\} \\
\cup \{w_1 = q_1 \cn(t-t_1, k_1), z = p_1 \cn (t, k_1), w_2 = \frac{p_1q_2}{p_2} \cn(t - it_2, k_1) &\}
\end{aligned}
\end{equation}
\end{enumerate}
\end{lem}

\begin{proof}
Branch points are given by part 4) of Lemma \ref{lem:BranchInvol}. Note that in the $\sn$ case either all four points are real or all purely imaginary, while in the $\cn$ case there are two real and two imaginary points. Thus we have two possibilities:
$$
\begin{aligned}
\{\pm p_1, \pm \frac{p_1}{k_1}\} &= \{\pm p_2, \pm \frac{p_2}{k_2}\} &\text{ if } F_i = \sn(\,\cdot\,,k_i)\\
\{\pm p_1, \pm ip_1\frac{k'_1}{k_1}\} &= \{\pm p_2, \pm ip_2 \frac{k'_2}{k_2}\} &\text{ if } F_i = \cn(\,\cdot\,, k_i)
\end{aligned}
$$
If $p_1 = p_2$, then in both cases we have $k_1 = k_2$, which results in the case \ref{it:EllPlain} of the Lemma. The parametrization \eqref{eqn:ParamEllPlain} follows from parametrizations \eqref{eqn:EllParam} and from the action of the involutions $i_1$ and $i_2$ on $Z_1$ and $Z_2$.

Recall that $p_i \in \R_+ \cup i\R_+$, so that $p_1 = -p_2$ is not possible. It remains to check if the elements of $A_1$ can be equal to those of $A_2$ ``crosswise''.

In the $\sn$-case because of $p_i \in \R_+ \cup i\R_+$ we have only the possibility
$$
p_1 = \frac{p_2}{k_2}, \quad p_2 = \frac{p_1}{k_1}
$$
which cannot occur because of $0<k_i<1$. In the $\cn$-case we have
$$
p_1 = \pm ip_2 \frac{k'_2}{k_2} \text{ and } p_2 = \pm ip_1 \frac{k'_1}{k_1} \quad \Leftrightarrow \quad k'_1 = k_2 \text{ and } \frac{p_1}{p_2} = \pm i\frac{k_1}{k_2}
$$
which leads to the case \ref{it:EllSkew} of the Lemma.

To obtain a parametrization of $Z_{12}$ in the second case, use Jacobi's imaginary transformation:
$$
\cn(t,k') = \frac1{\cn(it,k)} = i \frac{k}{k'} \cn(it + K + iK', k)
$$
It leads to the following parametrization of $Z_2$:
$$
Z_2 = \{ (z,w_2) \mid z = ip_2\frac{k_1}{k'_1} \cn t, w_2 = iq_2 \frac{k_1}{k'_1} \cn(t+it_2) \}
$$
By making, if needed, the parameter change $t \to t+2K$, this can be rewritten as $z = p_1\cn t$, $w_2 = \frac{p_1q_2}{p_2} \cn(t+it_2)$, and we obtain \eqref{eqn:ParamEllSkew}.
\end{proof}

\subsubsection{Reducible couplings of conic quadrilaterals}
\begin{lem}
\label{lem:RedCoupCon}
A coupling of two conic quadrilaterals $Q_1$ and $Q_2$ is reducible if and only if the parametizations \eqref{eqn:ConParam} of their configuration spaces satisfy one of the following sets of conditions:
\begin{enumerate}
\item
\label{it:ConPlain}
$m_1 = m_2 =: m$ and $p_1 = p_2 =: p$. In terms of the side lengths, this is equivalent to
\begin{equation}
\label{eqn:RedCondSides}
\frac{\sin\alpha_1\sin\delta_1}{\sin\beta_1\sin\gamma_1} = \frac{\sin\alpha_2\sin\delta_2}{\sin\beta_2\sin\gamma_2}
\end{equation}
with an additional condition that each of $(\alpha_i, \beta_i, \gamma_i, \delta_i)$ satisfies one of the equations
$$
\alpha_i + \gamma_i = \beta_i + \delta_i, \quad \alpha_i + \beta_i = \gamma_i + \delta_i \quad (\Leftrightarrow m_i = 1)
$$
or each satisfies one of the equations
$$
\alpha_i + \delta_i = \beta_i + \gamma_i, \quad \alpha_i + \beta_i + \gamma_i + \delta_i = 2\pi \quad (\Leftrightarrow m_i = -1)
$$
The components of the configuration space of the coupling have the following parametrizations:
\begin{equation}
\label{eqn:ParamConPlain}
\begin{aligned}
\{w_1^{n_1} = q_1 \sin(t-t_1), z^m = p \sin t, w_2^{n_2} = q_2 \sin(t + t_2) &\} \\
\cup \{w_1^{n_1} = q_1 \sin(t-t_1), z^m = p \sin t, w_2^{n_2} = q_2 \sin(t - t_2) &\}
\end{aligned}
\end{equation}
\item
\label{it:ConSkew}
$m_1 = - m_2$ and $p_1 = \pm\frac{1}{p_2}$.
The components of the configuration space of the coupling have the following parametrizations:
\begin{equation}
\label{eqn:ParamConSkew}
\begin{aligned}
\{w_1^{n_1} = q_1 \sin(t-t_1), z^{m_1} = p_1 \sin t, w_2^{n_2} = q_2 \frac{\cos t_2 + i \sin t_2 \cos t}{\sin t} &\} \\
\cup \{w_1^{n_1} = q_1 \sin(t-t_1), z^{m_1} = p_1 \sin t, w_2^{n_2} = q_2 \frac{\cos t_2 - i \sin t_2 \cos t}{\sin t} &\}
\end{aligned}
\end{equation}
\end{enumerate}
\end{lem}

\begin{proof}
Branch points are given by part 3) of Lemma \ref{lem:BranchInvol}. Thus $A_1 = A_2$ is equivalent to
$$
\begin{aligned}
\{\pm p_1\} &= \{\pm p_2\} &&\text{ if } m_1 = m_2\\
\{\pm p_1\} &= \left\{\pm \frac1{p_2}\right\} &&\text{ if } m_1 = -m_2
\end{aligned}
$$
The parametrization of $Z_{12}$ in the first case is obvious. In the second case we must set
$$
p_1 \sin t = z^{m_1} = \frac{1}{p_2 \sin t'} \Leftrightarrow \sin t = \frac1{\sin t'},
$$
where $t$ and $t'$ are parameters on $Z_1$ and $Z_2$, respectively. (If we have $p_1 = -\frac1{p_2}$, which happens when the amplitudes are imaginary, then make the parameter change $t \mapsto t + \pi$.) This determines two different automoprhisms of $\CP^1 = (\C / 2\pi \Z) \cup \{i\infty\}$. Since $\sin t' = \frac1{\sin t}$ implies $\cos t' = \pm i \cot t$, we have
$$
\sin(t' \pm t_2) = \frac{\cos t_2 \pm i \sin t_2 \cos t}{\sin t},
$$
which leads to the parametrization in the Lemma.
\end{proof}

\subsubsection{Reducible couplings involving deltoids}
An (anti)deltoid $Q_1$ is said to be \emph{frontally} coupled with $Q_2$ if the common vertex of $Q_1$ and $Q_2$ is an apex of $Q_1$. Otherwise, $Q_1$ is said to be \emph{laterally} coupled with $Q_2$. The map $f_1|_{Z_1^0}$ in the diagram \eqref{eqn:Z12CD} is two-fold if and only if $Q_1$ is coupled frontally.

\begin{lem}
\label{lem:RedCoupDelt}
A coupling $(Q_1,Q_2)$ with $Q_1$ a frontally coupled (anti)deltoid is reducible if and only if one of the following sets of conditions is satisfied.
\begin{enumerate}
\item
\label{it:ConDeltPlain}
The quadrilateral $Q_2$ is conic, $m_1 = m_2 =: m$ and $p_1 = p_2 =: p$. In terms of the side lengths, this is equivalent to
$$
\frac{\sin^2\alpha_1}{\sin^2\beta_1} = \frac{\sin\alpha_2\sin\delta_2}{\sin\beta_2\sin\gamma_2}
$$
with the additional condition:
$$
\begin{aligned}
Q_1 \text{ deltoid } &\Rightarrow \alpha_2 + \gamma_2 = \beta_2 + \delta_2 \text{ or } \alpha_2 + \beta_2 = \gamma_2 + \delta_2\\
Q_1 \text{ antideltoid } &\Rightarrow \alpha_2 + \delta_2 = \beta_2 + \gamma_2 \text{ or } \alpha_2 + \beta_2 + \gamma_2 + \delta_2 = 2\pi
\end{aligned}
$$
The components of $Z_{12}$ can in this case be parametrized as
\begin{equation}
\label{eqn:ParamConDeltPlain}
\begin{aligned}
\{w_1 = \epsilon_1 \sqrt{-\mu_1} e^{it}, z^m = p \sin t, w_2^{n_2} = q_2 \sin(t + t_2) &\} \\
\cup \{w_1 = \epsilon_1 \sqrt{-\mu_1} e^{it}, z^m = p \sin t, w_2^{n_2} = q_2 \sin(t - t_2) &\}
\end{aligned}
\end{equation}

\item
\label{it:ConDeltSkew}
The quadrilateral $Q_2$ is conic, $m_1 = -m_2$ and $p_1 = \pm \frac{1}{p_2}$.

The components of the configuration space of the coupling have the following parametrizations:
\begin{equation}
\label{eqn:ParamConDeltSkew}
\begin{aligned}
\{w_1 = \epsilon_1 \sqrt{-\mu_1} e^{it}, z^{m_1} = p_1 \sin t, w_2^{n_2} = q_2 \frac{\cos t_2 + i \sin t_2 \cos t}{\sin t} &\} \\
\cup \{w_1 = \epsilon_1 \sqrt{-\mu_1} e^{it}, z^m = p_1 \sin t, w_2^{n_2} = q_2 \frac{\cos t_2 - i \sin t_2 \cos t}{\sin t} &\}
\end{aligned}
\end{equation}

\item
\label{it:DeltDelt}
Both $Q_1$ and $Q_2$ are frontally coupled deltoids or frontally coupled antideltoids (so that $m_1 = m_2 =: m$) such that $p_1 = p_2 =: p$ in \eqref{eqn:DegConParam}. In terms of side lengths:
$$
\frac{\sin\alpha_1}{\sin\beta_1} = \frac{\sin\alpha_2}{\sin\beta_2} \quad \text{and} \quad
\begin{aligned}
&\text{either } \alpha_i = \delta_i, \beta_i = \gamma_i\\
&\text{or } \alpha_i + \delta_i = \pi = \beta_i + \gamma_i
\end{aligned}
\quad \text{for }i = 1,2
$$

The components of the configuration space of the coupling have the following parametrizations.
$$
\begin{aligned}
&\{w_1 = \epsilon_1 \sqrt{-\mu_1} e^{it}, z^{m} = p \sin t, w_2 = \epsilon_2 \sqrt{-\mu_2} e^{it} \} \\
\cup &\{w_1 = \epsilon_1 \sqrt{-\mu_1} e^{it}, z^{m} = p \sin t, w_2 = -\epsilon_2 \sqrt{-\mu_2} e^{-it} \}
\end{aligned}
$$

\item
\label{it:DeltAnti}
$Q_1$ is a deltoid, $Q_2$ is an antideltoid such that $p_1p_2 = 1$.

The components of the configuration space of the coupling have the following parametrizations.
$$
\begin{aligned}
&\{w_1 = \epsilon_1 \sqrt{-\mu_1} e^{it}, z^{m_1} = p_1 \sin t, w_2 = \epsilon_2 \sqrt{-\mu_2} \frac{e^{it} - 1}{e^{it} + 1} \} \\
\cup &\{w_1 = \epsilon_1 \sqrt{-\mu_1} e^{it}, z^{m_1} = p_1 \sin t, w_2 = -\epsilon_2 \sqrt{-\mu_2} \frac{e^{it} + 1}{e^{it} - 1} \}
\end{aligned}
$$
\end{enumerate}
\end{lem}

\begin{proof}
Cases \ref{it:ConDeltPlain}) and \ref{it:ConDeltSkew}) are treated similarly to Lemma \ref{lem:RedCoupCon}.

In the case \ref{it:DeltDelt}) we have $A_1 = A_2$ if and only if $p_1 = p_2$, see part 2) of Lemma \ref{lem:BranchInvol}.
Parametrizations of the two components of the fiber product is straightforward from \eqref{eqn:DegConParam} by taking into account the involution action $t \mapsto \pi - t$ (or, alternatively, its description in Lemma \ref{lem:DescendInvol}).

In the case \ref{it:DeltAnti}) we must adjoin to the parametrization \eqref{eqn:DegConParam} the equation $w_2 = \epsilon_2 \sqrt{-\mu_2} e^{it'}$, where $\sin t' \sin t = 1$. The latter equation is equivalent to
\begin{equation}
\label{eqn:ExpPrime}
(e^{i(t + t')} - e^{it} + e^{it'} + 1)(e^{i(t + t')} + e^{it} - e^{it'} + 1) = 0,
\end{equation}
which implies the formula in the lemma.
\end{proof}

%

\subsection{Fiber product and resultant}
The fiber product $Z_{12} = Z_1^0 \times_{\CP^1} Z_2^0$ lies in the space $(\CP^1)^3$ with coordinates $(w_1, z, w_2)$, and its projections to the $(z, w_1)$ and to the $(z, w_2)$ planes are the spaces $Z_1^0$ and $Z_2^0$, respectively. Consider now the projection to the $(w_1, w_2)$-plane and denote by $W$ its image:
\begin{equation}
\label{eqn:Pi}
\begin{aligned}
\pi \colon Z_{12} &\to W,\\
(w_1, z, w_2) &\mapsto (w_1, w_2),
\end{aligned}
\end{equation}
which we call the \emph{partial configuration space} of the coupling.
Then we have
$$
W := \{(w_1, w_2) \mid \exists z \in \CP^1 \text{ such that } (z, w_1) \in Z_1^0, (z, w_2) \in Z_2^0\}
$$
If $P_1^0$ and $P_2^0$ are polynomials defining the components $Z_1^0$ and $Z_2^0$, respectively, then the set $W$ (at least its affine part) is the zero set of the resultant of the polynomials $P_1^0$ and $P_2^0$ viewed as polynomials in $z$.


If the coupling is reducible, then the resultant is a reducible polynomial, so that $W$ consists of several irreducible components. In this case it is convenient to use the parametrizations of the irreducible components of $Z_{12}$ that we obtained in Section \ref{sec:ReduceCoupled} in order to obtain the descriptions of the components of $W$.

\begin{lem}
\label{lem:WRed}
The following are the equations of the irreducible components of the space $W$ for some of the reducible couplings.
\begin{enumerate}
\item
\label{it:WEll}
If $(Q_1, Q_2)$ is a reducible non-involutive coupling of elliptic quadrilaterals, then
$$
a_{22} w_1^2 w_2^2 + a_{20} w_1^2 + a_{02} w_2^2 + 2a_{11} w_1w_2 + a_{00} = 0
$$
\item\label{it:WConPlain}
If $(Q_1, Q_2)$ is a reducible coupling of conic quadrilaterals such that $m_1 = m_2$ and $p_1 = p_2$ (Case 1 of Lemma \ref{lem:RedCoupCon}), then
$$
a_{20} w_1^{2n_1} + a_{02} w_2^{2n_2} + 2 a_{11} w_1^{n_1} w_2^{n_2} + a_{00} = 0
$$
for some $a_{ij} \in \R$, except for the components corresponding to $t_1 - t_2 \equiv 0 (\mod \pi)$.
\item
\label{it:WConSkew}
If $(Q_1, Q_2)$ is a reducible coupling of conic quadrilaterals such that $m_1 = -m_2$ and $p_1 = \pm \frac1{p_2}$ (Case 2 of Lemma \ref{lem:RedCoupCon}), then
\begin{equation}
\label{eqn:WConSkew}
w_1^{2n_1}w_2^{2n_2} + a_{20} w_1^{2n_1} + a_{02} w_2^{2n_2} + 2 a_{11} w_1^{n_1} w_2^{n_2} + a_{10} w_1^{n_1} + a_{01} w_2^{n_2} + a_{00} = 0
\end{equation}
for some $a_{ij} \in \C$ with $a_{20}, a_{02}, a_{11}, a_{10}, a_{01} \ne 0$. The polynomial is irreducible.
\item
\label{it:WConDeltPlain}
If $Q_1$ is an (anti)deltoid, and $Q_2$ is a conic quadrilateral such that $m_1 = m_2$ and $p_1 = p_2$ (Case 1 of Lemma \ref{lem:RedCoupDelt}), then
$$
w_1 + \mu_1 e^{\pm 2it_2} w_1^{-1} = \epsilon_1 \frac{2i\sqrt{-\mu_1} e^{\pm it_2}}{q_2} w_2^{n_2}
$$
\item
\label{it:WConDeltSkew}
If $Q_1$ is an (anti)deltoid, and $Q_2$ is a conic quadrilateral such that $m_1 = \pm m_2$ and $p_1 = \pm \frac1{p_2}$ (Case 2 of Lemma \ref{lem:RedCoupDelt}), then
\begin{equation}
\label{eqn:WConDeltSkew}
\frac{w_1^2 + 2a w_1 + b}{w_1^2 - b} = cw_2^{n_2},
\end{equation}
where the fraction on the left hand side is irreducible.
\item
\label{it:WDeltAnti}
If $Q_1$ is a deltoid, and $Q_2$ is an antideltoid such that $p_1p_2 = 1$ (Case \ref{it:DeltAnti} of Lemma \ref{lem:RedCoupDelt}), then
$$
w_1w_2 \mp \epsilon_2 \sqrt{-\mu_2} w_1 \pm \epsilon_1 \sqrt{-\mu_1} w_2 + \epsilon_1 \epsilon_2 \sqrt{-\mu_1} \sqrt{-\mu_2} = 0
$$
\end{enumerate}
\end{lem}
\begin{proof}
In the case \ref{it:WEll}) each of the components of $W$ has the parametrization of the form
$$
w_1 = q_1 F(t), \quad w_2 = q'_2 F(t + t_{12}),
$$
where $t_{12}$ is either $t_1 \pm t_2$ or $t_1 \pm it_2$. For an irreducible component, $t_{12}$ is not a half-period of the elliptic function $F$, and the component is described by a biquadratic equation of the above form.

Similarly, in the case \ref{it:WConPlain}), the components of $W$ are parametrized by
$$
w_1 = q_1 \sin t, \quad w_2 = q_2 \sin(t + (t_1 \pm t_2))
$$
If $t_1 \pm t_2 \notin \{0,\pi\}$, then the corresponding component is described by a quadratic equation without a linear part.

In the case \ref{it:WConSkew}) by Lemma \ref{lem:RedCoupCon}, part \ref{it:ConSkew} we have a parametrization of the form
\[
w_1^{n_1} = a\sin t + b\cos t, \quad w_2^{n_2} = \frac{c + d\cos t}{\sin t}
\]
with $a, b, c, d \ne 0$. It suffices to consider the case $n_1 = n_2 = 1$. One easily computes
\[
\sin t = \frac{dw_1+ bc}{bw_2 + ad}, \quad \cos t = \frac{w_1w_2 - ac}{bw_2 + ad}
\]
By substituting this into $\sin^2 t + \cos^2 t = 1$ we obtain the equation
\[
w_1^2w_2^2 + d^2w_1^2 - b^2w_2^2 - 2acw_1w_2 + 2bcdw_1 - 2abdw_2 + (a^2c^2 + b^2c^2 - a^2d^2) = 0
\]
The polynomial at the left hand side is irreducible, since it describes a $2$-$2$ correspondence, and $W_1^0$ is also a double cover over both $Z_1$ and $Z_2$.

In the cases \ref{it:WConDeltPlain}) and \ref{it:WConDeltSkew}) the equations are found by substituting $w_1 = c_1 e^{it}$ into $\sin(t \pm t_2) = \frac{e^{i(t \pm t_2)} - e^{-i(t \mp t_2)}}{2i}$ and $\cos(t \pm t_2) = \frac{e^{i(t \pm t_2)} + e^{-i(t \mp t_2)}}{2}$ in the formulas of Lemma \ref{lem:RedCoupDelt}. The irreducibility of the rational function in the case \ref{it:WConDeltSkew}) is equivalent to the non-divisibility of the numerator by $e^{it} \pm 1$. This, in turn, is equivalent to $\tan t_2 \ne i$ that was observed in Section \ref{sec:ClassConfSp} after the proof of Theorem \ref{thm:ParamCon}.

Finally, equations in the case \ref{it:WDeltAnti}) follow from equation \eqref{eqn:ExpPrime}.
\end{proof}

\subsection{Involutive couplings}
\label{sec:InvCoupl}
Consider the projection \eqref{eqn:Pi}, where $Z_{12}$ is the fiber product of non-trivial components $Z_1^0$ and $Z_2^0$ of the configuration spaces $Z_1$ and $Z_2$. Let $Z_{12}^0$ be some irreducible component of $Z_{12}$. Then the map $\pi$ restricted to $Z_{12}^0$ is either an isomorphism or a two-fold (possibly branched) cover.

\begin{dfn}
\label{dfn:InvCoupl}
A coupling $(Q_1,Q_2)$ is called \emph{involutive}, if there exists an irreducible component $Z_{12}^0$ of $Z_1^0 \times_{\CP^1} Z_2^0$ such that the restriction of the projection \eqref{eqn:Pi} to $Z_{12}^0$ is two-fold.
\end{dfn}

\begin{lem}
The coupling $(Q_1,Q_2)$ is involutive if and only if the maps $g_1$ and $g_2$ on the diagram \eqref{eqn:Z12CD} (with $Z_i^0$ in place of $Z_i$) are two-fold, and the corresponding deck transformations $j_i \colon Z_i^0 \to Z_i^0$ lift to a common involution $j_{12} \colon Z_{12}^0 \to Z_{12}^0$.
\end{lem}
\begin{proof}
If $(Q_1,Q_2)$ is involutive, then the deck transformation of the two-fold cover $Z_{12}^0 \to W^0$ has the form
\begin{equation}
\label{eqn:J12}
(w_1, z, w_2) \mapsto (w_1, z', w_2)
\end{equation}
It follows that $j_1$ and $j_2$ are two-fold and that \eqref{eqn:J12} is a common lift of their deck transformations.

In the opposite direction, if $j_1$ and $j_2$ lift to an involution $j_{12}$, then we have $\pi \circ j_{12} = \pi$, and since $j_{12} \ne \id$, this means that $\pi$ is two-fold.
\end{proof}

\subsubsection{Classification of involutive couplings}
\begin{lem}
\label{lem:InvolLift}
Any involutive coupling $(Q_1,Q_2)$ has one of the following forms.
\begin{enumerate}
\item
\label{it:OrthoCompat}
The quadrilaterals $Q_1$ and $Q_2$ are orthodiagonal and have equal involution factors at their common vertex: $\lambda_1 = \lambda_2$, see Definition \ref{dfn:InvFactors}. (In particular this means that $\lambda_i$ are defined, that is if $Q_i$ is an (anti)deltoid, then it is coupled laterally.)

\item
\label{it:EllShifts}
The quadrilaterals $Q_1$ and $Q_2$ are elliptic and form a reducible elliptic coupling of the first type from Lemma \ref{lem:RedCoupEll}. Besides, $t_1 \pm t_2$ is a half-period of the corresponding elliptic function.

In terms of side lengths this means that, in addition to equations \eqref{eqn:EqAmpEll}, we have either
\begin{equation*}
\frac{\sin\alpha_1 \sin\gamma_1}{\sin\overline{\alpha}_1 \sin\overline{\gamma}_1} = \frac{\sin\alpha_2 \sin\gamma_2}{\sin\overline{\alpha}_2 \sin\overline{\gamma}_2},
\end{equation*}
which is equivalent to $t_1 - t_2 \in \{0, 2K\}$, or
\begin{equation*}
\frac{\sin\alpha_1 \sin\gamma_1}{\sin\overline{\alpha}_1 \sin\overline{\gamma}_1} = \frac{\sin\alpha_2 \sin\gamma_2}{\sin\overline{\beta}_2 \sin\overline{\delta}_2},
\end{equation*}
which is equivalent to $t_1 + t_2$ being a half-period with $\Im(t_1 + t_2) = K'$.

\item
\label{it:ConShifts}
The quadrilaterals $Q_1$ and $Q_2$ are conic and form a reducible conic coupling of the first type from Lemma \ref{lem:RedCoupCon}. Besides, $t_1 - t_2 \equiv 0 (\mod \pi)$.

In terms of the side lengths this is equivalent to
$$
\frac{\sin\alpha_1}{\sin\beta_1} = \frac{\sin\alpha_2}{\sin\beta_2}, \quad
\frac{\sin\gamma_1}{\sin\delta_1} = \frac{\sin\gamma_2}{\sin\delta_2}
$$
with an additional condition that each of $(\alpha_1, \beta_1, \gamma_1, \delta_1)$ and $(\alpha_2, \beta_2, \gamma_2, \delta_2)$ satisfies one of the equations
$$
\alpha + \gamma = \beta + \delta, \quad \alpha + \beta = \gamma + \delta \quad (\Leftrightarrow m = 1)
$$
or each of them satisfies one of the equations
$$
\alpha + \delta = \beta + \gamma, \quad \alpha + \beta + \gamma + \delta = 2\pi \quad (\Leftrightarrow m = -1)
$$
\end{enumerate}
\end{lem}


%

\begin{proof}
By assumption, both covers $g_1$ and $g_2$ are two-fold. Assume that the map $f_1$ is an isomorphism. Then, according to the classification of configuration spaces (in particular, Lemma \ref{lem:BranchInvol}), the quadrilateral $Q_1$ is an (anti)deltoid coupled to $Q_2$ frontally. Then, by Lemma \ref{lem:DescendInvol}, the deck transformation $j_1$ acts by
$$
z \mapsto \lambda_1 z^{-1}
$$
If $j_1$ and $j_2$ have a common lift \eqref{eqn:J12}, then $j_2$ must descend to an involution on $\CP^1$ given by the same formula. By Lemma \ref{lem:DescendInvol}, this happens if and only if the quadrilateral $Q_2$ is orthodiagonal and has the same involution factor at the $\delta\alpha$-vertex as $Q_1$.
Thus we arrive at a special case of the situation described in part \ref{it:OrthoCompat}) of the Lemma, with a laterally coupled (anti)deltoid as one of the orthodiagonal quadrilaterals.

From now on assume that both $f_1$ and $g_1$ are two-fold, that is each of $Z_1^0 = Z_1$ and $Z_2^0 = Z_2$ is either a conic or an elliptic curve. We will distinguish two cases: when the coupling $(Q_1,Q_2)$ is reducible and when not.

If $(Q_1,Q_2)$ is reducible, then consider case-by-case possible parametrizations of the components of $Z_{12}$ given in Lemmas \ref{lem:RedCoupEll} and \ref{lem:RedCoupCon}.
\begin{itemize}
\item $w_1 = q_1 F(t-t_1),\quad z = p F(t),\quad w_2 = q_2 F(t \pm t_2)$,
\end{itemize}
where $F = \sn$ or $\cn$, and the choice of $+t_2$ or $-t_2$ shift yields two different components of $Z_{12}$. The involution $j_1$ acts by
$$
j_1(t) =
\begin{cases}
2K + 2t_1 - t, &\text{if } F = \sn\\
2t_1 - t, &\text{if } F = \cn
\end{cases}
$$
In order for $j_1$ to preserve the value of $w_2$, we must have
$$
\sn(t \pm t_2) = \sn(2K + 2t_1 - t \pm t_2) \quad \text{or} \quad \cn(t \pm t_2) = \cn(2t_1 - t \pm t_2) \quad \forall t,
$$
respectively. This is equivalent to $t_1 \pm t_2$ being a half-period of $\sn$, respectively $\cn$. Hence we have the situation described in part \ref{it:EllShifts}) of the Lemma. Note that due to $\Im t_1, \Im t_2 \in (0, K')$ the possible half-periods are
\begin{equation}
\label{eqn:SumDiffShifts}
t_1 - t_2 \in \{0, 2K\} \quad \text{or} \quad t_1 + t_2 \in
\begin{cases}
\{iK', 2K + iK'\}, &\text{in the } \sn \text{ case}\\
\{K+iK', 3K+iK'\}, &\text{in the } \cn \text{ case}
\end{cases}
\end{equation}
The corresponding conditions on the side lengths follow from the formulas for $\dn t_0$ in Theorem \ref{thm:ParamSnCn} and from the identities $\dn(iK'- t_0) = i \frac{\cn t_0}{\sn t_0}$ and $\dn(K + iK' - t_0) = -ik' \frac{\sn t_0}{\cn t_0}$.

\begin{itemize}
\item $w_1 = q_1 \cn(t-t_1),\quad z = p_1 \cn t,\quad w_2 = \frac{p_1q_2}{p_2} \cn(t \pm it_2)$
\end{itemize}
Similarly to the previous case, $t_1 \pm it_2$ must be a half-period of $\cn$, that is belong to $2K\Z + (K+iK')\Z$. This is impossible, since $\Im t_1 \in (0,K')$ and $\Im (it_2)$ is a multiple of $K'$. Thus this coupling cannot be involutive.

\begin{itemize}
\item $w_1^{n_1} = q_1 \sin(t-t_1),\quad z^m = p \sin t,\quad w_2^{n_2} = q_2 \sin(t \pm t_2)$
\end{itemize}
Arguing as before, we see that $t_1 \pm t_2$ must be a multiple of $\pi$. Since $t_1$ and $t_2$ have positive imaginary parts, we have $t_1 - t_2 = n\pi$. By the formula for $\tan t_i$ from \ref{thm:ParamCon} this is equivalent to
$$
\frac{\sin\beta_1\sin\delta_1}{\sin\alpha_1\sin\gamma_1} = \frac{\sin\beta_2\sin\delta_2}{\sin\alpha_2\sin\gamma_2}
$$
Together with \eqref{eqn:RedCondSides}, this is equivalent to the condition on the side lengths in the Lemma.

\begin{itemize}
\item $w_1^{n_1} = q_1 \sin(t-t_1),\quad z^{m_1} = p_1 \sin t,\quad w_2^{n_2} = q_2 \frac{\cos t_2 + i \sin t_2 \cos t}{\sin t}$
\end{itemize}
The lift on involution $j_1$ acts by $t \mapsto \pi - 2t_1 - t$ which doesn't preserve the value of $w_2(t)$, so this coupling cannot be involutive.

If the coupling $(Q_1,Q_2)$ is irreducible, then we have
$$
Z_{12} = \{(t,t') \in Z_1 \times Z_2 \mid f_1(t) = f_2(t')\},
$$
where we identified $Z_1$ and $Z_2$ with their parameter domains. The common lift $j_{12}$ of the involutions $j_1$ and $j_2$ is the restriction to $Z_{12}$ of the map
$$
(j_1, j_2) \colon Z_1 \times Z_2 \to Z_1 \times Z_2
$$
On the other hand, we have another map of $Z_{12}$ to itself:
$$
(i_1, \id) \colon Z_{12} \to Z_{12},
$$
which changes $w_1$ while preserving $z$ and $w_2$. We have
$$
(i_1, \id) \circ (j_1, j_2) (t,t') = (t+2t_1, j_2(t')),
$$
hence
$$
((i_1, \id) \circ (j_2, j_2))^2 (t,t') = (t+4t_1, t')
$$
Since this maps $Z_{12}$ to itself, we have $f_1(t+4t_1) = f_1(t)$ for all $t$. That is, the phase shift $t_1$ is a quarter-period of the function parametrizing $Z_1$. By an argument from Lemma \ref{lem:QuartPeriod}, this implies that $Q_1$ is orthodiagonal. Similarly, by using $(\id, i_2)$ in place of $(i_1, \id)$, we show that $Q_2$ is orthodiagonal. But then, by Lemma \ref{lem:DescendInvol}, the involutions $j_1$ and $j_2$ descend to $\CP^1$. They descend to the same involution $z \mapsto z'$ if and only if $\lambda_1 = \lambda_2$.
Thus we are in the situation of the part~\ref{it:OrthoCompat}) of the Lemma.
\end{proof}

\subsubsection{The partial configuration space of an involutive coupling}

\begin{lem}
\label{lem:CompatOrthodiag}
Let $(Q_1,Q_2)$ be an involutive coupling of orthodiagonal quadrilaterals. Then the quotient space $W := Z_{12}/j_{12}$ is the solution set of the following equation.
\begin{enumerate}
\item
If neither $Q_1$ nor $Q_2$ is a deltoid, then
\begin{equation}
\label{eqn:CompatNonDelt}
w_1 + \mu_1 w_1^{-1} = \frac{\nu_1}{\nu_2} (w_2 + \mu_2 w_2^{-1}),
\end{equation}
\item
If $Q_2$ is an (anti)deltoid laterally coupled to $Q_1$ that is not an (anti)\-del\-toid, then
\begin{equation}
\label{eqn:CompatOneDelt}
w_1 + \mu_1 w_1^{-1} = \frac{\nu_1}{\xi_2} w_2^{-n_2}
\end{equation}
\item
If $Q_1$ and $Q_2$ are laterally coupled (anti)deltoids, then
$$
w_1^{n_1} = \frac{\xi_2}{\xi_1}w_2^{n_2}
$$
\end{enumerate}
Here $\mu_i$ are the involution factors from Definition \ref{dfn:InvFactors}, and $\nu_i$, $\xi_i$, $n_i$ are as in Corollary \ref{cor:ConfOrthInv}.
\end{lem}
\begin{proof}
Follows directly from the equations of the configuration spaces given in Corollary \ref{cor:ConfOrthInv}.
\end{proof}

\begin{lem}
\label{lem:InvolRed}
Let $(Q_1,Q_2)$ be a reducible involutive coupling as described in parts 2) and 3) of Lemma \ref{lem:InvolLift}, and let $Z_{12}^0$ be the component of its configuration space carrying the involution $j_{12}$. Then the quotient space $W^0 := Z_{12}^0/j_{12}$ has the following form.
\begin{enumerate}
\item
If $(Q_1,Q_2)$ is elliptic, and $t_1 - t_2$ is a half-period, then
$$
w_1 =
\begin{cases}
c w_2, &\text{if } \sin\sigma_1 \sin\sigma_2 > 0\\
-c w_2, &\text{if } \sin\sigma_1 \sin\sigma_2 < 0
\end{cases},
\quad \text{where} \quad
c = \sqrt{\frac{\frac{\sin\gamma_1\sin\delta_1}{\sin\overline\gamma_1\sin\overline\delta_1} - 1}{\frac{\sin\gamma_2\sin\delta_2}{\sin\overline\gamma_2\sin\overline\delta_2} - 1}}
$$
\item
If $(Q_1,Q_2)$ is conic such that
$$
\alpha_i + \gamma_i = \beta_i + \delta_i, \quad i = 1, 2,
$$
and $t_1 - t_2 \in \{0, \pi\}$, then
$$
w_1 =
\begin{cases}
c w_2, &\text{if } \sin\sigma_1 \sin\sigma_2 > 0\\
-c w_2, &\text{if } \sin\sigma_1 \sin\sigma_2 < 0
\end{cases},
\quad \text{where} \quad
c = \sqrt{\frac{\frac{\sin\gamma_1\sin\delta_1}{\sin\alpha_1\sin\beta_1} - 1}{\frac{\sin\gamma_2\sin\delta_2}{\sin\alpha_2\sin\beta_2} - 1}}
$$
\item
If $(Q_1,Q_2)$ is conic such that
$$
\alpha_i + \beta_i + \gamma_i + \delta_i = 2\pi, \quad i = 1, 2,
$$
and $t_1 - t_2 \in \{0, \pi\}$, then
$$
w_1 =
\begin{cases}
c^{-1} w_2, &\text{if } \sin\sigma_1 \sin\sigma_2 > 0\\
-c^{-1} w_2, &\text{if } \sin\sigma_1 \sin\sigma_2 < 0
\end{cases},
$$
where $c$ is given by the same formula as in the previous case.
\end{enumerate}
\end{lem}
In a reducible equimodular coupling of conic quadrilaterals, the side lengths can satisfy conditions other than those in parts 2) and 3) (see Lemma \ref{lem:RedCoupCon}), but we will not need the equation of $W$ in these cases.
\begin{proof}
Let $t_1 - t_2 \in \{0, 2K\}$ be a real half-period of an elliptic function $F$. Then $Z_{12}^0$ is the second component in \eqref{eqn:ParamEllPlain}, and its quotient (obtained by forgetting the coordinate $w_2$) is described by the equation
$$
w_1 =
\begin{cases}
\frac{q_1}{q_2} w_2, &\text{if } t_1 - t_2 = 0\\
-\frac{q_1}{q_2} w_2, &\text{if } t_1 - t_2 = 2K
\end{cases}
$$
If the coupling is geometrically realizable (that is $W \cap (\RP^1)^2$ is a curve), then $q_1$ and $q_2$ are either both real or both imaginary. Since we also have $p_1 = p_2$, formulas in Theorem \ref{thm:ParamSnCn} determining $\Re t_0$ imply that $\Re t_1 = \Re t_2$ if and only if either $\sigma_1, \sigma_2 < \pi$ or $\sigma_1, \sigma_2 > \pi$. By observing finally that $\frac{\sqrt{x}}{\sqrt{y}} = \sqrt{\frac{x}{y}}$ if $x$ and $y$ have the same sign, we arrive at the formula in the Lemma.

The argument in the conic case is the same.
\end{proof}

\subsection{Lateral coupling of deltoids}

\begin{lem}
\label{lem:LatDelt}
Let $Q_1$ and $Q_2$ be laterally coupled deltoids:
$$
\alpha_i = \beta_i, \gamma_i = \delta_i, i = 1,2
$$
If their coupling is non-involutive, then the configuration space $Z_{12}$ is described by the equation
\begin{equation}
\label{eqn:LatDelt}
aw_1^2 + 2b w_1w_2 + cw_2^2 + d = 0,
\end{equation}
where
$$
a = \frac{1-d_2^2}{\cos^2\delta_1}, \quad b = -\frac{1-d_1d_2}{\cos\delta_1\cos\delta_2}, \quad c = \frac{1-d_1^2}{\cos^2\delta_2}, \quad d = (d_1-d_2)^2
$$
and $d_1 = \dfrac{\tan\delta_1}{\tan\alpha_1}$, $d_2 = \dfrac{\tan\delta_2}{\tan\alpha_2}$.
\end{lem}
\begin{proof}
The spaces $Z_1$ and $Z_2$ have equations of the form $w_i = a_i z + b_i z^{-1}$, $i=1,2$, see Corollary \ref{cor:ConfOrthInv}. The coupling is non-involutive if and only if
$$
\det\begin{pmatrix} a_1 & b_1\\ a_2 & b_2 \end{pmatrix} \ne 0
$$
If this is the case, then $z$ and $z^{-1}$ can be expressed as linear functions of $w_1$ and $w_2$, whose substitution in $zz^{-1} = 1$ yields equation \eqref{eqn:LatDelt}.
\end{proof}

\section{Proof of the classification theorem}
\label{sec:Proof}
For every Kokotsakis polyhedron, the dihedral angles at the edges of its central face determine its shape uniquely. These angles were denoted by $\psi_1$, $\phi$, $\psi_2$, and $\theta$, and the tangents of their halves by $w_1$, $z$, $w_2$, and $u$, respectively. See Section \ref{sec:SpherLink} and the beginning of Section \ref{sec:AlgReform}. Therefore a continuous isometric deformation is represented by a non-constant map
\begin{equation}
\label{eqn:Deform}
\begin{aligned}
I &\to (\Sph^1)^4\\
t &\mapsto (\psi_1(t), \phi(t), \psi_2(t), \theta(t)),
\end{aligned}
\end{equation}
where $I \subset \R$ is a segment. Recall that the dihedral angles at each pair of adjacent edges are related through an equation that is polynomial in the corresponding tangents of half-angles, see \eqref{eqn:PolSystem}.

\subsection{Trivially flexible polyhedra}
\label{sec:TrivClass}
A deformation will be called trivial, if one of the functions $\psi_1(t)$, $\phi(t)$, $\psi_2(t)$, or $\theta(t)$ is constant. Let us classify trivial deformations.

In every trivial deformation there is a pair of adjacent dihedral angles one of which remains constant while the other one varies. Let $\phi(t) = \const$ and $\psi_1(t)$ be changing. From the classification of the configuration spaces of spherical quadrilaterals in Section \ref{sec:ConfSpace} it follows that either $\phi(t) = 0$ or $\phi(t) = \pi$; in the former case, $Q_1$ is a deltoid with $\alpha_1 = \delta_1$ and $\beta_1 = \gamma_1$, in the latter case $Q_1$ is an antideltoid with $\alpha_1 + \delta_1 = \pi = \beta_1 + \gamma_1$. A case-by-case analysis of the behavior of the dihedral angles $\psi_2(t)$ and $\theta(t)$ leaves us with the following possibilities.
\begin{itemize}
\item
Both $\theta(t)$ and $\psi_2(t)$ are changing;
\item
$\theta(t)$ is changing, $\psi_2(t) = 0$ or $\pi$;
\item
$\theta(t) = 0$ or $\pi$, while $\psi_2(t)$ is changing;
\item
$\theta(t) = 0$ or $\pi$, while $\psi_2(t) = \const \notin \{0, \pi\}$.
\end{itemize}
The corresponding flexible polyhedra are described in Section \ref{sec:TrivType}.

\subsection{A diagram of branched covers}
\label{sec:DiagCovers}
Assume that the deformation \eqref{eqn:Deform} is non-trivial, that is neither of the angles remains constant during the deformation. Make the substitution \eqref{eqn:Variables} and consider the irreducible components of the configuration spaces $Z_i$ containing the deformation \eqref{eqn:Deform}:
$$
\begin{aligned}
(w_1(t), z(t)) \in Z_1^0, \quad &(z(t), w_2(t)) \in Z_2^0,\\
(w_2(t), u(t)) \in Z_3^0, \quad &(u(t), w_1(t)) \in Z_4^0 \quad \forall t \in I
\end{aligned}
$$
The components $Z_i^0$ are well-defined since by choosing, if needed, a subinterval of $I$ we may assume that the path avoids singular points of all $Z_i$.

Denote further by $Z_{12}^0$ the irreducible component of the fiber product (see Section \ref{sec:Pullback})
$$
Z_1^0 \times_{\CP^1} Z_2^0 = \{(w_1, z, w_2) \mid P_1(w_1, z) = 0, P_2(z, w_2) = 0\}
$$
that contains the path $(w_1(t), z(t), w_2(t))$ (again we might need to choose a subinterval for $Z_{12}^0$ to be well-defined). Define similarly $Z_{23}^0$, $Z_{34}^0$, and $Z_{41}^0$. Then define $Z_{123}^0$ as the irreducible component of
$$
Z_{12}^0 \times_{Z_2^0} Z_{23}^0 = \{(w_1, z, w_2, u) \mid P_1(w_1, z) = 0, P_2(z, w_2) = 0, P_3(w_2, u) = 0\},
$$
and similarly $Z_{234}^0$, $Z_{341}^0$, and $Z_{412}^0$. All these algebraic sets have dimension~1, because they are branched covers of $Z_{ij}^0$, which are branched covers of $Z_i^0$. On the other hand, the intersection of any two of them contains a non-constant path. It follows that
$$
Z_{123}^0 = Z_{234}^0 = Z_{341}^0 = Z_{412}^0 =: Z_{all}^0
$$
The set $Z_{all}^0$ is thus a 1-dimensional irreducible component of the solution set of \eqref{eqn:PolSystem}.
We obtain the diagram of branched covers on Figure \ref{fig:BigDiagram}.

\begin{figure}[ht]
$$
\xymatrix{
&&& \CP^1 &&&\\
& Z_3^0 \ar[urr]^{f_3} \ar[ddl]_{g_3} &&&& Z_4^0 \ar[ull]_{f_4} \ar[ddr]^{g_4} &\\
&&& Z_{34}^0 \ar[ull]^{\widetilde{f_4}} \ar[urr]_{\widetilde{f_3}} &&&\\
\CP^1 && Z_{23}^0 \ar[uul]^{\widetilde{g_2}} \ar[ddl]_{\widetilde{g_3}} & Z_{all}^0 \ar[u]_{h_{34}} \ar[l]_{h_{23}} \ar[r]^{h_{41}} \ar[d]^{h_{12}} & Z_{41}^0 \ar[uur]_{\widetilde{g_1}} \ar[ddr]^{\widetilde{g_4}} && \CP^1\\
&&& Z_{12}^0 \ar[dll]_{\widetilde{f_1}} \ar[drr]^{\widetilde{f_2}} &&&\\
& Z_2^0 \ar[uul]^{g_2} \ar[drr]_{f_2} &&&& Z_1^0 \ar[uur]_{g_1} \ar[dll]^{f_1}&\\
&&& \CP^1 &&&
}
$$
\caption{A diagram of branched covers associated with a (non-trivially) flexible Kokotsakis polyhedron.}
\label{fig:BigDiagram}
\end{figure}
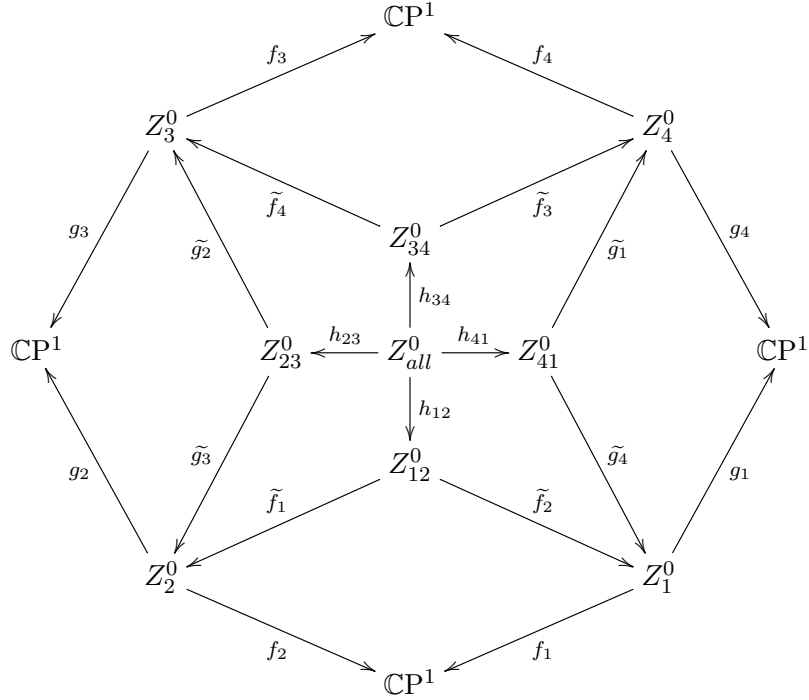

%
%

\subsection{Involutive polyhedra and involutive couplings}
\label{sec:InvNonInv}
Each of the maps on Figure \ref{fig:BigDiagram} is either an isomorphism or a two-fold branched cover. Let us concentrate on the multiplicities of the maps $h_{ij}$ in the center of the diagram.

\begin{dfn}
A deformation of a Kokotsakis polyhedron is called \emph{involutive}, if at least one of the maps $h_{ij}$ is two-fold. Otherwise, the deformation is called \emph{non-involutive}.
\end{dfn}
(By abuse of terminology, we will sometimes say that the polyhedron is involutive or non-involutive, although a priori it is possible that the same polyhedron has an involutive and a non-involutive deformation.)

If the deformation is non-involutive, then $Z_{all}^0$ can be identified with each of $Z_{ij}^0$ according to $h_{ij}$, and the big diagram collapses to a smaller one on Figure \ref{fig:SDiag1} in Section \ref{sec:WOInvol}, where the non-involutive case is dealt with.


\begin{lem}
\label{lem:InvCP}
An involutive polyhedron contains an involutive coupling.
\end{lem}
\begin{proof}
Let $h_{34}$ be two-fold. Consider the commutative square
\begin{equation}
\label{eqn:WSquare}
\xymatrix {
& Z_{all}^0 \ar[dl]_{h_{12}} \ar[dr]^{h_{34}} &\\
Z_{12}^0 \ar[dr]_{\pi_{12}} && Z_{34}^0 \ar[dl]^{\pi_{34}}\\
& W^0 &
}
\end{equation}
where $\pi_{ij}$ is as in \eqref{eqn:Pi} and $W^0$ is an irreducible component of the partial configuration space
\[
W = \{(w_1, w_2) \in (\CP^1)^2 \mid \exists z, u \in \CP^1 \text{ such that } (w_1, z, w_2, u) \in Z_{all}^0\}
\]
This is the diagram of a fiber product, and therefore if $h_{34}$ is two-fold, so is $\pi_{12}$. By Definition \ref{dfn:InvCoupl} this means that the coupling $(Q_1, Q_2)$ is involutive.
\end{proof}

The multiplicities of the maps in a fiber product diagram can be only as shown on Figure \ref{fig:PullMult}. Here a double arrow represents a two-fold branched cover, and a simple arrow represents an isomorphism.

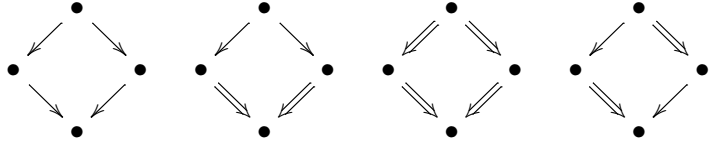
\begin{figure}[ht]
\centering
$$
{\xymatrix@=1pc{
& \bullet \ar[dl] \ar[dr] &\\
\bullet \ar[dr] && \bullet \ar[dl]\\
& \bullet &
}}
\quad
{\xymatrix@=1pc{
& \bullet \ar[dl] \ar[dr] &\\
\bullet \ar@2[dr] && \bullet \ar@2[dl]\\
& \bullet &
}}
\quad
{\xymatrix@=1pc{
& \bullet \ar@2[dl] \ar@2[dr] &\\
\bullet \ar@2[dr] && \bullet \ar@2[dl]\\
& \bullet &
}}
\quad
{\xymatrix@=1pc{
& \bullet \ar[dl] \ar@2[dr] &\\
\bullet \ar@2[dr] && \bullet \ar[dl]\\
& \bullet &
}}
$$
\caption{Possible multiplicities of maps in a fiber product.}
\label{fig:PullMult}
\end{figure}

The case when both $(Q_1, Q_2)$ and $(Q_3, Q_4)$ are involutive is studied in Section \ref{sec:TwoInvol}. Note that in this case the commutative square \eqref{eqn:WSquare} might have the form of the second square on Figure \ref{fig:PullMult}, so that a priori there may be non-involutive polyhedra with involutive couplings.

If $(Q_1, Q_2)$ is involutive, and $(Q_3, Q_4)$ is not, then $h_{34}$ is double and $h_{12}$ is simple. This case is dealt with in Section \ref{sec:InvolPol}.



\subsection{Combination of two involutive couplings}
\label{sec:TwoInvol}
If $(Q_1, Q_2)$ and $(Q_3, Q_4)$ are involutive, then we have
$$
Z_{12}^0/j_{12} = W^0 = Z_{34}^0/j_{34}
$$
Equations describing the quotient of the configuration space of an involutive coupling are given in Lemmas \ref{lem:CompatOrthodiag} and \ref{lem:InvolRed}. Consider them case by case.

\subsubsection{Cubic case}
The space $W^0$ is described by an equation of the form
$$
a_1w_1 + b_1w_1^{-1} = a_2w_2 + b_2w_2^{-1}
$$
This is the case 1) of Lemma \ref{lem:CompatOrthodiag}. Thus both couplings $(Q_1, Q_2)$ and $(Q_3, Q_4)$ consist of compatible orthodiagonal quadrilaterals neither of which is an (anti)deltoid. By comparing the coefficients at the corresponding terms, we obtain
$$
\mu_1 = \mu_4, \quad \mu_2 = \mu_3, \quad \nu_1\nu_3 = \nu_2\nu_4
$$
The first two equations say that the involution factors of $Q_1$ and $Q_4$, respectively of $Q_2$ and $Q_3$, at their common vertex are equal. The third equation is equivalent to $\cos\delta_1\cos\delta_3 = \cos\delta_2\cos\delta_4$. Thus we get the orthodiagonal involutive type \ref{sec:OrthoInv}.

\subsubsection{Rational case}
The space $W^0$ is described by an equation of the form
$$
aw_1 + bw_1^{-1} = w_2^{\pm1}
$$
This means that $Q_1$ and $Q_2$ form a compatible pair of orthodiagonal quadrilaterals, $Q_2$ is an (anti)deltoid while $Q_1$ not, and the same is true for the pair $(Q_4, Q_3)$. By comparing the coefficients, we obtain
$$
\mu_1 = \mu_4, \quad n_2 = n_3, \quad \nu_1\xi_3 = \nu_4\xi_2
$$
Thus $Q_1$ and $Q_4$ have equal involution factors at their common vertex, and $Q_2$, $Q_3$ are either both deltoids or both antideltoids. By Definition \ref{dfn:Compatible} this means that all pairs of adjacent quadrilaterals are compatible. The last equation implies $\cos\delta_1\cos\delta_3 = \cos\delta_2\cos\delta_4$ so that we get again the orthodiagonal involutive type \ref{sec:OrthoInv}.


\subsubsection{Linear case}
The space $W^0$ is described by an equation of the form
$$
w_1w_2 = \const \text{ or } \frac{w_1}{w_2} = \const
$$
By switching, if needed, the right boundary strip (that is, replacing $\beta_1$, $\gamma_1$, $\beta_4$, and $\gamma_4$ by their complements to $\pi$) the former case can be reduced to the latter. We thus have a linear compound as defined in Section \ref{sec:LinComp}. By Lemmas \ref{lem:CompatOrthodiag} and \ref{lem:InvolRed} each of the couplings $(Q_1, Q_2)$ and $(Q_3, Q_4)$ has one of the following forms.
\begin{itemize}
\item
two deltoids or two antideltoids, coupled laterally and with equal involution factors at the common vertex;
\item
a reducible coupling of two elliptic quadrilaterals with the shift difference $t_1 - t_2 \in \{0, 2K\}$;
\item
a reducible coupling of two conic quadrilaterals with the shift difference $t_1 - t_2 \in \{0, \pi\}$.
\end{itemize}
The first and the second linear couplings are listed in Sections \ref{sec:LinLatDelt} and \ref{sec:LinEll}. Let us show that the third one can be reduced to the cases described in Section \ref{sec:LinCon}.
We have $m_1 = m_2$ by Lemma \ref{lem:InvolLift} and $n_1 = n_2$ because of $w_1 = c w_2$, hence the side lengths of both $Q_1$ and $Q_2$ satisfy the same of the four possible relations $\alpha_i \pm \beta_i \pm \gamma_i \pm \delta_i \equiv 0 (\mod 2\pi)$. By switching, if needed, the lower boundary strip, these four restrict to two possibilities \eqref{eqn:BothCyc} and \eqref{eqn:BothPer2Pi}.
The same can be done with the coupling $(Q_3, Q_4)$.

\subsection{Combination of an involutive coupling with a non-involutive}
\label{sec:InvolPol}

Assume that $(Q_1, Q_2)$ is involutive, and $(Q_3,Q_4)$ is not. The polyhedron is flexible if and only if $W_z = W_u$, where
$$
W_z := \pi_{12}(Z_{12}^0) = Z_{12}^0/j_{12}, \quad W_u := \pi_{34}(Z_{34}^0)
$$
We will consider all of the forms that $(Q_1,Q_2)$ can take and study the multiplicities of the maps in the diagram on Figure \ref{fig:BigDiagram}.

\subsubsection{The involutive coupling is equimodular}
\label{sec:InvolEquimod}
That is, $(Q_1,Q_2)$ has the form described in parts 2) and 3) of Lemma \ref{lem:InvolLift}. By switching, if needed, the boundary strips we can assume that $(Q_1, Q_2)$ is one of the linear couplings described in Sections \ref{sec:LinEll} and \ref{sec:LinCon}.

The maps between configuration spaces have multiplicities as shown on Figure \ref{fig:Diag1}, where dotted lines stand for the maps with unknown multiplicities. However, with the help of Figure \ref{fig:PullMult} we can determine most of them, see Figure \ref{fig:Diag1}.

\begin{figure}[ht]
\centering
$$
{\xymatrix@=1pc{
&&& \bullet &&&\\
& \bullet \ar@.[urr] \ar@.[ddl] &&&& \bullet \ar@.[ull] \ar@.[ddr] &\\
&&& \bullet \ar@.[ull] \ar@.[urr] &&&\\
\bullet && \bullet \ar@.[uul] \ar@.[ddl] & \bullet \ar@2[u] \ar@.[l] \ar@.[r] \ar[d] & \bullet \ar@.[uur] \ar@.[ddr] && \bullet\\
&&& \bullet \ar[dll] \ar[drr] &&&\\
& \bullet \ar@2[uul] \ar@2[drr] &&&& \bullet \ar@2[uur] \ar@2[dll]&\\
&&& \bullet &&&
}}
\qquad
{\xymatrix@=1pc{
&&& \bullet &&&\\
& \bullet \ar@.[urr] \ar[ddl] &&&& \bullet \ar@.[ull] \ar[ddr] &\\
&&& \bullet \ar[ull] \ar[urr] &&&\\
\bullet && \bullet \ar@2[uul] \ar[ddl] & \bullet \ar@2[u] \ar[l] \ar[r] \ar[d] & \bullet \ar@2[uur] \ar[ddr] && \bullet\\
&&& \bullet \ar[dll] \ar[drr] &&&\\
& \bullet \ar@2[uul] \ar@2[drr] &&&& \bullet \ar@2[uur] \ar@2[dll]&\\
&&& \bullet &&&
}}
$$
\caption{Solving the diagram for Section \ref{sec:InvolEquimod}.}
\label{fig:Diag1}
\end{figure}
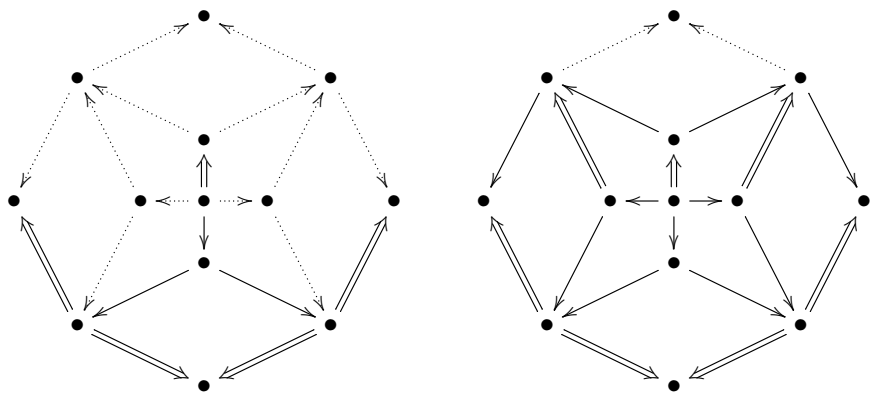

The two remaining dotted arrows are either both simple or both double. Thus $(Q_3,Q_4)$ is either a coupling of (anti)isograms or a reducible coupling of (anti)deltoids. Theorem \ref{thm:ParamLin} and Lemma \ref{lem:RedCoupDelt} imply that such a coupling yields a linear dependence $w_1 = cw_2$ if and only if $Q_3$ and $Q_4$ are of the same type (e.~g. both are antideltoids) and allow us to compute the coefficient $c$.
We thus have a polyhedron of linear compound type, where $(Q_3,Q_4)$ is as in Section \ref{sec:LinIso} or \ref{sec:LinFrontDelt}.

\subsubsection{The involutive coupling consists of two (anti)deltoids}
\label{sec:InvolTwoDelt}
Modulo switching, $(Q_1, Q_2)$ is as described in Section \ref{sec:LinLatDelt}. Again, the diagram can be solved as shown on Figure \ref{fig:Diag2}.

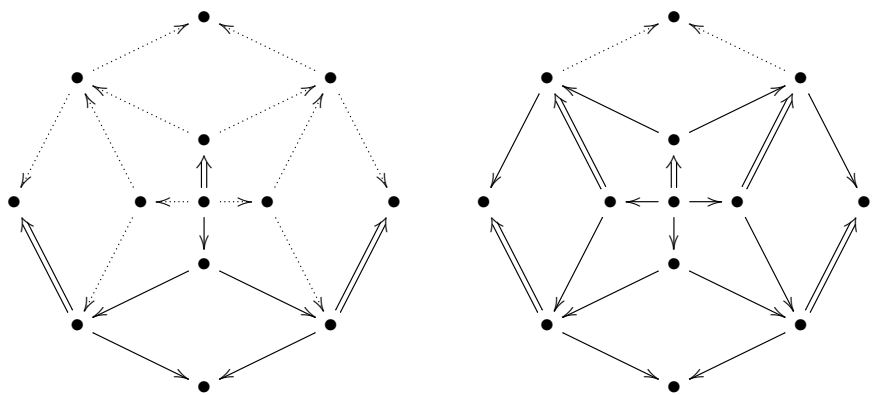
\begin{figure}[ht]
\centering
$$
{\xymatrix@=1pc{
&&& \bullet &&&\\
& \bullet \ar@.[urr] \ar@.[ddl] &&&& \bullet \ar@.[ull] \ar@.[ddr] &\\
&&& \bullet \ar@.[ull] \ar@.[urr] &&&\\
\bullet && \bullet \ar@.[uul] \ar@.[ddl] & \bullet \ar@2[u] \ar@.[l] \ar@.[r] \ar[d] & \bullet \ar@.[uur] \ar@.[ddr] && \bullet\\
&&& \bullet \ar[dll] \ar[drr] &&&\\
& \bullet \ar@2[uul] \ar[drr] &&&& \bullet \ar@2[uur] \ar[dll]&\\
&&& \bullet &&&
}}
\qquad
{\xymatrix@=1pc{
&&& \bullet &&&\\
& \bullet \ar@.[urr] \ar[ddl] &&&& \bullet \ar@.[ull] \ar[ddr] &\\
&&& \bullet \ar[ull] \ar[urr] &&&\\
\bullet && \bullet \ar@2[uul] \ar[ddl] & \bullet \ar@2[u] \ar[l] \ar[r] \ar[d] & \bullet \ar@2[uur] \ar[ddr] && \bullet\\
&&& \bullet \ar[dll] \ar[drr] &&&\\
& \bullet \ar@2[uul] \ar[drr] &&&& \bullet \ar@2[uur] \ar[dll]&\\
&&& \bullet &&&
}}
$$
\caption{Solving the diagram for Section \ref{sec:InvolTwoDelt}.}
\label{fig:Diag2}
\end{figure}

It follows that $(Q_3,Q_4)$ has the same form as in the previous case, and we have a linear compound coupling.

\subsubsection{The involutive coupling is orthodiagonal with one (anti)deltoid}
\label{sec:InvolOneDelt}
Without loss of generality, let $Q_2$ be an (anti)\-deltoid, $Q_1$ orthodiagonal elliptic.
By switching, if needed, the left boundary strip, we can transform $Q_2$ to antideltoid. Then, according to Lemma \ref{lem:CompatOrthodiag} we have
\begin{equation}
\label{eqn:Z12/j12}
W_z = \{(w_1,w_2) \mid w_1 + \mu_1 w_1^{-1} = \frac{\nu_1}{\xi_2} w_2\}
\end{equation}

Figure \ref{fig:Diag3} shows the map multiplicities that we obtain by diagram chasing. The further case distinction depends on the multiplicity of the question mark map
\begin{figure}[ht]
\centering
$$
{\xymatrix@=1pc{
&&& \bullet &&&\\
& \bullet \ar@.[urr] \ar@.[ddl] &&&& \bullet \ar@.[ull] \ar@.[ddr] &\\
&&& \bullet \ar@.[ull] \ar@.[urr] &&&\\
\bullet && \bullet \ar@.[uul] \ar@.[ddl] & \bullet \ar@2[u] \ar@.[l] \ar@.[r] \ar[d] & \bullet \ar@.[uur] \ar@.[ddr] && \bullet\\
&&& \bullet \ar@2[dll] \ar[drr] &&&\\
& \bullet \ar@2[uul] \ar[drr] &&&& \bullet \ar@2[uur] \ar@2[dll]&\\
&&& \bullet &&&
}}
\qquad
{\xymatrix@=1pc{
&&& \bullet &&&\\
& \bullet \ar@.[urr] \ar@.[ddl] &&&& \bullet \ar@.[ull] \ar[ddr] &\\
&&& \bullet \ar@.[ull] \ar[urr] &&&\\
\bullet && \bullet \ar@2[uul] \ar@.[ddl] & \bullet \ar@2[u] \ar@.[l]_{?} \ar[r] \ar[d] & \bullet \ar@2[uur] \ar[ddr] && \bullet\\
&&& \bullet \ar@2[dll] \ar[drr] &&&\\
& \bullet \ar@2[uul] \ar[drr] &&&& \bullet \ar@2[uur] \ar@2[dll]&\\
&&& \bullet &&&
}}
$$
\caption{Solving the diagram for Section \ref{sec:InvolOneDelt}.}
\label{fig:Diag3}
\end{figure}
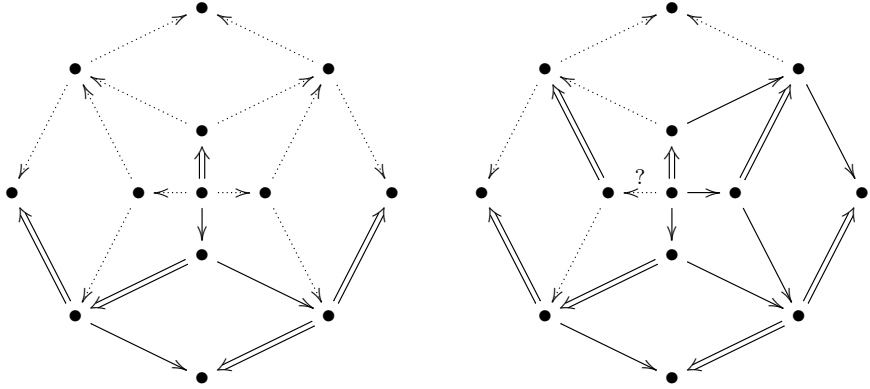

\noindent\emph{Case 1.} The question mark map on Figure \ref{fig:Diag3} is two-fold.

This yields the multiplicities on Figure \ref{fig:Diag4}, left. Thus $Q_4$ is an (anti)deltoid compatibly coupled with an orthodiagonal non-deltoid $Q_1$ (since the map $h_{23}$ is two-fold), and $Q_3$ is an (anti)isogram. By switching, if needed, the upper boundary strip we can transform $Q_4$ to an antideltoid. Let us find the equation of the space $W_u$.

By Theorem \ref{thm:ParamDegenCon} and Corollary \ref{cor:ConfOrthInv} we have
$$
Z_3^0 = \{(u,w_2) \mid u^{\pm 1} = \kappa_3 w_2\}, \quad Z_4^0 = \{(u,w_1) \mid w_1 + \mu_4 w_1^{-1} = \zeta_4 u^{-1}\}
$$
The equation of $W_u$ is obtained by substituting the first equation in the second one. We have $W_z = W_u$ if and only if the resulting equation is proportional to \eqref{eqn:Z12/j12}. Thus we must have $u^{-1} = \kappa_3 w_2$, which means that $Q_3$ is an isogram, and
$$
\mu_1 = \mu_4, \quad \kappa_3 \xi_2 \zeta_4 = \nu_1
$$
The first equation holds automatically, since $Q_1$ and $Q_4$ are compatible. The second one is an additional restriction. Altogether we obtain a polyhedron described in Section \ref{sec:TransType1}.

\smallskip

\begin{figure}[ht]
\centering
$$
{\xymatrix@=1pc{
&&& \bullet &&&\\
& \bullet \ar[urr] \ar[ddl] &&&& \bullet \ar@2[ull] \ar[ddr] &\\
&&& \bullet \ar@2[ull] \ar[urr] &&&\\
\bullet && \bullet \ar@2[uul] \ar[ddl] & \bullet \ar@2[u] \ar@2[l] \ar[r] \ar[d] & \bullet \ar@2[uur] \ar[ddr] && \bullet\\
&&& \bullet \ar@2[dll] \ar[drr] &&&\\
& \bullet \ar@2[uul] \ar[drr] &&&& \bullet \ar@2[uur] \ar@2[dll]&\\
&&& \bullet &&&
}}
\qquad
{\xymatrix@=1pc{
&&& \bullet &&&\\
& \bullet \ar@.[urr] \ar@2[ddl] &&&& \bullet \ar@.[ull] \ar[ddr] &\\
&&& \bullet \ar[ull] \ar[urr] &&&\\
\bullet && \bullet \ar@2[uul] \ar@2[ddl] & \bullet \ar@2[u] \ar[l] \ar[r] \ar[d] & \bullet \ar@2[uur] \ar[ddr] && \bullet\\
&&& \bullet \ar@2[dll] \ar[drr] &&&\\
& \bullet \ar@2[uul] \ar[drr] &&&& \bullet \ar@2[uur] \ar@2[dll]&\\
&&& \bullet &&&
}}
$$
\caption{Cases 1 and 2 of Section \ref{sec:InvolOneDelt}.}
\label{fig:Diag4}
\end{figure}
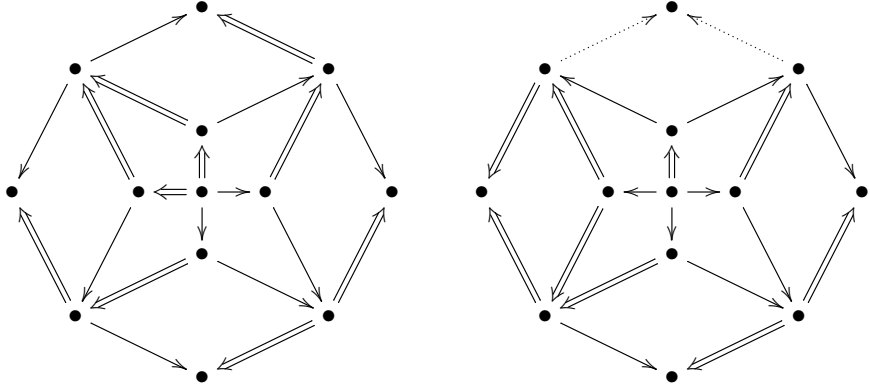

\noindent\emph{Case 2.} The question mark map on Figure \ref{fig:Diag3} is an isomorphism.

This yields the multiplicities on Figure \ref{fig:Diag4}, right, where the dotted lines are either both simple or both double. We make a further case distinction.

\smallskip

\noindent\emph{Case 2a).} The dotted lines on Figure \ref{fig:Diag4}, right, are simple.

Thus $Q_3$ is an (anti)deltoid, and $Q_4$ is an (anti)isogram. Let us find the equation of the space $W_u$.

By Theorem \ref{thm:ParamLin} and Corollary \ref{cor:ConfOrthInv} we have
$$
Z_3^0 = \{(u,w_2) \mid u + \lambda_3 u^{-1} = \xi_3 w_2^{n_2}\}, \quad Z_4^0 = \{(u,w_1) \mid u^{\pm 1} = \kappa_4 w_1\}
$$
By eliminating the variable $u$ we must obtain an equation of the form \eqref{eqn:Z12/j12}. It follows that $n_2 = 1$, so that $Q_3$ is a deltoid. By switching, if needed, the upper boundary strip, we can transform $Q_4$ to an antiisogram, so that $u = \kappa_4 w_1$. By making this substitution and comparing the coefficients with those in \eqref{eqn:Z12/j12}, we obtain the necessary and sufficient conditions for flexibility:
$$
\frac{\lambda_3}{\kappa_4^2} = \mu_1, \quad \frac{\xi_3}{\kappa_4} = \frac{\nu_1}{\xi_2}
$$
As a result, we obtain a polyhedron described in Section \ref{sec:TransType2}.


\smallskip

\noindent\emph{Case 2b).} The dotted lines on Figure \ref{fig:Diag4}, right, are double.

Thus $Q_4$ is an (anti)deltoid forming a reducible coupling with $Q_3$. It follows that $Q_3$ is conic. The two situation when $(Q_3,Q_4)$ is reducible are described in Lemma \ref{lem:RedCoupDelt}, and the equations for the corresponding space $W$ are given in Lemma \ref{lem:WRed}.

If $m_3 = m_4$, $p_3 = p_4$, then the equation of $W_u$ is
$$
w_1 + \mu_4 e^{\pm 2it_3} w_1^{-1} = \epsilon_4 \frac{2i\sqrt{-\mu_4} e^{\pm it_3}}{q_3} w_2^{n_3}
$$
which has the same form as \eqref{eqn:Z12/j12} if and only if $n_3 = 1$. By switching, if needed, the upper boundary strip, we can achieve $m_3 = m_4 = 1$. Then the conic quadrilateral $Q_3$ is circumscribed, and $Q_4$ is a deltoid. By equating the coefficients at $w_1^{-1}$ and $w_2$, we obtain necessary and sufficient conditions of flexibility, as described in Section \ref{sec:3bii}.

If $m_3 = -m_4$, $p_3 = \pm \frac1{p_4}$, then the set $W_u$ is described by equation of the form \eqref{eqn:WConDeltSkew}, thus $W_u = W_z$ cannot take place, and there is no flexible polyhedron in this case.

\subsubsection{The only involutive coupling is orthodiagonal without (anti)deltoids}
\label{sec:InvolWODelt}
By Lemma \ref{lem:CompatOrthodiag} we have
\begin{equation}
\label{eqn:EqnCase4}
W_z = \{(w_1,w_2) \mid w_1 + \mu_1w_1^{-1} = \frac{\nu_1}{\nu_2}(w_2 + \mu_2w_2^{-1})
\end{equation}


If both $h_{23}$ and $h_{41}$ are two-fold, then the couplings $(Q_2,Q_3)$ and $(Q_4,Q_1)$ are involutive. Up to a rotation of the diagram, this situation was considered in Section \ref{sec:TwoInvol}. Thus assume without loss of generality that $h_{23}$ is an isomorphism. Then we obtain the map multiplicities as on Figure \ref{fig:Diag5}.

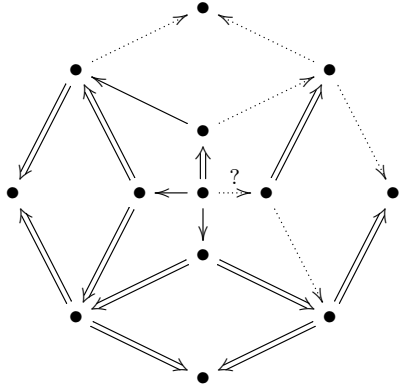
\begin{figure}[ht]
\centering
$$
\xymatrix@=1pc{
&&& \bullet &&&\\
& \bullet \ar@.[urr] \ar@2[ddl] &&&& \bullet \ar@.[ull] \ar@.[ddr] &\\
&&& \bullet \ar[ull] \ar@.[urr] &&&\\
\bullet && \bullet \ar@2[uul] \ar@2[ddl] & \bullet \ar@2[u] \ar[l] \ar@.[r]^{?} \ar[d] & \bullet \ar@2[uur] \ar@.[ddr] && \bullet\\
&&& \bullet \ar@2[dll] \ar@2[drr] &&&\\
& \bullet \ar@2[uul] \ar@2[drr] &&&& \bullet \ar@2[uur] \ar@2[dll]&\\
&&& \bullet &&&
}
$$
\caption{The diagram for Section \ref{sec:InvolWODelt}.}
\label{fig:Diag5}
\end{figure}

\smallskip

\noindent\emph{Case 1.} The question mark map on Figure \ref{fig:Diag5} is two-fold.

This yields the multiplicities on Figure \ref{fig:Diag6}, left. The quadrilateral $Q_3$ is either conic or elliptic, thus its configuration space $Z_3^0 = Z_3$ is an irreducible curve described by an equation of the form \eqref{eqn:AdjZ}.
Since $Q_4$ is an isogram, the equation of $W_u$ is obtained from that of $Z_3^0$ by a linear substitution $u = cw_2$ or $u = cw_2^{-1}$. Therefore it cannot have the form \eqref{eqn:EqnCase4}. Thus there are no flexible Kokotsakis polyhedra with this diagram.

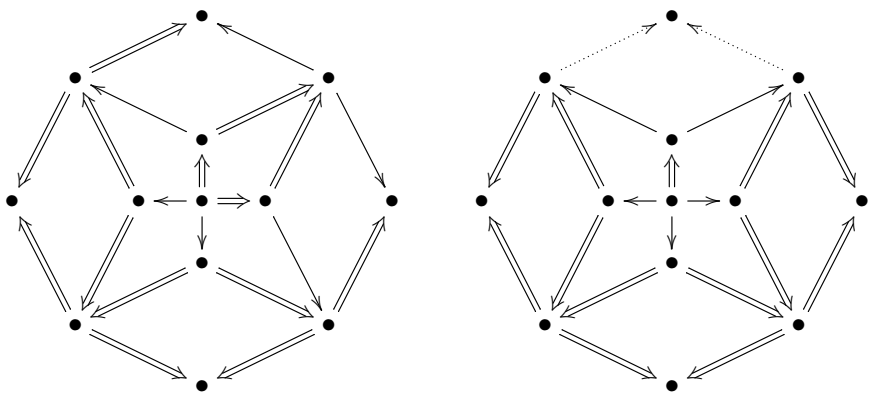
\begin{figure}[ht]
\centering
$$
{\xymatrix@=1pc{
&&& \bullet &&&\\
& \bullet \ar@2[urr] \ar@2[ddl] &&&& \bullet \ar[ull] \ar[ddr] &\\
&&& \bullet \ar[ull] \ar@2[urr] &&&\\
\bullet && \bullet \ar@2[uul] \ar@2[ddl] & \bullet \ar@2[u] \ar[l] \ar@2[r] \ar[d] & \bullet \ar@2[uur] \ar[ddr] && \bullet\\
&&& \bullet \ar@2[dll] \ar@2[drr] &&&\\
& \bullet \ar@2[uul] \ar@2[drr] &&&& \bullet \ar@2[uur] \ar@2[dll]&\\
&&& \bullet &&&
}}
\qquad
{\xymatrix@=1pc{
&&& \bullet &&&\\
& \bullet \ar@.[urr] \ar@2[ddl] &&&& \bullet \ar@.[ull] \ar@2[ddr] &\\
&&& \bullet \ar[ull] \ar[urr] &&&\\
\bullet && \bullet \ar@2[uul] \ar@2[ddl] & \bullet \ar@2[u] \ar[l] \ar[r] \ar[d] & \bullet \ar@2[uur] \ar@2[ddr] && \bullet\\
&&& \bullet \ar@2[dll] \ar@2[drr] &&&\\
& \bullet \ar@2[uul] \ar@2[drr] &&&& \bullet \ar@2[uur] \ar@2[dll]&\\
&&& \bullet &&&
}}
$$
\caption{Cases 1 and 2 of Section \ref{sec:InvolWODelt}.}
\label{fig:Diag6}
\end{figure}

\smallskip

\noindent\emph{Case 2.} The question mark map on Figure \ref{fig:Diag5} is an isomorphism.

This yields the multiplicities on Figure \ref{fig:Diag6}, right, and makes a further case distinction necessary.

\smallskip

\noindent\emph{Case 2a).} The dotted lines on Figure \ref{fig:Diag6}, right, are simple.

Then $(Q_3,Q_4)$ is a lateral coupling of (anti)deltoids. By switching, if needed, the left and/or the right boundary strips we can transform them to deltoids. Then $W_u$ has equation \eqref{eqn:LatDelt} which contradicts \eqref{eqn:EqnCase4} (the former polynomial cannot be a factor of the latter). Thus there are no flexible Kokotsakis polyhedra with this diagram.

\smallskip

\noindent\emph{Case 2b).} The dotted lines on Figure \ref{fig:Diag6}, right, are double.

Then $(Q_3,Q_4)$ is a reducible non-involutive coupling of conic or elliptic quadrilaterals. Equations of the reducible components of the space $W_u$ are described in Lemma \ref{lem:WRed} and are all different from \eqref{eqn:EqnCase4}. Hence this diagram also doesn't produce any flexible Kokotsakis polyhedra.



\subsection{Flexible polyhedra without involutive couplings}
\label{sec:WOInvol}
By the argument in Section \ref{sec:InvNonInv}, if a deformation contains no involutive couplings, then the diagram on Figure \ref{fig:BigDiagram} collapses to the diagram on Figure \ref{fig:SDiag1}. The further classification is based on the multiplicities of the maps $h_i$.

If $h_i$ is two-fold and $Z_i$ is conic or elliptic, then both couplings $(Q_{i-1}, Q_i)$ and $(Q_i, Q_{i+1})$ are irreducible. This case is analyzed in Section \ref{sec:IrredWOInvol}. Here we find the only flexible polyhedron whose vertices have elliptic configuration spaces coupled non-involutively and non-reducibly (which means that all resultant polynomials $R_{ij}$ are irreducible).

The cases when $h_i$ is two-fold, and the polyhedron $Q_i$ is either (anti)deltoid or (anti)isogram, are dealt with in Sections \ref{sec:HTwoDelt} and \ref{sec:HTwoIso}. In Section \ref{sec:AllIsom} we study the case when all $h_i$ are isomorphisms.

\subsubsection{Irreducible coupling with a conic or elliptic quadrilateral}
\label{sec:IrredWOInvol}



\begin{lem}
\label{lem:DoubleEll}
Assume that the quadrilateral $Q_1$ is conic or elliptic, and that $\widetilde{f_2}$ and $\widetilde{g_4}$ are equivalent two-fold branched covers. Then $Q_1$, $Q_2$, and $Q_4$ are all elliptic.

Besides, the corresponding phase shift $t_1$ satisfies either $\Im t_1 = \frac{iK'_1}{3}$ or $\Im t_1 = \frac{iK'_1}2$ and the branch set $C_1$ of $\widetilde{f_2}$ and $\widetilde{g_4}$ has one of the forms depicted on Figures \ref{fig:C1sn} and \ref{fig:C1cn}.
\end{lem}

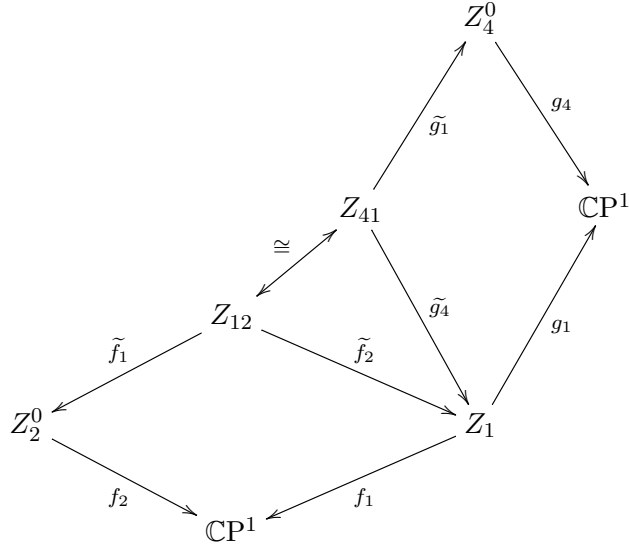
\begin{figure}[ht]
$$
\xymatrix{
&  &&&& Z_4^0 \ar[ddr]^{g_4} &\\
&&&  &&&\\
&&  &  & Z_{41} \ar[uur]_{\widetilde{g_1}} \ar[ddr]^{\widetilde{g_4}} && \CP^1\\
&&& Z_{12} \ar[dll]_{\widetilde{f_1}} \ar[drr]^{\widetilde{f_2}} \ar @{<->} [ur]^{\cong} &&&\\
& Z_2^0 \ar[drr]_{f_2} &&&& Z_1 \ar[uur]_{g_1} \ar[dll]^{f_1}&\\
&&& \CP^1 &&&
}
$$
\caption{A half of the diagram from Figure \ref{fig:BigDiagram}, under the assumption $\widetilde{f_2} \simeq \widetilde{g_4}$.}
\label{fig:PartDiag}
\end{figure}

\begin{proof}
Since $\widetilde{f_2}$ is two-fold, the map $f_2$ is also two-fold. The same holds for the map $g_4$. The two-fold branched covers $\widetilde{f_2}$ and $\widetilde{g_4}$ are equivalent if and only if they have the same branch set. By Lemma \ref{lem:FibProd}, this means
\begin{equation}
\label{eqn:C1}
f_1^{-1}(A_2) = g_1^{-1}(B_4) =: C_1,
\end{equation}
where $A_2$ and $B_4$ are the branch sets of the maps $f_2$ and $g_4$, respectively.

As a full preimage under $f_1$, the set $C_1$ is invariant under the action of the deck transformation $i_1$ of the two-fold cover $f_1$. Similarly, $C_1$ is invariant under the deck transformation $j_1$ of $g_1$:
$$
i_1(C_1) = C_1 = j_1(C_1)
$$
Consider a parametrization of $Z_1$ obtained in Theorems \ref{thm:ParamCon} and \ref{thm:ParamSnCn}:
$$
Z_1 = \{(p_1 F(t), q_1 F(t+t_1))\},
$$
where the function $F$ is $\sin$, $\sn$, or $\cn$. By identifying $Z_1$ with its parameter domain, we have
$$
f_1(t) = p_1 F(t), \quad g_1(t) = q_1 F(t+t_1)
$$
Due to
$$
i_1 \circ j_1(t) = t + 2t_1,
$$
the set $C_1$ is invariant under the shift by $2t_1$. This implies that $F \ne \sin$, because the orbit of the shift by $2t_1$ is infinite in $\C/2\pi Z$ (recall that $\Im t_1 \ne 0$), while the set $C_1$ is finite.

It follows that $Q_1$ is elliptic, $F = \sn(\cdot, k_1)$ or $\cn(\cdot, k_1)$. Let
$$
\Lambda =
\begin{cases}
4K_1\Z + 2iK'_1\Z, &\text{if }F = \sn\\
4K_1\Z + (2K + 2iK'_1)\Z, &\text{if }F = \cn
\end{cases}
$$
be the period lattice of $F$. We have $Z_1 = \C/\Lambda$. On the torus $\C/\Lambda$, the orbit of the shift by $2t_1$ is finite if and only if $\Im t_1 = \frac{iK'_1}{n}$ for some $n \ge 2$ (recall that $\Im t_1 \in (0, iK'_1)$ and $\Re t_1$ is a multiple of $K_1$).

Equation \eqref{eqn:C1} implies that $C_1 \subset \C/\Lambda$ is invariant also under the involution
$$
t \mapsto -t, \quad \text{if }F = \sn, \qquad t \mapsto 2K_1 - t, \quad \text{if }F = \cn
$$
Indeed, the set $A_2$ is symmetric with respect to $0$, and we have $-\sn t = \sn(-t)$ and $-\cn t = \cn(2K_1 - t)$.
Combined with the invariance with respect to involutions $i_1$ and $j_1$ this implies that $C_1$ is invariant under the group generated by three point reflections
$$
t \mapsto -t, \quad t \mapsto 2K_1 - t, \quad t \mapsto 2t_1 - t
$$
Another set of generators for the same group consists of two shifts and one point reflection:
$$
t \mapsto t + 2K_1, \quad t \mapsto t + 2t_1, \quad t \mapsto -t
$$
Thus we have
\begin{equation}
\label{eqn:CL}
C_1 = \left((x + \Lambda_n) \cup (-x + \Lambda_n)\right)/\Lambda
\end{equation}
for some $x \in \C$, where
$$
\Lambda_n := 2K_1\Z + 2t_1\Z = 2K_1\Z + \frac{2iK'_1}{n}\Z
$$
Since $|\Lambda_n/\Lambda| = 2n$, it follows that
\begin{equation}
\label{eqn:CardC1}
|C_1| =
\begin{cases}
4n, &\text{if }x \notin \frac12 \Lambda_n\\
2n, &\text{if }x \in \frac12 \Lambda_n
\end{cases}
\end{equation}
From the classification of the configuration spaces we know that $|A_2| = 2$ (if $Q_2$ is an (anti)deltoid or conic quadrilateral) or $|A_2| = 4$ (if $Q_2$ is an elliptic quadrilateral). Since $f_1$ is a two-fold cover, equation \eqref{eqn:C1} implies that $|C_1| \le 8$. Thus, due to \eqref{eqn:CardC1} we have $n \le 4$.

Let us show that $n \ne 4$. If $n = 4$, then by \eqref{eqn:CardC1} $x$ is a half-period of $\Lambda_n$. Without loss of generality,
$$
x \in \{0, K_1, -t_1, -t_1+K_1\}
$$
Consider first the case $F = \sn$. If $x = 0$, then $0 \in C_1$, hence $0 = f_1(0) \in A_2$, which is impossible, see Lemma \ref{lem:BranchInvol}. If $x = K_1$, then we have $|A_2| > 4$ because $K_1$ is a branch point for the function $\sn(\cdot, k_1)$ and
$$
|C_1| = 2|A_2| \Leftrightarrow C_1 \text{ contains no branch points of }f_1
$$
Similarly, $x = -t_1$ would imply $0 = g_1(-t_1) \in B_4$, and $x = -t_1 + K_1$ would imply that $C_1$ contains a branch point of $g_1$. Summarizing, there is no half-period shift of $\Lambda_4$ whose image under $f_1$ and $g_1$ would consist of four points, all different from $0$. In the case $F = \cn$ the argument is similar. Thus $n \ne 4$.

Let us look at the case $n = 3$. Due to \eqref{eqn:CardC1} and $|C_1| \le 8$, the point $x$ must be a half-period of $\Lambda_3$:
$$
x \in \{0, K_1, \frac{iK_1'}3, K_1 + \frac{iK_1'}3\}
$$
If $F = \sn$, then $x = 0$ would lead to $0 \in A_2$, and $x = \frac{iK_1'}3$ would lead to $\infty \in A_2$. It follows that $C_1$ is one of the sets depicted on Figure \ref{fig:C1sn}, with $C_1 + t_1$ being the other one. The images of $C_1$ and $C_1 + t_1$ under the map $t \mapsto p_1 \sn(t,k_1)$, respectively $t \mapsto q_1 \sn(t,k_1)$, consist of four points each. Thus we have $|A_2| = |B_4| = 4$, and therefore both $Q_2$ and $Q_4$ are elliptic quadrilaterals.

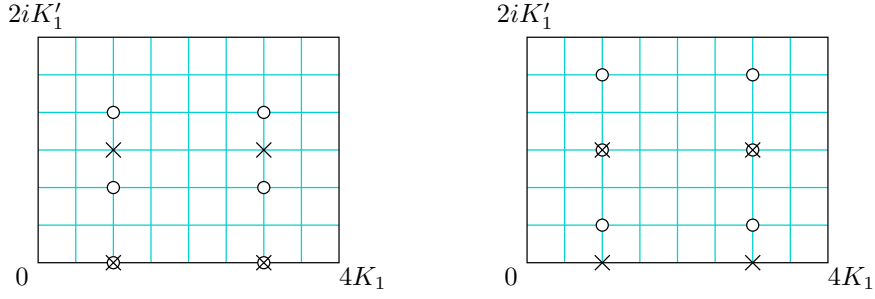
\begin{figure}
\begin{picture}(0,0)%
\includegraphics{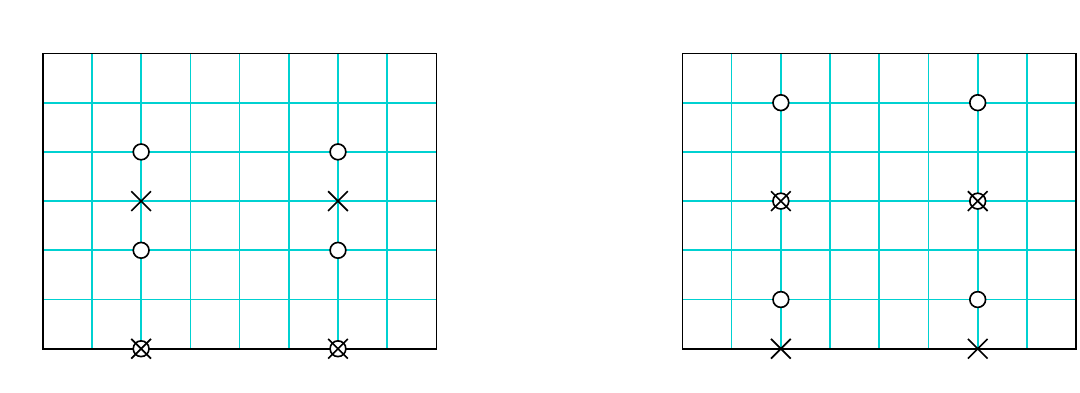}%
\end{picture}%
\setlength{\unitlength}{4144sp}%
\begingroup\makeatletter\ifx\SetFigFont\undefined%
\gdef\SetFigFont#1#2#3#4#5{%
  \reset@font\fontsize{#1}{#2pt}%
  \fontfamily{#3}\fontseries{#4}\fontshape{#5}%
  \selectfont}%
\fi\endgroup%
\begin{picture}(4935,1786)(-194,-710)
\put(2746,929){\makebox(0,0)[lb]{\smash{{\SetFigFont{10}{12.0}{\rmdefault}{\mddefault}{\updefault}{\color[rgb]{0,0,0}$2iK_1'$}%
}}}}
\put(2791,-646){\makebox(0,0)[lb]{\smash{{\SetFigFont{10}{12.0}{\rmdefault}{\mddefault}{\updefault}{\color[rgb]{0,0,0}$0$}%
}}}}
\put(4726,-646){\makebox(0,0)[lb]{\smash{{\SetFigFont{10}{12.0}{\rmdefault}{\mddefault}{\updefault}{\color[rgb]{0,0,0}$4K_1$}%
}}}}
\put(-179,929){\makebox(0,0)[lb]{\smash{{\SetFigFont{10}{12.0}{\rmdefault}{\mddefault}{\updefault}{\color[rgb]{0,0,0}$2iK_1'$}%
}}}}
\put(-134,-646){\makebox(0,0)[lb]{\smash{{\SetFigFont{10}{12.0}{\rmdefault}{\mddefault}{\updefault}{\color[rgb]{0,0,0}$0$}%
}}}}
\put(1801,-646){\makebox(0,0)[lb]{\smash{{\SetFigFont{10}{12.0}{\rmdefault}{\mddefault}{\updefault}{\color[rgb]{0,0,0}$4K_1$}%
}}}}
\end{picture}%
\caption{The sets $C_1$ and $C_1 + t_1$ (marked with $\circ$) for $\Im t_1 = \frac{iK_1'}3$ and $F = \sn$. The sets contains two branch points of $\sn$ (marked with $\times$) each.}
\label{fig:C1sn}
\end{figure}

For $F = \cn$ the situation is similar. The sets $C_1$ and $C_1 + t_1$ are shown on Figure \ref{fig:C1cn}.

\begin{figure}
\begin{picture}(0,0)%
\includegraphics{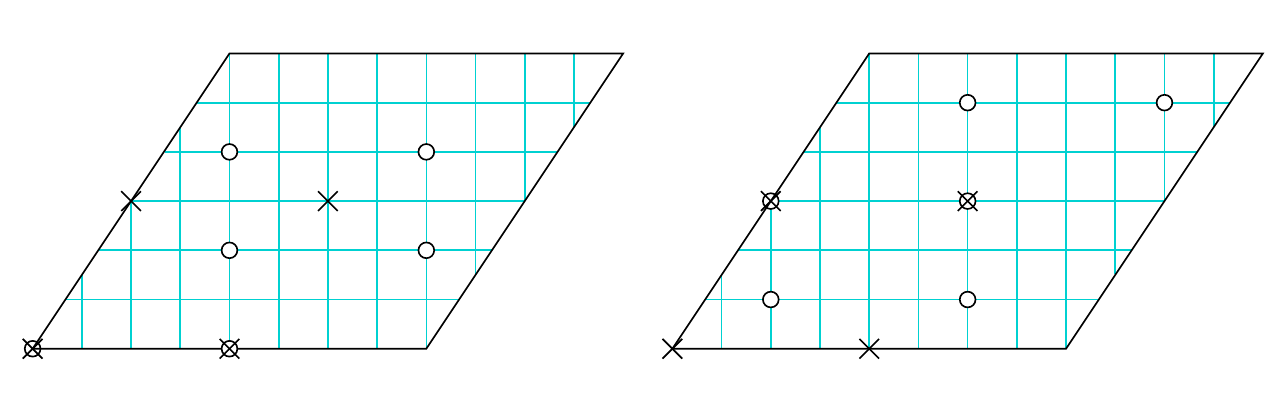}%
\end{picture}%
\setlength{\unitlength}{4144sp}%
\begingroup\makeatletter\ifx\SetFigFont\undefined%
\gdef\SetFigFont#1#2#3#4#5{%
  \reset@font\fontsize{#1}{#2pt}%
  \fontfamily{#3}\fontseries{#4}\fontshape{#5}%
  \selectfont}%
\fi\endgroup%
\begin{picture}(5787,1786)(-149,-710)
\put(2791,-646){\makebox(0,0)[lb]{\smash{{\SetFigFont{10}{12.0}{\rmdefault}{\mddefault}{\updefault}{\color[rgb]{0,0,0}$0$}%
}}}}
\put(4726,-646){\makebox(0,0)[lb]{\smash{{\SetFigFont{10}{12.0}{\rmdefault}{\mddefault}{\updefault}{\color[rgb]{0,0,0}$4K_1$}%
}}}}
\put(-134,-646){\makebox(0,0)[lb]{\smash{{\SetFigFont{10}{12.0}{\rmdefault}{\mddefault}{\updefault}{\color[rgb]{0,0,0}$0$}%
}}}}
\put(1801,-646){\makebox(0,0)[lb]{\smash{{\SetFigFont{10}{12.0}{\rmdefault}{\mddefault}{\updefault}{\color[rgb]{0,0,0}$4K_1$}%
}}}}
\put(316,929){\makebox(0,0)[lb]{\smash{{\SetFigFont{10}{12.0}{\rmdefault}{\mddefault}{\updefault}{\color[rgb]{0,0,0}$2K' + 2iK_1'$}%
}}}}
\put(3241,929){\makebox(0,0)[lb]{\smash{{\SetFigFont{10}{12.0}{\rmdefault}{\mddefault}{\updefault}{\color[rgb]{0,0,0}$2K' + 2iK_1'$}%
}}}}
\end{picture}%
\caption{The sets $C_1$ and $C_1 + t_1$ for $\Im t_1 = \frac{iK_1'}3$ and $F = \cn$. The sets contains two branch points of $\cn$ each.}
\label{fig:C1cn}
\end{figure}
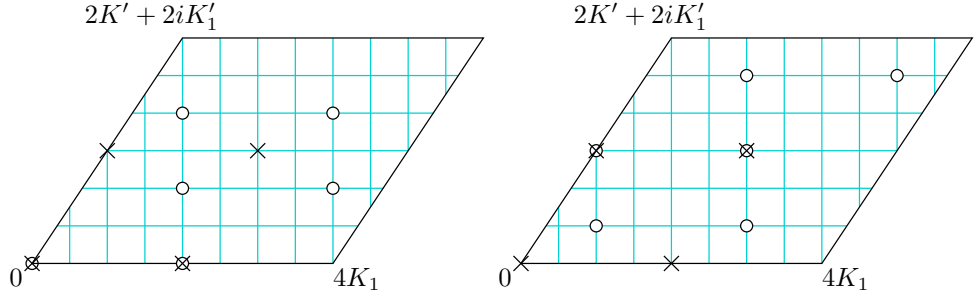

Finally, consider the case $n = 2$. We have $\Im t_1 = \frac{K'_1}2$, therefore by Lemma \ref{lem:QuartPeriod} the space $Z_1$ is parametrized by $F_1 = \sn$. If $x \in \frac12 \Lambda_2$, then we have either $\Im x \equiv 0 (\mod K'_1)$ or $\Im (x + t_1) \equiv 0 (\mod K'_1)$. Without loss of generality, assume the former to be the case. Then we have either
$$
C_1 = x + \Lambda_2 = \{0, 2K_1, iK'_1, 2K_1 + iK'_1\},
$$
in which case $A_2 = \{0,\infty\}$, which is a contradiction, or
$$
C_1 = x + \Lambda_2 = \{K_1, 3K_1, K_1 + iK'_1, 3K_1 + iK'_1\},
$$
in which case $A_2 = A_1$ and the coupling $(Q_1,Q_2)$ is reducible, so that $\widetilde{f_2}$ is not two-fold. This contradiction shows that $x \notin \frac12 \Lambda_2$, so that $|C_1| = 8$ and hence $|A_2| = |B_4| = 4$, that is $Q_2$ and $Q_4$ are elliptic.

This finishes the proof of the Lemma.
\end{proof}


\begin{figure}[ht]
\centering
$$
\xymatrix@=1pc{
& Z_{23} \ar[dl] \ar[dr]^{\widetilde{g_3}} \ar @{<->} [rr]^{\cong} && Z_{12} \ar[dl]_{\widetilde{f_1}} \ar[dr]^{\widetilde{f_2}} \ar @{<->} [rr]^{\cong} && Z_{41} \ar[dl]_{\widetilde{g_4}} \ar[dr] &\\
Z_3 \ar[dr] && Z_2 \ar[dl]^{g_2} \ar[dr]_{f_2} && Z_1 \ar[dl]^{f_1} \ar[dr]_{g_1} && Z_4 \ar[dl]\\
& \CP^1 && \CP^1 && \CP^1 &\\
}
$$
\caption{The diagram for Lemma \ref{lem:TwoEquiv}.}
\label{fig:TwoEquiv}
\end{figure}
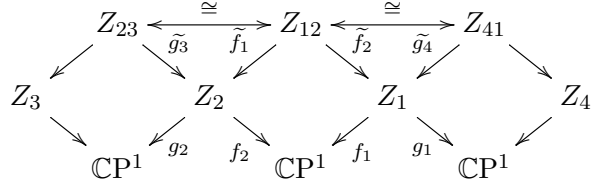

\begin{lem}
\label{lem:TwoEquiv}
Assume that the quadrilaterals $Q_1$ and $Q_2$ are elliptic, and that $\widetilde{f_2}$ and $\widetilde{g_4}$ are equivalent two-fold branched covers over $Z_1$, and $\widetilde{f_1}$ and $\widetilde{g_3}$ are equivalent two-fold branched covers over $Z_2$, see Figure \ref{fig:TwoEquiv}. Then the quadrilaterals $Q_1$ and $Q_2$ are orthodiagonal, and their involution factors at the common vertex (see Definition \ref{dfn:InvFactors}) are either equal or opposite: $\lambda_1 = \pm \lambda_2$.
\end{lem}
\begin{proof}
By Lemma \ref{lem:DoubleEll}, the imaginary part of the phase shift $t_1$ of the configuration space $Z_1$ equals either $\frac{K_1'}3$ or $\frac{K_1'}2$. In the former case, equation \eqref{eqn:C1} and Figures \ref{fig:C1sn} and \ref{fig:C1cn} provide us with the following information about the branch set $A_2$.

If $\Im t_1 = \frac{K'_1}3$ and $F_1 = \sn(\cdot, k_1)$, then one of the following holds:
\begin{subequations}
\label{eqn:A2sn}
\begin{equation}
\label{eqn:A2sn1}
A_2 = \{\pm p_1, \pm p_1 \sn(K_1 + \frac{2iK'_1}3, k_1)\}
\end{equation}
\begin{equation}
\label{eqn:A2sn2}
A_2 = \{\pm \frac{p_1}{k_1}, \pm p_1 \sn(K_1 + \frac{iK'_1}{3}, k_1)\}
\end{equation}
\end{subequations}

If $\Im t_1 = \frac{K'_1}3$ and $F_1 = \cn(\cdot, k_1)$, then one of the following holds:
\begin{subequations}
\label{eqn:A2cn}
\begin{equation}
\label{eqn:A2cn1}
A_2 = \{\pm p_1, \pm p_1 \cn(\frac{2iK'_1}3, k_1)\}
\end{equation}
\begin{equation}
\label{eqn:A2cn2}
A_2 = \{\pm ip_1\frac{k'_1}{k_1}, \pm p_1 \cn(K_1 + \frac{iK'_1}{3}, k_1)\}
\end{equation}
\end{subequations}

If $\Im t_1 = \frac{K'_1}2$, then by Lemma \ref{lem:QuartPeriod} $F_1 = \sn(\cdot,k_1)$. Hence, equations \eqref{eqn:C1} and \eqref{eqn:CL} imply
\begin{equation}
\label{eqn:A2snFor2}
A_2 = \{\pm p_1 \sn(x,k_1), \pm \frac{p_1}{k_1 \sn(x,k_1)}\}
\end{equation}

On the other hand, by Lemma \ref{lem:BranchInvol} the set $A_2$ consists of four real or four purely imaginary points if $F_2 = \sn(\cdot,k_2)$, and from two real and two purely imaginary points if $F_2 = \cn(\cdot,k_2)$. Since
$$
\sn(K + \frac{2iK'}3),\, \sn(K + \frac{iK'}{3}),\, \cn(\frac{2iK'}3) \in \R, \quad \cn(K + \frac{iK'}{3}) \in i\R,
$$
each of the quadruples in \eqref{eqn:A2sn} and \eqref{eqn:A2cn} consists either of four real or four imaginary points. In \eqref{eqn:A2snFor2} also, we cannot have two real and two purely imaginary points, but only four of the same kind. It follows that $F_2 = \sn(\cdot,k_2)$.

By reversing the roles of $Q_1$ and $Q_2$ in the above argument we see that $F_1 = \sn(\cdot,k_1)$, so that the case \eqref{eqn:A2cn} falls out of consideration. Moreover, by Lemma \ref{lem:BranchInvol} we have
\begin{equation}
\label{eqn:A2}
A_2 = \{\pm p_2, \pm \frac{p_2}{k_2}\}
\end{equation}

%

\noindent\textit{Claim 1.} If $\Im t_1 = \frac{K'_1}3$, then $k_1 < k_2$ and either $p_1 = p_2$ or $\frac{p_1}{k_1} = \frac{p_2}{k_2}$.


Indeed, since $0 < k_1 < 1$, we have $|p_2| < \left|\frac{p_2}{k_2}\right|$. On the other hand,
due to the monotonicity of $\sn$ on the segment $[K_1, K_1 + iK'_1]$ we have
$$
1 < \sn(K_1 + \frac{2iK'_1}3, k_1) < \frac{1}{k_1}
$$
Thus, if \eqref{eqn:A2sn1} occurs, then the first entry in \eqref{eqn:A2} equals the first entry in \eqref{eqn:A2sn1}, and the second one equals the second one:
$$
p_2 = p_1, \quad \frac{1}{k_2} = \sn(K_1 + \frac{2iK'_1}3, k_1) < \frac{1}{k_1},
$$
which proves the claim in the situation of \eqref{eqn:A2sn1}. Similarly, if \eqref{eqn:A2sn2} occurs, then we have
$$
\frac{p_1}{k_1} = \frac{p_2}{k_2}, \quad p_1 \sn(K_1 + \frac{iK'_1}{3}, k_1) = p_2,
$$
which also implies $k_1 < k_2$.

\smallskip

\noindent\textit{Claim 2.} If $\Im t_1 = \frac{K'_1}2$, then $\frac{p_1^2}{k_1} = \pm \frac{p_2^2}{k_2}$.

\smallskip

This follows from taking the products of entries in \eqref{eqn:A2snFor2} and \eqref{eqn:A2}.

\smallskip

By reversing the roles of $Q_1$ and $Q_2$ we obtain analogs of Claims 1 and 2 with $k_1$ and $k_2$ as well as $p_1$ and $p_2$ interchanged.

It follows that we cannot have $\Im t_1 = \frac{K'_1}3$ and $\Im t_2 = \frac{K'_2}3$, because by Claim 1 this would imply $k_1 < k_2$ and $k_2 < k_1$. We cannot have $\Im t_1 = \frac{K'_1}3$ and $\Im t_2 = \frac{K'_2}2$ (or vice versa), because then $\frac{p_1^2}{k_1} = \pm\frac{p_2^2}{k_2}$ together with $p_1 = p_2$ or $\frac{p_1}{k_1} = \frac{p_2}{k_2}$ implies $k_1 = k_2$, which contradicts $k_1 < k_2$.

Hence we have $\Im t_1 = \frac{K'_1}2$ and $\Im t_2 = \frac{K'_2}2$.
By Lemma \ref{lem:QuartPeriod} both quadrilaterals $Q_1$ and $Q_2$ are then orthodiagonal. Their involution factors are equal or opposite by Claim 2 and Lemma \ref{lem:InvFactAmpl}.
\end{proof}

\begin{lem}
Let $(Q_1, Q_2, Q_3, Q_4)$ represent a flexible Kokotsakis polyhedron without involutive couplings. Assume that $Q_1$ is elliptic or conic, and that the coupling $(Q_1,Q_2)$ is irreducible. Then the polyhedron belongs to the orthodiagonal antiinvolutive type described in Section \ref{sec:OrthoAnti}.
\end{lem}
\begin{proof}
If $(Q_1,Q_2)$ is irreducible while $Q_1$ is conic or elliptic, then the map $\widetilde{f_2} \colon Z_{12} \to Z_1$ is two-fold. Since there are no involutive coupling components, the coverings $\widetilde{f_2}$ and $\widetilde{g_4}$ are equivalent, see the beginning of Section \ref{sec:WOInvol}. Thus we are in the situation of Lemma \ref{lem:DoubleEll}.

Lemma \ref{lem:DoubleEll} implies that $Q_1$, $Q_2$, and $Q_4$ are elliptic, and their configuration spaces are all doubly covered. By non-involutivity, the double covers over $Z_2$ and $Z_4$ are equivalent, which implies that $Q_3$ is also elliptic. It follows that we can apply Lemma \ref{lem:TwoEquiv} to any pair of adjacent quadrilaterals.

Thus all $Q_i$ are elliptic orthodiagonal quadrilaterals, and the involution factors at their common vertices are equal or opposite. If two involution factors are equal, then the corresponding coupling is involutive, which contradicts our assumption. Thus we have
$$
\lambda_1 = -\lambda_2, \quad \mu_2 = -\mu_3, \quad \lambda_3 = -\lambda_4, \quad \mu_4 = -\mu_1
$$

By Corollary \ref{cor:ConfOrthInv}, the configuration space of the polyhedron is the solution set of the system
$$
\begin{aligned}
(z + \lambda_1 z^{-1})(w_1 + \mu_1 w_1^{-1}) &= \nu_1\\
(z - \lambda_1 z^{-1})(w_2 - \mu_3 w_2^{-1}) &= \nu_2\\
(u + \lambda_3 u^{-1})(w_2 + \mu_3 w_2^{-1}) &= \nu_3\\
(u - \lambda_3 u^{-1})(w_1 - \mu_1 w_1^{-1}) &= \nu_4
\end{aligned}
$$
By computing the resultants or by making trigonometric substitutions, we obtain the necessary and sufficient conditions for this system to have a one-parameter set of solutions:
$$
\frac{\nu_1^2}{\lambda_1\mu_1} = \frac{\nu_3^2}{\lambda_3\mu_3}, \quad \frac{\nu_2^2}{\lambda_1\mu_3} = \frac{\nu_4^2}{\lambda_3\mu_1}, \quad \frac{\nu_1^2}{\lambda_1\mu_1} + \frac{\nu_2^2}{\lambda_1\mu_3} = 1
$$
Thus the polyhedron has the form described in Section \ref{sec:OrthoAnti}.
\end{proof}

\begin{figure}[ht]
\centering
$$
\xymatrix@=1pc{
&& \CP^1 &&\\
& Z_3^0 \ar[ur]^{f_3} \ar[dl]_{g_3} && Z_4^0 \ar[ul]_{f_4} \ar[dr]^{g_4} &\\
\CP^1 && Z_{all}^0 \ar[ul]_{h_3} \ar[dl]^{h_2} \ar[ur]^{h_4} \ar[dr]_{h_1} && \CP^1\\
& Z_2^0 \ar[ul]^{g_2} \ar[dr]_{f_2} && Z_1^0 \ar[ur]_{g_1} \ar[dl]^{f_1} &\\
&& \CP^1 &&
}
$$
\caption{A diagram of branched covers associated with a deformation without involutive couplings.}
\label{fig:SDiag1}
\end{figure}
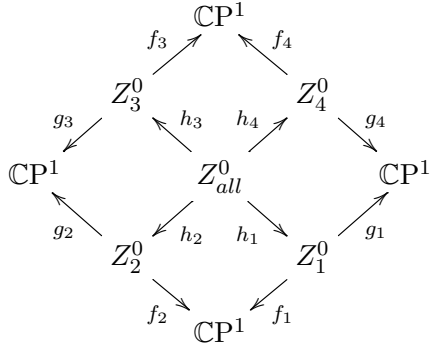

\subsubsection{A double cover over an (anti)deltoid}
\label{sec:HTwoDelt}
Without loss of generality, the map $h_2$ is two-fold, and the quadrilateral $Q_2$ is an (anti)deltoid. Besides, we may assume that $Q_2$ is coupled to $Q_1$ frontally, so that the map $f_2$ is two-fold and the map $g_2$ an isomorphism. Completing fiber product squares according to Figure \ref{fig:PullMult} and using the assumption that no conic or elliptic is doubly covered, we determine multiplicities of all but two of the other maps, see Figure \ref{fig:SDiag2}, left and middle.

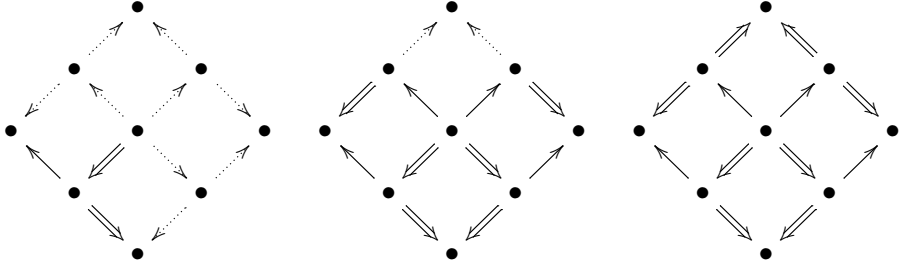
\begin{figure}[ht]
\centering
$$
{\xymatrix@=1pc{
&& \bullet &&\\
& \bullet \ar@.[ur] \ar@.[dl] && \bullet \ar@.[ul] \ar@.[dr] &\\
\bullet && \bullet \ar@.[ul] \ar@2[dl] \ar@.[ur] \ar@.[dr] && \bullet\\
& \bullet \ar[ul] \ar@2[dr] && \bullet \ar@.[ur] \ar@.[dl] &\\
&& \bullet &&
}}
\quad
{\xymatrix@=1pc{
&& \bullet &&\\
& \bullet \ar@.[ur] \ar@2[dl] && \bullet \ar@.[ul] \ar@2[dr] &\\
\bullet && \bullet \ar[ul] \ar@2[dl] \ar[ur] \ar@2[dr] && \bullet\\
& \bullet \ar[ul] \ar@2[dr] && \bullet \ar[ur] \ar@2[dl] &\\
&& \bullet &&
}}
\quad
{\xymatrix@=1pc{
&& \bullet &&\\
& \bullet \ar@2[ur] \ar@2[dl] && \bullet \ar@2[ul] \ar@2[dr] &\\
\bullet && \bullet \ar[ul] \ar@2[dl] \ar[ur] \ar@2[dr] && \bullet\\
& \bullet \ar[ul] \ar@2[dr] && \bullet \ar[ur] \ar@2[dl] &\\
&& \bullet &&
}}
$$
\caption{Solving the diagram for Section \ref{sec:HTwoDelt}.}
\label{fig:SDiag2}
\end{figure}

Since the maps $h_1$ and $h_2$ are two-fold, the coupling $(Q_1,Q_2)$ is irreducible, that is the branch sets of $f_1$ and $f_2$ are different: $A_1 \ne A_2$. Then the branch set $f_2^{-1}(A_1)$ of $h_2$ consists of four points. Since the covering $g_3$ is equivalent to $h_2$, it branches also over four points. It follows that the quadrilateral $Q_3$ is elliptic. Similarly, $Q_4$ is also elliptic, so that $f_3$ and $f_4$ are two-fold, see Figure \ref{fig:SDiag2}. Note that $(Q_3,Q_4)$ is a reducible coupling.

The spaces $Z_1^0$ and $Z_2^0$ have the equations
$$
w_1 + \mu_1 w_1^{-1} = \zeta_1 z^{m_1}, \quad w_2 + \mu_2 w_2^{-1} = \zeta_2 z^{m_2}
$$
If $m_1 = m_2$, then the space $W_z$ is the solution set of the equation
$$
w_1 + \mu_1 w_1^{-1} = \frac{\zeta_1}{\zeta_2}(w_2 + \mu_2 w_2^{-1}),
$$
which must be irreducible, because the coupling $(Q_1,Q_2)$ is irreducible by assumption. Thus we have $W_z \ne W^0_u$, because the latter is described by an irreducible equation from part \ref{it:WEll} of Lemma \ref{lem:WRed}. Thus we have $m_1 = -m_2$, that is one of $Q_1$, $Q_2$ is a deltoid, and the other an antideltoid.

For $m_1 = -m_2$ the equation of $W_z$ has the form
\begin{equation}
\label{eqn:Wz}
(w_1 + \mu_1 w_1^{-1})(w_2 + \mu_2 w_2^{-1}) = \zeta_1 \zeta_2
\end{equation}
At the same time, by Lemma \ref{lem:RedCoupEll} $W^0_u$ has the parametrization
$$
w_1 = q_1 F(t), \quad w_2 = q'_2 F(t + t_{12}),
$$
where $t_{12} = t_1 \pm t_2$ or $t_{12} = t_1 \pm it_2$, and $q'_2 = q_2$, respectively $\frac{p_1q_2}{p_2}$.
If this curve satisfies an equation of the form \eqref{eqn:Wz}, then the involution $(w_1, w_2) \mapsto (w'_1, w_2)$ descends to $\CP^1$, and thus $t_{12}$ is a quarter-period of $F$, by an argument from Section \ref{sec:Orthodiag}. If $F = \cn$, then $\mu_1$ and $\mu_2$ in \eqref{eqn:Wz} must be purely imaginary, which is not the case. Thus we have $F = \sn$ and $t_{12} = t_1 \pm t_2$. By changing the sign of $t_{12}$, if needed, we can assume
$$
t_{12} = \pm t_1 \pm t_2 \in K\Z + i\frac{K'}2
$$
It follows that $W^0_u$ has equation of the form $(z + \lambda z^{-1})(w + \mu w^{-1}) = \nu$ with coefficients given by Lemma \ref{lem:InvFactAmpl}, where $(q_1,q_2)$ must be substituted for $(p,q)$. By equating the coefficients to those in \eqref{eqn:Wz}, we obtain conditions described in Section \ref{sec:DeltAntideltEll}.

\subsubsection{A double cover over an (anti)isogram}
\label{sec:HTwoIso}
Without loss of generality, the map $h_2$ is two-fold, and the quadrilateral $Q_2$ is an (anti)isogram.

This yields the multiplicities on Figure \ref{fig:SDiag3}, and we make a case distinction according to the multiplicities of the two question mark maps.

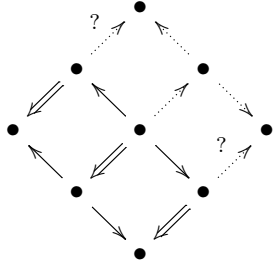
\begin{figure}[ht]
\centering
$$
\xymatrix@=1pc{
&& \bullet &&\\
& \bullet \ar@.[ur]^{?} \ar@2[dl] && \bullet \ar@.[ul] \ar@.[dr] &\\
\bullet && \bullet \ar[ul] \ar@2[dl] \ar@.[ur] \ar[dr] && \bullet\\
& \bullet \ar[ul] \ar[dr] && \bullet \ar@.[ur]^{?} \ar@2[dl] &\\
&& \bullet &&
}
$$
\caption{The diagram in the Case 2.}
\label{fig:SDiag3}
\end{figure}

\smallskip

\noindent\emph{Case 1.} Both question mark maps are isomorphisms.

This leads to the diagram on Figure \ref{fig:SDiag35}, left. Thus $Q_1$ and $Q_3$ are antideltoids, and $Q_2$ and $Q_4$ are (anti)\-isograms. By switching, if needed, the lower and/or the left boundary strips, we can transform $Q_1$ and $Q_3$ to deltoids, so that their configuration spaces have equations
\begin{equation}
\label{eqn:LinDelt}
\begin{aligned}
w_1 + \mu_1 w_1^{-1} &= \zeta_1 z^{-1}\\
u + \lambda_3 u^{-1} &= \xi_3 w_2^{-1}
\end{aligned}
\end{equation}
Configuration spaces of (anti)isograms correspond to linear substitutions:
$$
u^{\pm 1} = \kappa_4w_1, \quad z^{\pm 1} = \kappa_2w_2
$$
In order for these substitutions to transform one of the equations \eqref{eqn:LinDelt} to the other, $Q_2$ must be an antiisogram: $z = \kappa_2 w_2$. But $Q_4$ may be either isogram or antiisogram (and its type can be changed by switching the upper boundary strip). This leads us to linearly conjugate antideltoids described in Section \ref{sec:LinDelt}.

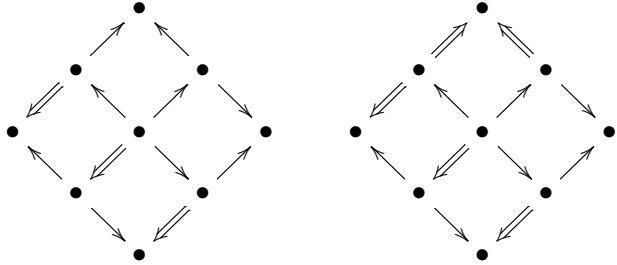
\begin{figure}[ht]
\centering
$$
{\xymatrix@=1pc{
&& \bullet &&\\
& \bullet \ar[ur] \ar@2[dl] && \bullet \ar[ul] \ar[dr] &\\
\bullet && \bullet \ar[ul] \ar@2[dl] \ar[ur] \ar[dr] && \bullet\\
& \bullet \ar[ul] \ar[dr] && \bullet \ar[ur] \ar@2[dl] &\\
&& \bullet &&
}}
\qquad
{\xymatrix@=1pc{
&& \bullet &&\\
& \bullet \ar@2[ur] \ar@2[dl] && \bullet \ar@2[ul] \ar[dr] &\\
\bullet && \bullet \ar[ul] \ar@2[dl] \ar[ur] \ar[dr] && \bullet\\
& \bullet \ar[ul] \ar[dr] && \bullet \ar[ur] \ar@2[dl] &\\
&& \bullet &&
}}
$$
\caption{Cases 1 and 2 of Section \ref{sec:HTwoIso}.}
\label{fig:SDiag35}
\end{figure}

\smallskip

\noindent\emph{Case 2.} One question mark map is two-fold, the other an isomorphism.

Without loss of generality, we obtain the diagram on Figure \ref{fig:SDiag35}, right.

Then $(Q_3,Q_4)$ is a reducible coupling, and $Q_4$ is an (anti)deltoid. It follows that $Q_3$ is conic.

There are two types of reducible couplings between a conic quadrilateral and an (anti)deltoid, see Lemma \ref{lem:RedCoupDelt}. For the second type, equation of an irreducible component $Z_{34}^0$ has the form \eqref{eqn:WConDeltSkew}, which differs from the equation of $Z_{12}^0$ obtained by a linear substitution in the equation of an (anti)deltoid. Thus $Q_3$ and $Q_4$ satisfy the conditions described in part \ref{it:ConDeltPlain} of Lemma \ref{lem:RedCoupDelt}: $m_3 = m_4$, $p_3 = p_4$.

By Lemma \ref{lem:WRed}, the two irreducible components of $Z_{34}$ are described by equations
$$
w_1 + \mu_4 e^{\pm 2it_3} w_1^{-1} = \frac{2\epsilon_4 i\sqrt{-\mu_4} e^{\pm it_3}}{q_2} w_2^{n_3}
$$
At the same time, the non-trivial irreducible component of $Z_{12}$ is given by a linear substitution $z = cw_2^{\pm 1}$ in
$$
w_1 + \mu_1 w_1^{-1} = \zeta_1 z^{m_1}
$$
By switching, if needed, the lower and/or the upper boundary strips, we can transform $Q_1$ and $Q_4$ to antideltoids: $m_1 = -1$, $m_3 = m_4 = -1$. By switching the left boundary strip, we achieve $n_3 = -1$, so that $Q_3$ has perimeter $2\pi$. Then we must have $z = cw_2$, that is $Q_2$ is an antiisogram.

By performing the substitution, we obtain the conditions on the coefficients described in Section \ref{sec:2b}.

\smallskip

\noindent\emph{Case 3.} Both question mark maps are two-fold.

This leaves two possibilities shown on Figure \ref{fig:SDiag4}.

\smallskip

\noindent\emph{Case 3a).} Figure \ref{fig:SDiag4}, left.

Quadrilaterals $Q_1$ and $Q_3$ are conic or elliptic, $Q_2$ and $Q_4$ are (anti)\-iso\-grams. That is, irreducible components of $Z_2$ and $Z_4$ correspond linear substitutions
\begin{equation}
\label{eqn:LinSubst}
z^{\pm 1} = \kappa_2w_2, \quad u^{\pm 1} = \kappa_4w_1
\end{equation}
that should transform the equation of $Z_1$ into the equation of $Z_3$ (both $Z_1$ and $Z_3$ are irreducible curves). It follows that either both $Q_1$ and $Q_3$ are conic or both are elliptic.

If $Q_1$ and $Q_3$ are conic, then by switching boundary strips we can make them both to have perimeter $2\pi$. This means that the equations of $Z_1$ and $Z_3$ have the form \eqref{eqn:AdjZ} with $c_{00} = 0$ and all other coefficients different from $0$. Substitutions \eqref{eqn:LinSubst} preserve this property if and only if $\pm 1 = 1$ in both cases, that is if and only if $Q_2$ and $Q_4$ are antiisograms. The substitutions \eqref{eqn:LinSubst} must establish bijections between the branch sets of $f_1$ and $g_3$, and respectively those of $g_1$ and $f_3$. This implies that the following conditions are necessary:
$$
q_3 = \pm \kappa_2 p_1, \quad q_1 = \pm \kappa_4 p_3
$$
It can be seen that they are also sufficient, provided that the shifts $t_1$ and $t_3$ are accordingly related, see Section \ref{sec:LinConjCon}.

If $Q_1$ and $Q_3$ are elliptic, then we use switching to transform $Q_2$ and $Q_4$ into antiisograms, so that both exponents in \eqref{eqn:LinSubst} equal $1$. Then we have
$$
\begin{aligned}
W_z &= \{w_1 = q_1 F_1(t),\, w_2 = \frac{p_1}{\kappa_2} F_1(t+t_1)\}\\
W_u &= \{w_1 = \frac{p_3}{\kappa_4} F_3(t),\, w_2 = q_3 F_3(t+t_3)\}
\end{aligned}
$$
It follows that $F_1$ and $F_3$ is the same elliptic function with the same modulus. If $\kappa_2, \kappa_4 > 0$, then we have $W_u = W_z$ if and only if the amplitudes and the shifts in the two parametrizations coincide. (Recall that the amplitudes belong by definition to $\R_+ \cup i\R_+$.) If $\kappa_2 < 0$, then we have
$$
W_z = \{w_1 = q_1 F_1(t),\, w_2 = -\frac{p_1}{\kappa_2} F_1(t+ t_1 + 2K)\}
$$
Dealing with the other combinations of signs of $\kappa_2$ and $\kappa_4$ similarly, we obtain conditions described in Section \ref{sec:LinConjEll}.

\begin{figure}[ht]
\centering
$$
{\xymatrix@=1pc{
&& \bullet &&\\
& \bullet \ar@2[ur] \ar@2[dl] && \bullet \ar[ul] \ar[dr] &\\
\bullet && \bullet \ar[ul] \ar@2[dl] \ar@2[ur] \ar[dr] && \bullet\\
& \bullet \ar[ul] \ar[dr] && \bullet \ar@2[ur] \ar@2[dl] &\\
&& \bullet &&
}}
\qquad
{\xymatrix@=1pc{
&& \bullet &&\\
& \bullet \ar@2[ur] \ar@2[dl] && \bullet \ar@2[ul] \ar@2[dr] &\\
\bullet && \bullet \ar[ul] \ar@2[dl] \ar[ur] \ar[dr] && \bullet\\
& \bullet \ar[ul] \ar[dr] && \bullet \ar@2[ur] \ar@2[dl] &\\
&& \bullet &&
}}
$$
\caption{Subcases of Case 3, Section \ref{sec:HTwoIso}.}
\label{fig:SDiag4}
\end{figure}
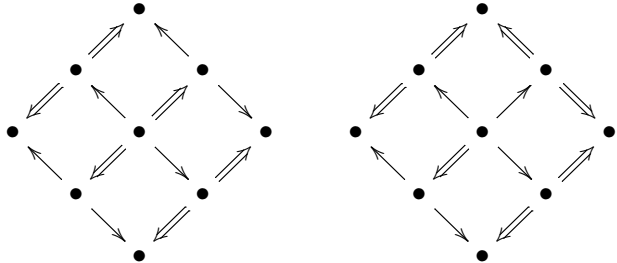

\smallskip

\noindent\emph{Case 3b.} Figure \ref{fig:SDiag4}, right.

Quadrilaterals $Q_1$, $Q_3$, $Q_4$ are conic or elliptic, $Q_2$ is an (anti)isogram.
Both couplings $(Q_3,Q_4)$ and $(Q_4,Q_1)$ are reducible. Therefore either all of these quadrilaterals are conic or all elliptic.

Let us show that in the conic case each of the reducible couplings is of the first of the two types described in Lemma \ref{lem:RedCoupCon}. Assume the converse; then without loss of generality $(Q_3,Q_4)$ is of the second type. Then the space $W^0_u$ is the solution set of an irreducible equation of the form \eqref{eqn:WConSkew}. The equation of $W_z$ has a different form, since it is obtained by a linear substitution $z^{\pm 1} = \kappa_2 w_2$ in the equation of $Z_1$. Hence $W_z \ne W^0_u$. It follows that the couplings $(Q_3,Q_4)$ and $(Q_4,Q_1)$ satisfy $m_3 = m_4$, $n_4 = n_1$. By switching boundary strips, we can achieve
$$
m_i = n_i = -1, \quad \text{for }i = 3, 4, 1
$$
It is easy to see that $Q_2$ must be an antiisogram (cf. Case 1 of this Section).
Thus we have
$$
\begin{aligned}
W_z &= \{w_1^{-1} = q_1 \sin t,\, w_2^{-1} = \kappa_2p_1 \sin(t + t_1)\}\\
W^0_u &= \{w_1^{-1} = q_4 \sin t,\, w_2^{-1} = q_3 \sin(t + t_3 \pm t_4)\}
\end{aligned}
$$
If $\kappa_2 > 0$, then the amplitudes must be equal, and the shifts are either equal or opposite. If $\kappa_2 < 0$, then the shift $t_1$ must be changed by $\pi$. We obtain a polyhedron from Section \ref{sec:2ciiCon}.

If $Q_3$, $Q_4$, and $Q_1$ are elliptic, then the reducibility of the couplings $(Q_3, Q_4)$ and $(Q_4, Q_1)$ implies that either all three quadrilaterals are of $\sn$-type, have equal moduli and equal amplitudes at common vertices, or all three are of $\cn$-type with equal or conjugate moduli, see Section \ref{sec:RedEll}. Let us show that the moduli cannot be conjugate. Assume the converse. Then, using the parametrizations from Lemma \ref{lem:RedCoupEll}, one can show that some linear combination of $t_1$, $t_3$, $t_4$ with coefficients $\pm 1$, $\pm i$, involving at least one $\pm 1$ and at least one $\pm i$, must be a real half-period. But we cannot have, say,
$$
t_1 \pm it_3 \pm it_4 \in \{0, 2K\},
$$
since $\Im t_1 \in (0, K')$ and $\Im (it_3), \Im (it_4) \in \{0, K', 2K', 3K'\}$. In the case of one $\pm i$ coefficient, one looks at the real parts. Thus all moduli $k_i$, $i = 3,4,1$, are equal.

By switching, the quadrilateral $Q_2$ can be transformed to an antiisogram: $z = \kappa_2 w_2$. Then we have
$$
\begin{aligned}
W_z &= \{w_1 = q_1 F(t),\, w_2 = \frac{p_1}{\kappa_2} F(t + t_1)\}\\
W^0_u &= \{w_1 = q_4 F(t),\, w_2 = q_3 F(t + t_3 \pm t_4)\},
\end{aligned}
$$
where $F = \sn(\cdot,k)$ or $F = \cn(\cdot,k)$. We obtain a polyhedron from Section \ref{sec:2ciiEll}.

\subsubsection{All maps $h_i$ are isomorphisms}
\label{sec:AllIsom}
Then for every $\CP^1$ on Figure \ref{fig:SDiag1} the two covers over it have the same multiplicity. Up to rotation of the diagram, there are six possibilities shown on Figures \ref{fig:SDiag5} and \ref{fig:SDiag6}.

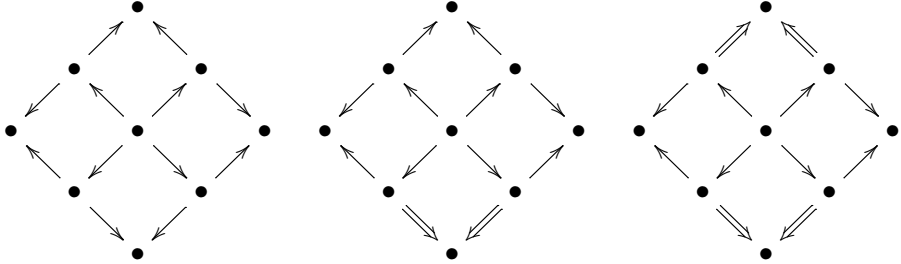
\begin{figure}[ht]
\centering
$$
{\xymatrix@=1pc{
&& \bullet &&\\
& \bullet \ar[ur] \ar[dl] && \bullet \ar[ul] \ar[dr] &\\
\bullet && \bullet \ar[ul] \ar[dl] \ar[ur] \ar[dr] && \bullet\\
& \bullet \ar[ul] \ar[dr] && \bullet \ar[ur] \ar[dl] &\\
&& \bullet &&
}}
\quad
{\xymatrix@=1pc{
&& \bullet &&\\
& \bullet \ar[ur] \ar[dl] && \bullet \ar[ul] \ar[dr] &\\
\bullet && \bullet \ar[ul] \ar[dl] \ar[ur] \ar[dr] && \bullet\\
& \bullet \ar[ul] \ar@2[dr] && \bullet \ar[ur] \ar@2[dl] &\\
&& \bullet &&
}}
\quad
{\xymatrix@=1pc{
&& \bullet &&\\
& \bullet \ar@2[ur] \ar[dl] && \bullet \ar@2[ul] \ar[dr] &\\
\bullet && \bullet \ar[ul] \ar[dl] \ar[ur] \ar[dr] && \bullet\\
& \bullet \ar[ul] \ar@2[dr] && \bullet \ar[ur] \ar@2[dl] &\\
&& \bullet &&
}}
$$
\caption{Cases 1, 2, 3 of Section \ref{sec:AllIsom}.}
\label{fig:SDiag5}
\end{figure}

\smallskip

\noindent\emph{Case 1.} There is no doubly covered $\CP^1$.

Then all $Q_i$ are (anti)isograms, and non-trivial irreducible components of their configuration spaces are described by the equations
$$
z^{\pm 1} = \kappa_1 w_1, \quad z^{\pm 1} = \kappa_2 w_2, \quad u^{\pm 1} = \kappa_3 w_2, \quad u^{\pm 1} = \kappa_4 w_1
$$
In order for this system to have a one-parameter set of solutions, the number of $+1$ among the exponents must be even. By switching a boundary strip, we change the exponents in two consecutive equations. Therefore we can transform all exponents to $+1$, and thus all quadrilaterals to antiisograms. As a result, we obtain a polyhedron described in Section~\ref{sec:Voss}.

\smallskip

\noindent\emph{Case 2.} One doubly covered $\CP^1$.

By switching, quadrilaterals $Q_3$ and $Q_4$ can be transformed to antiisograms. Then the space $W^0_u$ has an equation of the form $w_1 = cw_2$. The coupling $(Q_1, Q_2)$ is a reducible coupling of (anti)deltoids. There are two classes of such couplings, see parts 3 and 4 of Lemma \ref{lem:RedCoupDelt}. It is only in the case 3 that $W^0_z$ is described by a linear equation, that is if $Q_1$ and $Q_2$ are both deltoids or both antideltoids. Thus we obtain a compound of linear couplings from Sections \ref{sec:LinIso} and \ref{sec:LinFrontDelt}.

\smallskip

\noindent\emph{Case 3.} Two doubly covered, non-adjacent $\CP^1$.

If $Q_1$ and $Q_2$ are both deltoids or both antideltoids, then the space $W^0_z$ is described by a linear equation. Thus $Q_3$ and $Q_4$ must also be either both deltoids or both antideltoids, and we have a compound of two linear couplings from Section \ref{sec:LinFrontDelt}.

If $Q_1$ is a deltoid, and $Q_2$ an antideltoid (or vice versa, which is related by switching the lower boundary strip), then by the part \ref{it:WDeltAnti} of Lemma \ref{lem:WRed} the spaces $W^0_z$ and $W^0_u$ are described by equations
$$
\begin{aligned}
W^0_z = \{w_1 w_2 \mp \epsilon_2 \sqrt{-\mu_2} w_1 \pm \epsilon_1 \sqrt{-\mu_1} w_2 + \epsilon_1 \epsilon_2 \sqrt{-\mu_1} \sqrt{-\mu_2} = 0\}\\
W^0_u = \{w_1 w_2 \mp \epsilon_3 \sqrt{-\mu_3} w_1 \pm \epsilon_4 \sqrt{-\mu_4} w_2 + \epsilon_3 \epsilon_4 \sqrt{-\mu_3} \sqrt{-\mu_4} = 0\}
\end{aligned}
$$
It follows that we must have $\mu_2 = \mu_3$ and $\mu_1 = \mu_4$, that is the couplings $(Q_2, Q_3)$ and $(Q_4, Q_1)$ are involutive. This contradicts our assumption that the polyhedron contains no involutive couplings.

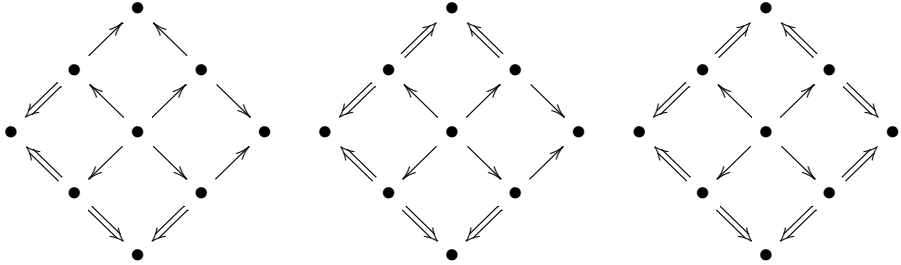
\begin{figure}[ht]
\centering
$$
{\xymatrix@=1pc{
&& \bullet &&\\
& \bullet \ar[ur] \ar@2[dl] && \bullet \ar[ul] \ar[dr] &\\
\bullet && \bullet \ar[ul] \ar[dl] \ar[ur] \ar[dr] && \bullet\\
& \bullet \ar@2[ul] \ar@2[dr] && \bullet \ar[ur] \ar@2[dl] &\\
&& \bullet &&
}}
\quad
{\xymatrix@=1pc{
&& \bullet &&\\
& \bullet \ar@2[ur] \ar@2[dl] && \bullet \ar@2[ul] \ar[dr] &\\
\bullet && \bullet \ar[ul] \ar[dl] \ar[ur] \ar[dr] && \bullet\\
& \bullet \ar@2[ul] \ar@2[dr] && \bullet \ar[ur] \ar@2[dl] &\\
&& \bullet &&
}}
\quad
{\xymatrix@=1pc{
&& \bullet &&\\
& \bullet \ar@2[ur] \ar@2[dl] && \bullet \ar@2[ul] \ar@2[dr] &\\
\bullet && \bullet \ar[ul] \ar[dl] \ar[ur] \ar[dr] && \bullet\\
& \bullet \ar@2[ul] \ar@2[dr] && \bullet \ar@2[ur] \ar@2[dl] &\\
&& \bullet &&
}}
$$
\caption{Cases 4, 5, 6 of Section \ref{sec:AllIsom}.}
\label{fig:SDiag6}
\end{figure}

\smallskip

\noindent\emph{Case 4.} Two doubly covered, adjacent $\CP^1$.

Similarly to the case 2 of Section \ref{sec:HTwoIso}, the reducible coupling between $Q_1$ and $Q_2$, as well as that between $Q_2$ and $Q_3$ must have the form described in part \ref{it:ConDeltPlain} of Lemma \ref{lem:RedCoupDelt}. By switching, we can achieve $m_1 = m_2 = -1$ and $n_2 = n_3 = -1$, so that $Q_2$ has perimeter $2\pi$, and $Q_1$ and $Q_3$ are antideltoids.

By Lemma \ref{lem:WRed} the space $W_z^0$ is described by one of the equations
\[
w_1 + \mu_1 e^{\pm 2it_2} w_1^{-1} = \epsilon_1 \frac{2i\sqrt{-\mu_1} e^{\pm it_2}}{q_2} w_2^{-1}
\]
The equation of $W_u$ is obtained by substituting in
\[
u + \lambda_3 u^{-1} = \xi_3 w_2^{-1}
\]
either $u = \kappa_4 w_1$, if $Q_4$ is an antiisogram, or $u^{-1} = \kappa_4 w_1$, if $Q_4$ is an isogram. The equation of $W_u$ has then the same form as equation of $W_z^0$. Note that the equality of the coefficients at $w_2^{-1}$ implies the equality of the coefficients at $w_1^{-1}$ due to $q_2 = q_3$. We obtain flexible polyhedra from Sections \ref{sec:IConIso} and \ref{sec:IConAntiiso}. 
%
%

\smallskip

\noindent\emph{Case 5.} Three doubly covered $\CP^1$.

Then each of $(Q_1,Q_2)$ and $(Q_4,Q_3)$ is a reducible coupling of an (anti)del\-toid with a conic quadrilateral. Lemma \ref{lem:WRed} implies that we have either $m_1 = m_2$ and $m_3 = m_4$ or $m_1 = -m_2$ and $m_3 = -m_4$, and furthermore $n_2 = n_3$ and $q_2 = q_3$. Thus by switching we can achieve that $Q_2$ and $Q_3$ have perimeter $2\pi$; then $Q_1$ and $Q_4$ are either both deltoids or both antideltoids.

If $Q_1$ and $Q_4$ are antideltoids, then we have
$$
\begin{aligned}
W_z^0 &= \{w_1 = \epsilon_1 \sqrt{-\mu_1} e^{it}, \quad w_2^{-1} = q_2 \sin(t \pm t_2)\}\\
W_u^0 &= \{w_1 = \epsilon_4 \sqrt{-\mu_4} e^{it'}, \quad w_2^{-1} = q_3 \sin(t' \pm t_3)\}
\end{aligned}
$$
The two parametrizations of $w_1$ differ by a shift: $t' = t + t_0$, where
\begin{equation}
\label{eqn:MuMu}
\epsilon_1 \sqrt{-\mu_1} = \epsilon_4 \sqrt{-\mu_4} e^{it_0}
\end{equation}
By substituting $t' = t + t_0$ in the second parametrization of $w_2$, we obtain $t_0 = \pm t_2 \pm t_3$. This leads to a polyhedron described in Section \ref{sec:3e}.

If $Q_1$ and $Q_4$ are deltoids, then we have
$$
\begin{aligned}
W_z^0 &= \{w_1 = \epsilon_1 \sqrt{-\mu_1} e^{it}, \quad w_2^{-1} = q_2 \frac{\cos t_2 + i \sin t_2 \cos t}{\sin t}\}\\
W_u^0 &= \{w_1 = \epsilon_4 \sqrt{-\mu_4} e^{it}, \quad w_2^{-1} = q_3 \frac{\cos t_3 + i \sin t_3 \cos t}{\sin t}\}
\end{aligned}
$$
Again, by looking at $w_1$ we obtain $t' = t+t_0$. But comparing the zeros of $w_2$ in both parametrizations we arrive at $t_0 \in \{0, \pi\}$. Then \eqref{eqn:MuMu} implies that $\mu_1 = \mu_4$, that is the coupling $(Q_1,Q_4)$ is involutive. This contradicts our assumption. (But we obtain a flexible polyhedron which is a compound of couplings from \ref{sec:LinLatDelt} and \ref{sec:LinCon}.)

\smallskip

\noindent\emph{Case 6.} All $\CP^1$ are doubly covered.

In this case every coupling must be reducible, hence either all quadrilaterals are elliptic or all of them are conic.

\smallskip

\noindent\emph{Case 6a.} All quadrilaterals are elliptic.

By Lemma \ref{lem:BranchInvol}, in the $\sn$ case the branch points are either all real or all imaginary, while in the $\cn$ case two of them are real, and two imaginary. Hence either all quadrilaterals are of $\sn$ type or all of them are of $\cn$ type.

In the $\sn$ case by Lemma \ref{lem:RedCoupEll} we have the following parametrizations of a component $W_z^0 = W_u^0$.
$$
\begin{aligned}
W_z^0 &= \{w_1 = q_1 \sn t, \quad w_2 = q_2 \sn(t + t_1 \pm t_2)\}\\
W_u^0 &= \{w_1 = q_4 \sn t', \quad w_2 = q_3 \sn(t' + t_4 \pm t_3)\}
\end{aligned}
$$
The parameters $t$ and $t'$ are either equal or related by $t' = 2K - t$. The second substitution leads to a reparametrization of $W_u^0$ with $w_2 = q_3 \sn(t - t_4 \mp t_3)$. It follows that the spaces $W_z$ and $W_u$ have a common irreducible component if and only if $t_1 \pm t_2 \pm t_3 \pm t_4$ is a period of $\sn$ for some of the eight possible choices of signs. Together with the reducibility conditions $p_1 = p_2$, $q_2 = q_3$, $p_3 = p_4$, $q_4 = q_1$ this leads to the description of the equimodular type of flexible polyhedra given in Section \ref{sec:EllCase}.

Let all quadrilaterals be of $\cn$ type. A reducible coupling of elliptic quadrilaterals of $\cn$ type can have one of two forms described in Lemma \ref{lem:RedCoupEll}. If all couplings are of the first type, then the situation is similar to that in the $\sn$ case, and the polyhedron is of equimodular type.
Otherwise let $k = k_1$ be the Jacobi modulus of $Q_1$. Each of the other moduli $k_i$ equals $k$ or $k' = \sqrt{1-k^2}$. From the parametrization \eqref{eqn:ParamEllSkew} the following necessary condition for flexibility can be deduced:
$$
t_1 + \epsilon_2 t_2 + \epsilon_3 t_3 + \epsilon_4 t_4 \in \{0, 2K_1\},
$$
where $\epsilon_i = \pm 1$ if $k_i = k$, and $\epsilon_i = \pm i$ if $k_i = k'$. By looking at the real or imaginary part of this linear combination, we conclude that exactly two of the moduli $k_i$ must be conjugate to $k$. Without loss of generality, there are two possibilities:
\begin{subequations}
\begin{equation}
\label{eqn:One}
k_1 = k_3 = k, \quad k_2 = k_4 = k'
\end{equation}
\begin{equation}
\label{eqn:Two}
k_1 = k_4 = k, \quad k_2 = k_3 = k'
\end{equation}
\end{subequations}
In both cases we have
$$
\frac{p_1}{p_2} = \pm i \frac{k_1}{k_2}, \quad \frac{p_3}{p_4} = \pm i \frac{k_3}{k_4}
$$
If \eqref{eqn:One} takes place, then we have the parametrizations
$$
\begin{aligned}
W_z^0 &= \{w_1 = q_1 \cn t, \quad w_2 = \frac{p_1q_2}{p_2} \cn(t + t_1 \pm it_2)\}\\
W_u^0 &= \{w_1 = \frac{p_3q_4}{p_4} \cn t', \quad w_2 = q_3 \cn(t' + t_3 \pm it_4)\}
\end{aligned}
$$
It follows that $t' = \pm t$ or $t' = \pm t + 2K$. In the former case we have $q_1 = \frac{p_3q_4}{p_4}$, in the latter $q_1 = -\frac{p_3q_4}{p_4}$. By comparing the two representations of $w_2(t)$, we arrive at the description given in Section \ref{sec:ConjMod1}.

If \eqref{eqn:Two} takes place, then we have $q_1 = q_4$ and $q_2 = q_3$, because the corresponding couplings are reducible and quadrilaterals have equal moduli. The parametrizations are
$$
\begin{aligned}
W_z^0 &= \{w_1 = q_1 \cn t, \quad w_2 = \frac{p_1q_2}{p_2} \cn(t + t_1 \pm it_2)\}\\
W_u^0 &= \{w_1 = q_4 \cn t', \quad w_2 = \frac{p_4q_3}{p_3} \cn(t' + t_4 \pm it_3)\}
\end{aligned}
$$
It follows that either $t' = t$ or $t' = -t$, so that we finally arrive at
$$
\frac{p_1}{p_2} \cn(t + t_1 \pm it_2) = \frac{p_4}{p_3} \cn(t \pm t_4 \pm it_3)
$$
Thus either $\frac{p_1}{p_2} = \frac{p_4}{p_3}$ and $t_1 \pm it_2 \pm it_3 \pm t_4 = 0$ or $\frac{p_1}{p_2} = -\frac{p_4}{p_3}$ and $t_1 \pm it_2 \pm it_3 \pm t_4 = 2K$. This type is described in Section \ref{sec:ConjMod2}.

\smallskip

\noindent\emph{Case 6b.} All quadrilaterals are conic.

If all couplings are of the type described in part \ref{it:ConPlain} of Lemma \ref{lem:RedCoupCon}, then we can achieve by switching that $m_i = -1$ and $n_i = -1$ for all $i$. We have the following parametrizations:
$$
\begin{aligned}
W_z^0 &= \{w_1 = q_1 \sin t, \quad w_2 = q_2 \sin(t + t_1 \pm t_2)\}\\
W_u^0 &= \{w_1 = q_4 \sin t', \quad w_2 = q_3 \sin(t' + t_4 \pm t_3)\}
\end{aligned}
$$
As $q_1 = q_4$, we have either $t' = t$ or $t' = \pi - t$. By equating the two expressions for $w_2$, we obtain $t_1 \pm t_2 \pm t_3 \pm t_4 = 0 (\mod\, 2\pi)$. Thus we have a polyhedron of conic equimodular type, see Section \ref{sec:ConEquimod}.

If the coupling $(Q_1,Q_2)$ is of the second kind, then the space $W_z^0$ is described by an equation of the form given in part \ref{it:WConSkew} of Lemma \ref{lem:WRed}. By comparing this with the equation in part \ref{it:WConPlain} of the same Lemma, we conclude that the coupling $(Q_3,Q_4)$ must also be of the second kind. We claim that then $(Q_2, Q_3)$ and $(Q_4, Q_1)$ must be of the first kind. Indeed, if all couplings are of the second kind, then we can achieve by switching $n_1 = n_2 = 1$ and $n_3 = n_4 = -1$. Then the spaces $W_z^0$ and $W_u^0$ are described by equations of the form
\begin{gather*}
a_{22}w_1^2w_2^2 + a_{20} w_1^2 + a_{02} w_2^2 + 2a_{11}w_1w_2 + a_{10} w_1 + a_{01} w_2 + a_{00} = 0\\
a'_{00}w_1^2w_2^2 + a'_{01} w_1^2w_2 + a'_{10} w_1w_2^2 + \cdots = 0
\end{gather*}
respectively, where $a_{01} \ne 0$. This contradicts the assumption $W_z^0 = W_u^0$.

Thus the only possibility, up to a cyclic shift of indices, is that $(Q_1,Q_2)$ and $(Q_3,Q_4)$ are of the second kind, while $(Q_2,Q_3)$ and $(Q_4,Q_1)$ are of the first kind. By switching, we can achieve $n_1 = n_4 = -1$ and $n_2 = n_3 = -1$, so that we have the parametrizations
$$
\begin{aligned}
W_z^0 &= \{w_1^{-1} = q_1 \sin t, \quad w_2 = q_2 \frac{\cos t_2 \pm i \sin t_2 \cos(t+t_1)}{\sin(t+t_1)}\}\\
W_u^0 &= \{w_1^{-1} = q_4 \sin t', \quad w_2 = q_3 \frac{\cos t_3 \pm i \sin t_2 \cos(t'+t_4)}{\sin(t'+t_4)}\}
\end{aligned}
$$
Since $q_1 = q_4$, we have either $t' = t$ or $t' = \pi - t$. In the first case, by looking at the zeros of $w_2$, we obtain $t_1 = t_4$, so that the coupling $(Q_4,Q_1)$ is involutive that contradicts the assumption of this Section.
The parameters cannot be related by $t' = \pi - t$, since this would lead to $t_1 = -t_4$, which contradicts $\Im t_i > 0$.

%
%
%

\def\cprime{$'$}


\begin{thebibliography}{10}

\bibitem{AC11}
V.~Alexandrov and R.~Connelly.
\newblock Flexible suspensions with a hexagonal equator.
\newblock {\em Illinois J. Math.}, 55:127--155, 2011.

\bibitem{BS08}
A.~I. Bobenko and Y.~B. Suris.
\newblock {\em Discrete differential geometry}, volume~98 of {\em Graduate
  Studies in Mathematics}.
\newblock American Mathematical Society, Providence, RI, 2008.
\newblock Integrable structure.

\bibitem{Con74}
R.~Connelly.
\newblock An attack on rigidity. {I}, {II}.
\newblock {\em Bull. Amer. Math. Soc.}, 81:566--569, 1975.
\newblock (An extended version is available at: {\tt
  http://www.math.cornell.edu/\textasciitilde connelly}).

\bibitem{Dar79}
J.-G. Darboux.
\newblock De l'emploi des fonctions elliptiques dans la th\'eorie du
  quadrilat\`ere plan.
\newblock {\em Bull. Sci. Math.}, 3(1):109--128, 1879.

\bibitem{Gai13}
A.~A. Gaifullin.
\newblock Flexible cross-polytopes in spaces of constant curvature.
\newblock \url{http://arxiv.org/abs/1312.7608}.

\bibitem{Izm_Conf4Bar}
I.~Izmestiev.
\newblock Deformation of quadrilaterals and addition on elliptic curves.

\bibitem{Kar10}
O.~N. Karpenkov.
\newblock On the flexibility of {K}okotsakis meshes.
\newblock {\em Geom. Dedicata}, 147:15--28, 2010.

\bibitem{Kok33}
A.~Kokotsakis.
\newblock \"{U}ber bewegliche {P}olyeder.
\newblock {\em Math. Ann.}, 107(1):627--647, 1933.

\bibitem{Kri81}
I.~M. Krichever.
\newblock The {B}axter equations and algebraic geometry.
\newblock {\em Funktsional. Anal. i Prilozhen.}, 15(2):22--35, 96, 1981.

\bibitem{Naw10}
G.~Nawratil.
\newblock {Flexible octahedra in the projective extension of the Euclidean
  3-space.}
\newblock {\em J. Geom. Graph.}, 14(2):147--169, 2010.

\bibitem{Naw11}
G.~Nawratil.
\newblock Reducible compositions of spherical four-bar linkages with a
  spherical coupler component.
\newblock {\em Mechanism and Machine Theory}, 46:725--742, 2011.

\bibitem{Naw12}
G.~Nawratil.
\newblock Reducible compositions of spherical four-bar linkages without a
  spherical coupler component.
\newblock {\em Mechanism and Machine Theory}, 49:87--103, 2012.

\bibitem{Ni12}
Y.~Nishiyama.
\newblock Miura folding: applying origami to space exploration.
\newblock {\em Int. J. Pure Appl. Math.}, 79(2):269--279, 2012.

\bibitem{Pak09}
F.~Pakovich.
\newblock Prime and composite {L}aurent polynomials.
\newblock {\em Bull. Sci. Math.}, 133(7):693--732, 2009.

\bibitem{PW08}
H.~Pottmann and J.~Wallner.
\newblock The focal geometry of circular and conical meshes.
\newblock {\em Adv. Comput. Math.}, 29(3):249--268, 2008.

\bibitem{Ritt22}
J.~F. Ritt.
\newblock Prime and composite polynomials.
\newblock {\em Trans. Amer. Math. Soc.}, 23(1):51--66, 1922.

\bibitem{Sau70}
R.~Sauer.
\newblock {\em Differenzengeometrie}.
\newblock Springer-Verlag, Berlin-New York, 1970.

\bibitem{SG31}
R.~Sauer and H.~Graf.
\newblock \"{U}ber {F}l\"achenverbiegung in {A}nalogie zur {V}erknickung
  offener {F}acettenflache.
\newblock {\em Math. Ann.}, 105(1):499--535, 1931.

\bibitem{SBH08}
W.~K. Schief, A.~I. Bobenko, and T.~Hoffmann.
\newblock On the integrability of infinitesimal and finite deformations of
  polyhedral surfaces.
\newblock In {\em Discrete differential geometry}, volume~38 of {\em
  Oberwolfach Semin.}, pages 67--93. Birkh\"auser, Basel, 2008.

\bibitem{Sta10}
H.~Stachel.
\newblock A kinematic approach to {K}okotsakis meshes.
\newblock {\em Comput. Aided Geom. Design}, 27(6):428--437, 2010.

\end{thebibliography}
\end{document}